\documentclass[a4paper,11pt]{amsart}

\usepackage{color}
\usepackage{comment}
\usepackage{amssymb}
\usepackage{amsfonts}
\usepackage{amscd}
\usepackage{slashed}
\usepackage{pdflscape}
\usepackage{tikz}
\usepackage{enumerate}
\usepackage{multirow}
\usepackage{stmaryrd}
\usepackage{ marvosym }

\newcommand{\hh}{{\hspace{.3mm}}}

\newcommand{\pdot}{{\boldsymbol{\cdot}}}

\usepackage{bbold}
\usepackage[backref,pagebackref,linkcolor=blue]{hyperref}
\usepackage{mathrsfs}

\newcommand{\IIo}{\mathring{{\bf\rm I\hspace{-.2mm} I}}{\hspace{.2mm}}}

\newcommand{\bm}[1]{\mbox{\boldmath $ #1 $}}

\newtheorem{theorem}{Theorem}[section]
\newtheorem{lemma}[theorem]{Lemma}
  
\newtheorem{proposition}[theorem]{Proposition}
\newtheorem{corollary}[theorem]{Corollary}

\theoremstyle{definition}
\newtheorem{definition}[theorem]{Definition}

\theoremstyle{remark}
\newtheorem{remark}[theorem]{Remark}

\newtheorem{problem}[theorem]{Problem}

\usepackage{graphicx} 
\usepackage{amsmath} 
\usepackage{amsfonts}
\usepackage{amssymb}

\newcommand{\be}{\begin{equation}}

\newcommand{\ee}{\end{equation}}

\newcommand{\g}{\go }

\newcommand{\om}{\omega}

\newcommand{\ba}{\begin{array}}

\newcommand{\ea}{\end{array}}

\newcommand{\beq}{\begin{eqnarray}}

\newcommand{\eeq}{\end{eqnarray}}

\newtheorem{lm}{lemma}

\newtheorem{thee}{theorem}

\newtheorem{proo}{proposition}

\newtheorem{co}{corollary}

\newtheorem{rem}{remark}

\newtheorem{deff}{definition}

\newcommand{\bd}{\begin{deff}}

\newcommand{\ed}{\end{deff}}

\newcommand{\bl}{\begin{lm}}

\newcommand{\el}{\end{lm}}

\newcommand{\bp}{\begin{proo}}

\newcommand{\ep}{\end{proo}}

\newcommand{\bt}{\begin{thee}}

\newcommand{\et}{\end{thee}}

\newcommand{\bc}{\begin{co}}

\newcommand{\ec}{\end{co}}

\newcommand{\brm}{\begin{rem}}

\newcommand{\erm}{\end{rem}}

\usepackage{t1enc}

\def\Cal{\mathcal}
\newcommand{\cL}{{\Cal L}}

\newcommand{\newc}{\newcommand}

\renewcommand{\exp}{\operatorname{exp}}

\let\ccdot\cdot
\def\cdot{\hbox to 2.5pt{\hss$\ccdot$\hss}}

\newc{\aR}{\mbox{\boldmath{$ R$}}}
\newc{\aS}{\mbox{\boldmath{$ S$}}}
\newc{\aT}{\mbox{\boldmath{$ T$}}}
\newc{\aW}{\mbox{\boldmath{$ W$}}}

\newc{\aD}{\mbox{\boldmath{$ D$}}\hspace{-.2mm}}

\newc{\aK}{\mbox{\boldmath{$ K$}}}
\newc{\aL}{\mbox{\boldmath{$ L$}}}

\newcommand{\ce}{{\Cal E}}

\usepackage{amssymb}
\usepackage{amscd}

\newcommand{\Rho}{P}

\newcommand{\End}{\operatorname{End}}

\newcommand{\nn}[1]{(\ref{#1})}

\newcommand{\bg}{\mbox{\boldmath{$ g$}}}

\newc{\obstrn}[2]{B^{#1}_{#2}}

\newcommand{\rpl}                         
{\mbox{$
\begin{picture}(12.7,8)(-.5,-1)
\put(0,0.2){$+$}
\put(4.2,2.8){\oval(8,8)[r]}
\end{picture}$}}

\newcommand{\lpl}                         
{\mbox{$
\begin{picture}(12.7,8)(-.5,-1)
\put(2,0.2){$+$}
\put(6.2,2.8){\oval(8,8)[l]}
\end{picture}$}}

\usepackage{ifthen}

\newc{\tensor}[1]{#1}
\newc{\Mvariable}[1]{\mbox{#1}}
\newc{\down}[1]{{}_{#1}}
\newc{\up}[1]{{}^{#1}}

\newcommand{\II}{{\rm  I\hspace{-.2mm}I}}

\newc{\JulyStrut}{\rule{0mm}{6mm}}
\newc{\midtenPan}{\mbox{\sf S}}
\newc{\midten}{\mbox{\sf T}}
\newc{\midtenEi}{\mbox{\sf U}}
\newc{\ATen}{\mbox{\sf E}}
\newc{\BTen}{\mbox{\sf F}}
\newc{\CTen}{\mbox{\sf G}}

\newcommand{\bj}{\bar\jmath}

\def\sideremark#1{\ifvmode\leavevmode\fi\vadjust{\vbox to0pt{\vss
 \hbox to 0pt{\hskip\hsize\hskip1em
 \vbox{\hsize3cm\tiny\raggedright\pretolerance10000
 \noindent #1\hfill}\hss}\vbox to8pt{\vfil}\vss}}}%
                        
\numberwithin{equation}{section}

\begin{document}

\renewcommand{\today}{}
\title{Conformally compact and higher  
conformal 
Yang--Mills equations 
}
\author{A. Rod Gover${}^\heartsuit$, Emanuele Latini${}^\clubsuit$, Andrew Waldron${}^\spadesuit$ \& Yongbing Zhang${}^\diamondsuit$}

\address{${}^\heartsuit$Department of Mathematics\\
  The University of Auckland\\
  Private Bag 92019\\
  Auckland 1\\
  New Zealand} \email{gover@math.auckland.ac.nz}
  
  \address{${}^{\spadesuit}$
  	Center for Quantum Mathematics and Physics (QMAP) and\\
  Department of Mathematics\\
  University of California\\
  Davis, CA95616, USA} \email{wally@math.ucdavis.edu}
  
  \address{${}^\clubsuit$
  Dipartimento di Matematica, Universit\`a di Bologna, Piazza di Porta S. Donato 5,
 and  INFN, Sezione di Bologna, Via Irnerio 46, I-40126 Bologna,  Italy}
 \email{emanuele.latini@UniBo.it}

  \address{${}^\diamondsuit$
School of Mathematical Sciences and Wu Wen-Tsun Key Laboratory of Mathematics \\
University of Science and Technology of China\\ Hefei, Anhui, 230026, P. R. of China.}
\email{ybzhang@amss.ac.cn}

\vspace{10pt}

\vspace{10pt}

\renewcommand{\arraystretch}{1}

\begin{abstract}
On
conformally compact manifolds
we study Yang--Mills equations,  their
boundary conditions, formal asymptotics, and  Dirichlet-to-Neumann maps.
We find that smooth solutions with ``magnetic'' Dirichlet boundary data  
are obstructed by a conformally invariant, higher order boundary current. 
We study the asymptotics of the interior Yang--Mills energy functional and show that the  obstructing current is the variation of the conformally invariant coefficient of the first  log term in this expansion which is a higher derivative conformally invariant analog of the Yang--Mills energy.
The invariant energy is the anomaly for the renormalized interior  Yang--Mills functional
and its variation gives higher conformal Yang--Mills equations.
Global solutions to the magnetic boundary problem determine higher order ``electric'' Neumann data. This yields the  Dirichlet-to-Neumann map.
We also construct conformally invariant, higher transverse  derivative boundary operators. Acting on interior connections, they 
give obstructions to solving the Yang--Mills boundary problem,
determine the asymptotics of formal solutions,  and
 yield conformally invariant tensors capturing the (non-local) electric Neumann data. We also characterize a renormalized Yang--Mills action functional that encodes global features analogously to the renormalized volume for Poincar\'e--Einstein structures.


%

  \vspace{2cm}
\end{abstract}

\maketitle

\pagestyle{myheadings} \markboth{Gover, Latini, Waldron \& Zhang}{Yang-Mills on conformally compact manifolds}

\newpage

\tableofcontents

\section{Introduction}

\newcommand{\go}{{\mathring {g}}}
\newcommand{\cc}{\boldsymbol{c}}
\newcommand{\DD}{{\sf D}}
\newcommand{\ext}{{\rm d}}

\newc{\LagDen}{\cL_A}
\newc{\hatg}{\hat{\g}}
\newc{\gsup}{\mbox{\textsl{\tiny g}}}
\newc{\hatgsup}{\hat{\mbox{\textsl{\tiny g}}}}
\newc{\Action}{{S}}

Yang--Mills equations play a critical {\it r\^ole} in particle physics and the study of four-manifolds
\cite{YM,Donaldson,DonaldsonKron}. Let $\nabla^A$ be a gauge
connection with curvature $F_{ab}$  on a 
Riemannian or pseudo-Riemannian $n$-manifold $(M,g)$.  The Yang--Mills action functional is 

\newcommand{\Trace}{\operatorname{Trace}}
\begin{equation}\label{act1}
\Action[A]:=-\frac14 \, \int_M \ext {\rm Vol}(g)\Trace (g^{ab}g^{cd}F_{ac} F_{bd})\, ,
\end{equation}
where  the curvature $F$ is a two-form taking values in endomorphisms of a vector bundle ${\mathcal V}M$ over $M$; our detailed conventions for tensors {\it etcetera} are given at the end  of this section.
The corresponding 
Euler--Lagrange equations, with respect to  compactly supported variations,
are the (source-free) {\em Yang-Mills equations}
\begin{equation}\label{ymeq}
j[A]:=\delta^A F=0\, .
\end{equation}
Here $\delta^A $ is the formal adjoint of the connection-twisted
exterior derivative.

We consider two main problems. 
The first is to find higher, conformally invariant analogs of the energy functional~\nn{act1}. The second is to study solutions to Equation~\eqref{ymeq} on  conformally compact manifolds. 
We work with Riemannian
signature metrics~$g$, though many results adapt easily to other
cases. 
The action (\ref{act1}) is conformally
invariant when $n=4$, meaning that  upon replacing the metric~$g$ with a conformally related metric~$\hat{g}=e^{2 \om} g$ for some $\om \in C^\infty M$,
 the action $\Action[A]$ is left  unchanged. 
The well-known, dimension four,
conformal covariance of the
 Yang--Mills equations follows.
 The first question is
whether there are natural replacements of~(\ref{act1}) and~(\ref{ymeq}) that are conformally invariant in other dimensions. 
In fact for dimension $n=6$ a positive answer, and an
application, is provided in~\cite{GoPeSl}, but 
at this stage it is not clear how to apply those ideas to
 other dimensions. 
We solve this problem
for all even dimensions; see Theorem~\ref{blabla}.
This relies on  holographic methods applied to 
our asymptotic solution to the second main problem, namely that of solving  the Yang-Mills equations~(\ref{ymeq}) on conformally compact manifolds of any dimension $d=n+1\geq 3$; see Theorem~\ref{recurse}.
Part of the importance of these results is that they apply to any linear connection on any vector bundle, and thus generate large families of invariants. For example, 
as illustrated in six dimensions in~\cite{{GoPeSl}},
the functional gradients of the higher Yang--Mills energies 
discussed below are generalizing analogs of the 
   Fefferman--Graham~\cite{FG-ast,thebookFG} obstruction tensor, yielding such an invariant for any connection.
 
In four dimensions self or anti-self dual connections are Yang--Mills~\cite{BPST,Donaldson}. Hence there is additional  interesting structure when either the bulk or boundary dimension is four. Since we predominantly study general  behaviour for the Yang--Mills systems in higher dimensions, we do not treat this aspect.

\smallskip
  A {\it conformally compact manifold} is a compact $d$-manifold $M$ with
  boundary $\partial M =: \Sigma$ whose interior ${M}_+=M\backslash \partial M$ is equipped
  with a metric $\g$ such that $g=r^2\g$ extends smoothly as a metric
  to the boundary, where $r$ is any boundary defining function
  (meaning $\Sigma =r^{-1}(0)$ and $\ext r$ is nowhere zero on
  $\Sigma$). In this case $({M}_+,\go)$ is complete. It is said to be
  {\it asymptotically hyperbolic}  if $|\ext r|^2_{g}=1 $ at all points of
  $\Sigma$, 
  {\it singular Yamabe} if the scalar curvature of $\go$ is $-d(d-1)$ on $M_+$, 
  and {\it Poincar\'e-Einstein} if  $\go$ is Einstein
  (again with scalar curvature $-d(d-1)$). 
  When the trace-free Schouten tensor  of $\go$
  vanishes to order
 $\mathcal{O}(r^{d-3})$, it is termed {\it asymptotically Poincar\'e--Einstein}.
  These
  structures have been studied intensively since the pioneering work
  of Fefferman and Graham \cite{FG-ast}; important aspects include
quantum mechanical
  scattering on Poincar\'e--Einstein manifolds~\cite{GZ}
  and related spectral analysis~\cite{Mazzeo,MM}, the
  conjectured AdS/CFT correspondence of physics \cite{Maldacena, Witten, Henningson}, as
  well the treatment of conformal manifold and submanifold invariants
  \cite{GrWi,Will2,Will1, Case,Juhl}. Broadly, the second main problem is to
  study the Yang--Mills action and equations in this
  context. Specifically the first step is to set up boundary problems
  so that (\ref{ymeq}) holds on the interior~$(M_+,\g)$  with suitable asymptotics to the boundary. 
  Ellipticity of this system has been studied in~\cite{deLima}.
  Existence of solutions given by  small deformations of trivial 
  solutions has been established in~\cite{Usula}; this generalizes the 
  analogous result by Graham and Lee for deformations of Poincar\'e--Einstein metrics on the ball~\cite{GLee} and confirms a conjecture due to Witten~\cite{Witten}.
  In this article we 
  solve these Yang--Mills boundary problems formally in order to compute the asymptotics of the
  connection, its curvature and corresponding (regulated)
  action~(\ref{act1}) on $M_+$. In the process we recover key
  conformal invariants of connections that, in particular, answer the
  first question, but also yield significantly more, including a gauge
  field Dirichlet-to-Neumann map, and boundary operators providing natural analogs of the conformal fundamental forms developed recently  in the context of the Poincar\'e--Einstein problem~\cite{Blitz}.

\medskip

To handle the Yang--Mills equations on a conformally compact manifold,
it is propitious to recast the problem using conformal geometry of the manifold $M$
equipped with the conformal class of
metrics $\cc:=[g]=[\Omega^2 g]$, 
where $0<\Omega\in C^\infty M$.
Of particular import are  conformal densities and density-valued tensors; these
are explained in detail in Section~\ref{CG}. A weight $w$
scalar-valued conformal density~$\varphi$ is a section of a certain
line bundle $\ce M[w]$. The former  may be viewed as an
equivalence
class of metric-function pairs $\varphi=[g;f]=[\Omega^2 g;\Omega^w f]$.  (We  often  use the
notation $\varphi \stackrel g = f$ when  labeling $\varphi$ by
a particular equivalence class representative, and even recycle the density's moniker $\varphi$ for this purpose when context allows.)  A {\it defining
  density} $\sigma\in \Gamma(\ce M[1])$ is a weight one
density~$\sigma:=[g;r]$ where $r $ is a defining function for
$\Sigma=\partial M$. The {\it conformal metric} is the weight two,
symmetric rank~2 tensor-valued conformal density given tautologically
by~${\bm g}_{ab}=[g;g_{ab}]$. A conformally compact structure can now be
labeled by the data~$(M,\cc,\sigma)$, where $\sigma$ is a defining
density for $\Sigma$.  The conformally compact metric on~$M_+$ is then
$$
\go_{ab}=\sigma^{-2} {\bm g}_{ab }\, .
$$
The Yang--Mills equations~\nn{ymeq} on  $(M_+,\go)$
can  be  re-expressed invariantly as 
$$
0=\go^{ab}\nabla^{\go }_a F_{bc}^{\phantom{g}}= \sigma  j_c \, ,
$$
where the {\it  conformally compact Yang--Mills current} $j$ is the weight $-1$ covector given by 
$$
{ j_b}[A]:\stackrel g{=} r \nabla^a F_{ab}-(d-4) F_{nb}\, ,\qquad n :\stackrel g{=}\ext r\, .
$$ 
Importantly, if the curvature $F$ extends smoothly to $\Sigma$, it follows that so too does the conformally compact Yang--Mills current $j$.
See 
Section~\ref{CCYME} for a more detailed discussion of the  vector bundle endomorphism and conformal density-valued
conformally compact Yang--Mills current  $j$, and of why it is optimal for our study.


In Section~\ref{Apollosrevenge} we show, via a characteristic exponent analysis,  that  there are two classes of boundary problems for the {\it conformally compact Yang--Mills equation} defined on all of~$M$ by 
\begin{equation}\label{CCYM}
j[A]=0\, .
\end{equation}
 In particular when the curvature behaves according to  
\begin {equation}F^A=\sigma^\alpha \big(G + {\mathcal O}(\sigma)\big)\, ,\end{equation}
for some $G\in\Gamma(\wedge^2 T^*M\otimes \End{\mathcal V}M[-\alpha])$ (here ${\mathcal O}(\sigma)$
denotes $\sigma$ times any smooth section
of $\wedge^2 T^*M\otimes \End{\mathcal V}M[-\alpha-1]$ and  ${\mathcal V}M[w]$ denotes the product bundle ${\mathcal V}M\otimes \ce M[w]$ for any vector bundle ${\mathcal V}M$), 
the two cases are
\begin{itemize}
\item{\it Magnetic:} $\alpha=0\, $ and $G_{\hat n b}\stackrel\Sigma=0$,
\item{\it Electric:} $\alpha=d-4\, $ and $\hat n_{[a} G_{bc] }\stackrel\Sigma=0$.
\end{itemize}
In the above $\hat n\in \Gamma(T^*M[1])|_\Sigma$ denotes the unit conormal to $\Sigma$ .
By analogy to Lorentzian electromagnetism, we have dubbed the first of the above two cases ``magnetic'', because in that setting the normal component of curvature   to a spacelike hypersurface is an electric field, while the magnetic field is  the component of curvature that is non-vanishing when skewed with the unit normal.

A connection  $\nabla^A$ on $M$, smooth on~$M_+$ and sufficiently differentiable for $j[A]$ to be defined, 
 and such that
$$
j[A]=0 \, ,$$ 
is termed a {\it conformally compact Yang--Mills connection}.
A vector bundle ${\mathcal V}M$ canonically induces a vector bundle on the boundary $\Sigma$ by restriction. This will be denoted ${\mathcal V}\Sigma$ throughout this article. 
Also, the tangent bundle to $\Sigma$ can be canonically identified with a subbundle of $TM|_\Sigma$.
Hence any 
connection induces  a smooth boundary connection  $\bar \nabla^{\bar A}$ on $\Sigma$ by pullback, {\it i.e.} for any $ v\in \Gamma(T\Sigma)$
$$
 \bar \nabla_{ v}^{{\bar A}}\stackrel\Sigma=
\nabla^A_{v}\, .
$$
A fundamental question is whether a conformally compact Yang--Mills connection can be determined from the data of a boundary connection $\bar \nabla^{\bar A}$. Therefore our
 first boundary problem is stated and solved as follows.
\begin{problem}
\label{yeswehavethem}
Let $(M,\cc,\sigma)$ be a conformally compact structure for which  a connection~$\bar \nabla^{{\bar A}}$ along $\Sigma$ is given. Find a smooth connection $\nabla^A$ on $M$ such that 
\begin{equation}\label{bc}\nabla^A_v\stackrel \Sigma  = \bar \nabla_v^{{\bar A}}
\end{equation}
for any $v\in \Gamma(T\Sigma)$, and 
$$
 j[A]={\mathcal O}(\sigma^\ell)\, ,
$$
for $\ell\in {\mathbb Z}_{\geq 0}$ as high as possible.
\hfill$\blacksquare$
\end{problem}
\begin{theorem}\label{recurse}
Let $d\geq 3$ and $(M,\cc,\sigma)$ be a conformally compact structure. Then when $d\geq 4$ there exists a solution of Problem~\ref{yeswehavethem} to order $\ell=d-4$. When  $d=3$ there exists an order $\ell=\infty$ solution.
\end{theorem}
\noindent
The proof of this theorem is given in Section~\ref{Apollosrevenge}. 
We call  any connection~$\nabla^A$ 
on a dimension $d\geq 4$ conformally compact structure
that obeys 
\begin{equation}\label{obst}
j[A]=\sigma^{d-4} k\, ,
\end{equation} 
for some smooth $k\in \Gamma(T^*M\otimes\End{\mathcal V}M[3-d])$,
an {\it asymptotically conformally compact  Yang--Mills connection}, or for brevity an {\it asymptotically Yang--Mills connection}. The so-defined boundary connection $\bar \nabla^{\bar A}$  is called the  {\it boundary condition}.
Theorem~\ref{gaugeeq} gives a uniqueness result for solutions to Problem~\ref{yeswehavethem}
up to gauge equivalence.

 The order $\ell=d-4$ is critical, in the sense that there is an obstruction to higher order smooth asymptotic solutions.  This depends on data local to the boundary.  
Corollary~\ref{forgotten-soap} of the uniqueness result in Theorem~\ref{gaugeeq} 
gives a canonical map
\begin{equation}\label{obcu}
(M,\cc,\sigma,{\bar A})
\mapsto \bar k
 \in
\Gamma(T^*\Sigma\otimes\End{\mathcal V}\Sigma [3-d]) \, ,
\end{equation} whose image we dub the {\it obstruction current}.  This already gives a conformally invariant higher
Yang--Mills equation for the boundary connection $\bar \nabla^{\bar
  A}$. For the case of the Maxwell system, {\it i.e.} abelian Yang--Mills, Equation~\nn{ymeq}
simplifies to $\mathring{\delta} F=0$ for $F\in
\Omega^2 M_+$ where~$\mathring{\delta}$ is the divergence operator of the conformally
compact metric. In that case  it is
known~\cite{BG,GLW} that for~$d$ odd,~$\bar k$ has leading structure
$\bar \Delta^{\frac{d-5}{2}} \, \bar \delta \bar F$. 
This leading behavior persists for general Yang--Mills systems.
Here bars are
used to denote boundary objects, so that~$\bar \Delta$ is the boundary
form Laplacian, $\bar \delta$ the boundary divergence, and $\bar F\in
\Omega^2 \Sigma$.  When the obstruction current~$\bar k$ is non-vanishing,
Problem~\ref{yeswehavethem} can be modified to include logarithms of
the defining density, see Problem~\ref{lets_get_loggy} whose solution
is given in Proposition~\ref{logsol}.
 
While Corollary~\ref{forgotten-soap}
shows how to obtain higher Yang--Mills equations {\it holographically}, by which we mean 
solving a bulk problem on $M_+$ to find information along $\Sigma$, we would like finer information about $\bar k$. Before discussing that, we note that a second ``electric'' boundary problem appears exactly at the order of the obstruction  to the magnetic one. This problem is stated as follows.
\begin{problem}\label{problemono}
Let $(M,\cc,\sigma)$ be a conformally compact structure and let $0>\alpha\in {\mathbb R}$. 
Find a smooth connection $\nabla^A$ on $M$ such that 
$$
\sigma^\alpha F^A\in \Gamma(\wedge^2 M\otimes\End{\mathcal V}M[\alpha])
$$
and
$$
 j[A]={\mathcal O}(\sigma^\ell)\, ,
$$
for $\ell\in {\mathbb Z}_{\geq 0}$ as high as possible.
\hfill$\blacksquare$
\end{problem}
\noindent
Lemmas~\ref{handofgod} and~\ref{fingerofgod} show that neither the boundary data nor the power $\alpha$ appearing in the above problem are arbitrary. Instead we must have $-\alpha=d-4$ and $\sigma^\alpha F_{ab}|_\Sigma=\hat n_a E_b -\hat n_b E_a$. Moreover  the {\it electric field} $E\in  \Gamma(T^* \Sigma\otimes\End{\mathcal V}\Sigma[3-d])$ must obey an analog of the (vacuum) ``Gau\ss\ law''.
The following theorem establishes the existence of an  all order asymptotic solution to Problem~\ref{problemono} when these conditions hold.
\begin{theorem}\label{electrictheorem}
Let $(M,\cc,\sigma)$ be a conformally compact structure with $d\geq 5$. 
Also, given a smooth connection $\bar \nabla^{\bar A}$ on  $\Sigma$, 
suppose that there exists 
$E\in  \Gamma(T^* \Sigma\otimes\End{\mathcal V}\Sigma[3-d])$ satisfying the Gau\ss\ law 
$$
\bar \nabla^{{\bar A}}_a E^a=0 \, .
$$
Then if
  $\nabla^{A_0}$  is a smooth connection on $M$ with $ \nabla_v^{A_0}\stackrel\Sigma=\bar \nabla_v^{\bar  A}$ for  any  $v\in \Gamma(T\Sigma)$,
  and 
$$
F^{A_0}=\sigma^{d-4} G
$$
for some
$
G\in 
 \Gamma(\wedge^2 M\otimes\End{\mathcal V}M[4-d])
$
where $G|_\Sigma
:=\hat n_{a} E_{b}-\hat n_b E_a 
$,
 there exists 
$a\in\Gamma(
T^*M\otimes\End{\mathcal V}M[2-d])$
such that 
$
\nabla^A=
\nabla^{A_0}+\sigma^{d-2} a
$
obeys
$$j[A]={\mathcal O}(\sigma^\infty)\, .
$$

\end{theorem}
Theorem~\ref{all-orders-magnetic} describes how 
 magnetic and electric solutions can be combined  to make all order asymptotic solutions in the case that the obstruction current vanishes. 
  \begin{theorem}\label{all-orders-magnetic}
 Let $\nabla^{A_0}$ be an
asymptotically Yang--Mills connection
 on a conformally compact manifold  $(M,\cc,\sigma)$. Suppose that the boundary condition~$\bar \nabla^{{\bar A}}$   is 
 such that the corresponding obstruction current $\bar k=0$, then $$
 j[A_0]=\sigma^{d-3} k^{(1)}\, ,
 $$
 for some $k^{(1)}\in \Gamma(T^* M\otimes\End{\mathcal V}M[2-d])$.

Now let $d\geq 4$, and suppose that there is a section $E\in  \Gamma(T^* \Sigma\otimes\End{\mathcal V}\Sigma[3-d])$ satisfying the Gau\ss\ law
\begin{equation}
\bar \nabla^{{\bar A}}_a E^a\stackrel\Sigma= \frac 1{d-3}\, \hat n^a k_a^{(1)} \, .
\end{equation}
Then, given the above data, there exists 
a smooth
asymptotically Yang--Mills connection~$\nabla^A$ on $M$ with boundary condition
$ \bar \nabla^{{\bar A}}$
such that 
$$
 j[A]={\mathcal O}(\sigma^\infty)\, .
$$
 \end{theorem}
 \noindent
 The proofs of Theorems~\ref{electrictheorem} and~\ref{all-orders-magnetic}  are given in Section~\ref{EP}.
 
 \smallskip
 
 The  formal solution of Theorem~\ref{all-orders-magnetic} admits  arbitrary  combinations of magnetic boundary data and electric data (albeit subject to the boundary Gau\ss\ law and vanishing obstruction current).
 This freedom will  be lost when searching for global, rather than  asymptotic solutions. Instead there is  a notion of a Dirichlet-to-Neumann map for the conformally compact Yang--Mills system encoding the relationship between magnetic and electric boundary data for glabal solutions. Dirichlet-to-Neumann maps are defined for a model Maxwell system in Section~\ref{models} and for  the Yang--Mills problem in Section~\ref{connexp}.

\smallskip

It is interesting to find conformally invariant energy functionals for the 
 conformally invariant higher Yang--Mills equations given in Corollary~\ref{forgotten-soap}. For that we apply holographic methods to the 
 {\it regulated Yang--Mills action functional}
  $$
S^\varepsilon[A]  := -\frac14\int_{M_\varepsilon} 
\ext {\rm Vol}(\go)\,
\operatorname{Trace}(\go^{ab}\go^{cd}F_{ac}F_{bd})\, .
$$
Here $M_\varepsilon\subseteq M_+$ is a smooth one-parameter family of truncated manifolds such that $M_0=M_+$ (see Section~\ref{renorm}). The above expression is in general singular in the limit $\varepsilon \to 0$ because  the  conformally compact metric $\go$  itself is singular along  the boundary $\Sigma$. However, its   small $\varepsilon$ asymptotics can be computed and are given as follows.
\begin{theorem}\label{Laurent}
Let $d \geq 6$, $(M,\cc,\sigma)$ be  a conformally compact structure, and $\nabla^A$ a connection for a vector bundle over~$M$. Then the regulated Yang--Mills action functional has a Laurent series plus log term expansion 
$$
S^\varepsilon[A]  = 
\frac{v_{d-5}}{(d-5)\varepsilon^{d-5}}+
\frac{v_{d-6}}{(d-6)\varepsilon^{d-6}}+
\cdots + \frac{v_1}{\varepsilon} + {\rm En}[A]\hh \log\frac1\varepsilon  + S^{\rm ren}[A]
+ {\mathcal O}(\varepsilon) \, ,
$$
where the coefficients 
are given by the local integrals
$$
v_\ell = \frac{(\ell+3)!}{(d-5-\ell)!(d-2)!} \int_\Sigma {\ext {\rm Vol}(\bm {\bar g})}{\sqrt{I^2}\, }\,  \Big(\frac1{I^2} I.D\Big)^{d-5-\ell}
\Big(\frac{-\operatorname{Tr} F^2}{4I^2 \tau^\ell}\Big)\, ,\:\: 0\leq\! \ell\!\leq \!d-5\, .
$$
Moreover, the energy ${\rm En}[A]$ is independent of the choice of regulator $0<\tau\in\Gamma(\ce M[1])$
and so is an invariant of the structure.
\end{theorem}
\noindent
This result is proved in  Section~\ref{renorm}.
In the above 
$F^2=\bg^{ab}\bg^{cd} F_{ac}F_{bd}$.
Also, the operator $I.D: \Gamma({\mathcal V}M[w])\to \Gamma({\mathcal V}M[w-1])$, is  the 
{\it (connection twisted) Laplace--Robin operator} (see~\cite{Goal,GW}) which gives 
 a canonical degenerate extension of a combination of the  interior Laplace operator $\Delta^{\go}$ and scalar curvature $Sc^{g^o}$. Indeed, the smooth function denoted by $I^2\in C^\infty M$ is a canonical extension of $-Sc^{g^o}/\big(d(d-1)\big)$ to $M$; again
see Section~\ref{renorm}. Note  that $\bar {\bm g}$ is the conformal metric on $\Sigma$ induced by ${\bm g}$, and $\ext {\rm Vol}(\bar {\bm g})$ the corresponding (density-valued) volume element. The notion of regulators  given in terms of non-vanishing conformal densities $\tau$ used here, 
as well as how they control both the choice of truncation~$M_\varepsilon$ and a corresponding metric representative along $\Sigma$,
is  described in Section~\ref{renorm}.

The boundary energy ${\rm En}[A]$ and renormalized Yang--Mills action functional~$S^{\rm ren}[A]
$ given, respectively, by the coefficient of the logarithm and $\varepsilon$-independent term in the asymptotics of the regulated Yang--Mills action functional are of particular interest. The latter may be viewed as the Yang--Mills analog of the renormalized volume of a Poincar\'e--Einstein manifold. Its dependence 
on the regulator is captured by the following theorem.
\begin{theorem}\label{wehaveenergy}
The renormalized action $S^{\rm ren}[A]$ of Theorem~\ref{Laurent} depends on the choice of regulator $\tau$ according to
\begin{equation}\label{Srentr}
S^{\rm ren}[A;e^\omega \tau]-S^{\rm ren}[A;\tau]=\frac{3!}{(d-5)!(d-2)!}\int_\Sigma \!
\ext {\rm Vol}(\bm {\bar g}){\sqrt{I^2}}\, 
\Big(\frac1{I^2} I.D\Big)^{d-5}\Big(\frac{
-\operatorname{Tr}F^2\hh \omega
}{4I^2}\Big)\, ,
\end{equation}
where $\omega$ is any element of $C^\infty M$. 
\end{theorem}

For even dimensional Poincar\'e--Einstein structures, a stronger result for the renormalized action is available. 
\begin{theorem}\label{evenvarySren}
Let $(M,\cc,\sigma)$ be Poincar\'e--Einstein with $d\in\{4,6,8,\ldots\}$.
Then the renomalized Yang--Mills action is conformally invariant, {\it i.e.}
$$
S^{\rm ren}[A;e^\omega \tau]-S^{\rm ren}[A;\tau]=0\, ,
$$
where $\omega$ is any element of $C^\infty M$. 
\end{theorem}

When $\omega$ is a non-zero constant, the right hand side of Equation~\nn{Srentr} gives 
a holographic formula for a conformally invariant
higher Yang--Mills energy functional 
along the boundary of any conformally compact structure with dimension $d\geq 5$,
\begin{equation}\label{HolFor}
{\rm En}[A] = 
\frac{3!}{(d-5)!(d-2)!} 
\int_\Sigma {\ext {\rm Vol}(\bm {\bar g})}{\sqrt{I^2} }\,  
\Big(\frac1{I^2} I.D\Big)^{d-5}
\Big(\frac{-\operatorname{Tr}F^2}
{4I^2}\Big)\, .
\end{equation}

When $A$ is a solution to the conformally compact Yang--Mills equations, the above gives an energy formula depending only on the boundary connection. 
For Poincar\'e--Einstein structures, this can be easily computed on low even dimensional boundaries and vanishes whenever the boundary dimension is odd.
 \begin{theorem}\label{blabla}
Let   $(M^d,\cc,\sigma)$  be conformally compact with $d\geq 5$ and $\nabla^A$ an asymptotically Yang--Mills connection with boundary condition $\bar \nabla^{\bar A}$. 
 Then the energy functional $\overline{\operatorname{En}}[\bar A]:=\operatorname{En}[ A]$  given by Equation~\nn{HolFor} is completely determined by $\bar A$. Moreover, 
 $$
\overline{ \operatorname{En}}[\bar A]=\int_\Sigma {\ext {\rm Vol}(\bm {\bar g})} Q[\bm g,A]\, ,
 $$
 where $$Q[\bm g,A]
 :=-\tfrac14\tfrac{3!}{(d-5)!(d-2)!} \, 
 I.D^{d-5}
\operatorname{Tr} F^2
 \in \Gamma(\ce \Sigma[1-d])\, .$$
 When $(M^d,\cc,\sigma)$ is Poincar\'e--Einstein, then 
 $$
 Q[\bm g,A]=
 \left\{
 \begin{array}{cl}
 -\tfrac14 \operatorname{Tr} \bar F^2
 \, , & d=5\, ,\\[2mm]
-\tfrac1{32} \bar \Delta \operatorname{Tr} \bar F^2
-\tfrac18 \operatorname{Tr} \big( 
{\bj}^2
-\bar J 
{\bar F}^2
+2{\bar F}^{ab}
\bar \nabla_{[a} {\bj}_{b]}
-4{\bar F}^{ab}\bar P_{a}{}^c \bar F_{bc}
\big)\, ,
& d=7\, .\\[2mm]
\end{array}\right.
$$
Moreover,  $\overline{\operatorname{En}}[\bar A]=0$ for all $d\in\{6,8,10,\ldots\}$.
\end{theorem}

The functional gradient of this energy gives the obstruction current.
\begin{theorem}
\label{fruit=vary}
Let $(M,\cc,\sigma)$ be conformally compact and $\nabla^A$ an asymptotically Yang--Mills connection 
with  boundary condition $\bar \nabla^{{\bar A}}$.
Moreover let~$\bar A_t$ be a one parameter family of  boundary connections such that $\bar A_t-{\bar A}$ has compact support (on $\Sigma$)
and $\bar A^0={\bar A}$.
Then, the functional gradient of the energy 
$\ext{\rm  En}\big[A[\bar A_t]\big]/\ext t\big|_{t=0}$, 
for any one parameter family of asymptotically Yang--Mills connections $A[\bar A_t]$ with boundary conditions $\bar \nabla^{\bar A_t}$,
 is a non-zero multiple of the 
obstruction current $\bar k$.
\end{theorem}
\noindent

\smallskip

In the case that $(M,\cc ,\sigma)$ is Poincar\'e--Einstein it is not difficult, at least for low-lying values of $d$,  to 
express the holographic formula for the energy in terms of quantities intrinsic to~$\Sigma$, see Theorem~\ref{blabla}; this recovers the energy of~\cite{GoPeSl} when $d=7$.
We also construct 
a recursion to compute the obstruction current; this yields the following result.
\begin{theorem}\label{obsts}
Let  $(M^d,\cc,\sigma)$ be  asymptotically Poincar\'e--Einstein and $\nabla^A$ an asymptotically Yang--Mills connection. 
 Then the obstruction current $\bar k \in\Gamma(T^*\Sigma\otimes\End{\mathcal V}\Sigma[3-d])$ is given by
  $$
\bar k_b=
 \left\{
 \begin{array}{cl}
 \bj_b \, , & d=5\, ,\\[3mm]
  0 \, , & d=6\, ,\\[2mm]
 \tfrac12\bar \nabla^a
\Big(
 \bar \nabla_{[a} {\bj}_{b]} 
 - 4\bar P_{[a}{}^{c} \bar F_{b]c}
-\bar J \bar F_{ab}
\Big)
+\frac14 [\bj^a,\bar F_{ab}] 
\, ,& d=7\, ,\\[3mm]
0\, , & d=8,10,12,\ldots \, .
\end{array}\right.
$$
\end{theorem}
\noindent
Here $\bj_b:=\bar \nabla^a \bar F_{ab}$ is the {\it boundary Yang--Mills current}.
The proofs of Theorems~\ref{Laurent} through~\ref{obsts} are given in Section~\ref{renorm}. 
The vanishing results for higher even dimensions $d$ in 
Theorems~\ref{blabla} and~\ref{obsts}
 rely on 
an evenness property of the asymptotics of conformally compact Yang--Mills connections when expressed in a canonical coordinate system. This is discussed in Sections~\ref{canexp} and~\ref{connexp}.

\medskip

Linking the above results we also consider the problem of transverse boundary operators and invariants. 
For  scalar  Laplacian eigenvector boundary problems on a 
conformally compact  manifold
it is useful to have conformally invariant boundary operators that probe transverse jets of purported solutions. The first example of such is the conformal Robin operator of~\cite{Cherrier,Escobar}
$$
\nabla_{\hat n}-w H^g
$$
acting on weight $w$ densities, where $H^g$ is the boundary mean curvature for the choice of metric $g\in\cc$ used to write the formula. See~\cite{Case,GPt,Blitz} for higher order examples and constructions, as well as Section~\ref{no} for explicit formul\ae\ on Poincar\'e--Einstein structures.
These ideas extend to the
Poincar\'e--Einstein filling problem where there are invariants, termed conformal fundamental forms~\cite{Blitz}, that probe transverse jets of the singular Yamabe metric on a conformally compact manifold, and whose vanishing is a necessary condition for this metric to be Poincar\'e--Einstein. The first non-trivial conformal fundamental form is the trace-free second fundamental form~\cite{LeBrun,Goal}, see~\cite{Blitz} for  higher order constructions.

It is natural to investigate Yang--Mills analogs of these operators and invariants.
By explicit computation we show that
there indeed exist connection-coupled and non-linear analogs of these operators, at least for transverse orders one through four
on Poincar\'e--Einstein backgrounds.
These operators act on bulk connections and are evaluated along the boundary. The first example 
is the normal component of the curvature tensor~$
F_{\hat n b}|_\Sigma\, ,
$
which probes the first transverse jet of the corresponding connection. The second example is
the conformally invariant, transverse order two, one-form valued, endomorphism field
$$(d-5)\big[\hat n^a  (\nabla_{\hat n} F_{ab})+2H^g F_{\hat n b}\big]\big|_\Sigma
-{\bj}_b\, .$$
The third and fourth examples are considerably more involved and are presented in Section~\ref{YMNO}. The displayed operator already exhibits many of the main features. In dimensions $d\geq6$, it vanishes when evaluated on solutions to Problem~\ref{yeswehavethem}, {\it i.e.} asymptotcally Yang--Mills solutions.
Hence it can be used to extract the second order coefficient $\frac1{2!(d-5)} \bj$  in the
canonical coordinate
 expansion of the connection. In the critical case $d=5$, the above operator returns the obstruction current of 
Theorem~\ref{obsts}. Moreover, if the latter vanishes, and one also imposes flatness of the boundary curvature~$\bar F$, the operator in square brackets above gives a conformally invariant expression for the non-local Neumann data determined by global solutions when such exist. 
This is the Yang--Mills analog of a phenomenon observed for the Poincar\'e--Einstein filling problem in~\cite{Blitzedagain}.

\smallskip

The boundary operators of~\cite{GPt,Case} can be viewed as a conformal encoding of the boundary data for higher order conformally invariant problems. These include the kernel of the Paneitz operator. Similarly, conformal fundamental forms encode boundary data for the Bach equation on conformally compact four manifolds as well as the Fefferman--Graham obstruction-flat equations in higher even dimensions.
The Yang--Mills boundary operators presented in Section~\ref{YMNO}
 play  the same {\it r\^ole} for 
 the higher  Yang--Mills 
equations extremizing the invariant energies presented in Theorem~\ref{wehaveenergy}.


\subsection{Riemannian and Vector Bundle Conventions}


Throughout  $M$ will denote a smooth, for simplicity oriented,  $d$-manifold.
 We denote the boundary~$\partial M$ by~$\Sigma$.
 In general smooth will mean $C^\infty$ and where not explicitly stated, all structures such as sections of bundles will assumed to be smooth, although in some cases we will deliberately point out the smoothness properties of given sections for emphasis.
 
The exterior derivative will be denoted $\ext$ and
tensors on  $M$ will most often be handled using an abstract index notation~\cite{Penrose} so that, for example, $T^{abc}_{de}$  denotes a section of the bundle $(\otimes^3 TM) \otimes (\otimes^2 T^*M)$. 
A metric tensor $g\in \Gamma(\odot^2 T^*M)$  is then denoted $g_{ab}$ and its inverse by $g^{ab}$, so that $g^{ac}g_{cb}=\delta_b^a$ is the identity endomorphism on $\Gamma(TM)$. Indices are raised and lowered in the standard way, so for example $v_a := g_{ab} v^b\in \Gamma(T^*M)$ when $v^b\in \Gamma(TM)$. We also employ  shorthand notations such as
$v^2:=g_{ab} v^a v^b$, $u.v=u_a v^a$ and
more generally
 $X^2:=X_{abc\cdots} X^{abc\cdots}$.
 When $v$ is a vector and $\mu$ a covector, the notations $X_{avc}$ denote $v^b X_{abc}$ and $X_{a\mu c}:=\mu_bg^{bd}X_{adc}$. 
Also, square brackets denote antisymmetrization (with unit weight) over a group of indices, so for example
$
X_{[ab}Y_{c]}:=\frac16(X_{ab}Y_c +
X_{bc}Y_a +
X_{ca}Y_b -
X_{ba}Y_c -
X_{ac}Y_b -
X_{cb}Y_a)
$.
Similarly round brackets denote symmetrization, so $X^{(a}Y^{b)}:=\frac12(X^a Y^b + X^b Y^a)$.
The trace-free part of a tensor with respect to a metric is denoted by a $\circ$ notation, so 
$
X_{(a}Y_{b)_\circ}:=\frac12(X_a Y_b+ X_b Y_a)-\frac1d g_{ab} X.Y\in \Gamma(\odot_\circ T^*M) $. The  volume element defined by a metric $g$ is denoted
$\ext {\rm Vol}(g)$.

\smallskip
The Levi--Civita connection of a metric $g$ on $M$ is denoted $\nabla$ (or $\nabla^g$ when disambiguation requires). The corresponding Riemann curvature tensor $R$ is defined by the commutator of connections acting on a vector $v\in \Gamma(TM)$, in the abstract index notation this reads
$$
[\nabla_a , \nabla_b] v^c := R_{ab}{}^c{}_d v^d\, .
$$
The Ricci tensor $Ric_{ab}:=-R_{ac}{}^c{}_b$ and is related to the Schouten tensor $P$ and its trace~$J:=P_a{}^a$ by
\begin{equation}
\label{Schouten}
Ric=:(d-2) P + g J\, .
\end{equation}
The trace-free Weyl tensor $W$ is then defined by the identity
\begin{equation}
\label{Weyl}
R_{abcd}=:W_{abcd}+g_{ac} P_{bd}
-g_{bc} P_{ad}
+g_{bd} P_{ac}
-g_{ad} P_{bc}\, ,
\end{equation}
and the
Cotton tensor
\begin{equation}
\label{Cotton}
C_{abc}:= \nabla_a P_{bc}-\nabla_b P_{ac}\, .
\end{equation}
Note also the useful identity
\begin{equation}\label{divWeyl}
\nabla^a W_{abcd}
=(d-3) C_{cdb}\, .
\end{equation}
We sometimes employ the divergence operator on differential forms, which we call $\delta$; also $\iota$ is used for interior multiplication. The rough Laplacian $\Delta:=g^{ab} \nabla_a \nabla_b$ and acts on tensors of any type. 

Given a vector bundle ${\mathcal V}M$ over $M$, we denote a composition $X\circ Y$ of endomorphisms by juxtaposition $XY$,
and denote the trace of such by $\Trace$ or simply $\operatorname{Tr}$. Note that, by convention, the zeroth power of an operator or endomorphism is taken to be the identity $\operatorname{ Id}$.
We recycle the symbol $\nabla$ for a connection $A$ on ${\mathcal V}M$, and write $\nabla^A$ when disambiguation is necessary. We also use $\nabla$ or $\nabla^A$ when this  connection is Levi-Civita coupled. The curvature $F^A$ of $\nabla^A$ is the endomorphism-valued two-form, defined by a commutator  of connections acting on sections of ${\mathcal V}M$
$$
[\nabla^A_a,\nabla^A_b]=F^A_{ab}\, .
$$
Note, for example, if $v^a\in \Gamma(TM\otimes{\mathcal VM})$, then
$$
[\nabla_a,\nabla_b] v^c = (\nabla_a\nabla_b -\nabla_b\nabla_a) v^c = R_{ab}{}^c{}_d v^d + F_{ab} v^c\, .
$$
Also, we employ a ${\mathcal O}$ notation where $X={\mathcal O}(f^k)$ means that the tensor field $X=f^k Y$ for some other smooth tensor $Y$. The notation  ${\mathcal O}(f^\infty)$ says there exists $Y$ such that $X=f^k Y$ for any positive integer of value of $k$.

\section{Conformal Geometry}
\label{CG}

A conformal manifold $(M,\cc)$ is the data of a smooth manifold  $M$ equipped with a class of conformally related metrics $\cc$, such that $g,g'\in \cc \Rightarrow g'=\Omega^2 g$ for some $0<\Omega\in C^\infty M$.
A key notion for a conformal manifold is that of a conformal density.
A conformal class of metrics determines a ray  
subbundle ${\mathcal G}M$ of~$\odot^2T^*M$ consisting of
 conformally related metrics in $\odot^2T^*_pM$  at each point $p$ in $M$.
The bundle ${\mathcal G}M$ is naturally an ${\mathbb R}^+$-principal bundle, and  the bundle~$\ce M[w]$ of {\it weight $w\in {\mathbb C}$ conformal densities} is then the associated vector bundle $$\ce M[w]:={\mathcal G}\times_{{\mathbb R}_+} {\mathbb R}\, ,$$ where~$t\in {\mathbb R}_+$ acts on $x\in {\mathbb R}$ according to $x\mapsto t^{-\frac w2} x$. Sections  of $\ce M[w]$ are thus equivalence classes $[g;f]=[\Omega^2 g;\Omega^w f]$ where $\cc =[g]$ and $f\in C^\infty M$. 
Each density bundle~$\ce M[w]$ is oriented
and, for $M$ oriented,  is isomorphic to the $-(w/d)$-th power of the top exterior power of the cotangent bundle.
 The Levi--Civita connection thus acts on  $\ce M[w]$.
For any vector bundle~$B M$ over $M$, we denote  the bundle $BM\otimes \ce M[w]$ by~$BM[w]$.

Densities afford a useful description of distinguished metrics, {\it viz.}
a strictly positive section $\tau_g\in \Gamma(\ce M[1])$ determines a metric $g\in  \cc$, or equivalently a section of ${\mathcal G}M$, characterized by ~$\tau_g =[g;1]$. Thus the product
$$
\bm g := \tau_g^2 g
$$
defines a tautological section of ${\mathcal G}M[2]$, termed the {\it conformal metric}.
Alternatively, the data  $0<\tau \in \Gamma(\ce M[1])$ determines a metric $g\in \cc$ via $g= \tau^{-2} {\bm g}$. The data $0<\tau\in \Gamma(\ce M[1])$ is called a {\it true scale} and, along similar lines, the choice of a metric~$g\in \cc$ can be referred to as a {\it choice of scale}. Often, it will be convenient to perform 
(Riemannian geometry) 
computations for a choice of $g\in \cc$. For that,  given a section $B\in \Gamma(BM[w])$, we obtain a section $B/\tau_g^{w}\in \Gamma(BM)$, for which we recycle the notation $B$ when the choice of $g$ used has been made clear (often by the notation $\stackrel g=$). This maneuver is often termed ``trivializing density bundles''. Note that the density-valued volume element defined by a conformal metric $\bm g$ will be denoted $\ext {\rm Vol}(\bm g)$. 

\subsection{Conformally Compact Manifolds} \label{setup}

%
%

It is advantageous to describe conformally compact structures in a way that brings the conformal geometry of their interior to the fore,
as encapsulated by the following definition.

\begin{definition}
A {\em conformally compactified manifold} is a conformal $d$-manifold $(M,\cc)$ with boundary $\Sigma:=\partial M$ and  
 a non-negative section $\sigma\in \Gamma(\ce M[1])$  that {\em defines} $\Sigma$, meaning $$
\partial M = \sigma^{-1}(0)
\:\:\mbox{ 
and }\:\:
\nabla \sigma (P)\neq 0 \, ,\:\forall P\in \Sigma\,  ,
$$
where $\nabla$ is the Levi--Civita connection.
\hfill$\blacksquare$
\end{definition}
\noindent
Throughout, we focus on the case where $\Sigma=\partial M$ is a hypersurface.
The interior $M_+:=M\backslash\Sigma$ of a conformally compactified manifold is evidently equipped 
with a distinguished metric $\go \in \cc|_{M_+}$ given by
$$
\go_{ab} :=\frac{\bg_{ab}}{\sigma^2}\, .
$$
The boundary $(\Sigma,\cc_\Sigma)$ equipped with the conformal class of metrics induced by $\cc$ is termed the {\it conformal infinity} of the {\it singular metric} $\go$. The density $\sigma$ is often termed a {\it  defining density}. We often use either  a sub or superscript $\Sigma$ for quantities defined along the hypersurface $\Sigma$. Also, a bar notation will  be employed for the special case of quantities that can be defined intrinsically along $\Sigma$, meaning that their metric dependence is only through the induced metric 
$\bar g$ of $\Sigma$.

\medskip
The following lemma leads to a useful alternative characterization of a conformally compact structure.
\begin{lemma}\label{II}
Let $(M,\cc,\sigma)$ be a conformally compact structure with hypersurface boundary $\Sigma=\partial M$. Then the smooth function defined, for any $g\in \cc$, by
$$
I^2:=
|\ext \sigma|_{g}^2 - \frac {2\sigma}d (\Delta^{g} \sigma + J^{g} \sigma) 
$$
is conformally invariant, and in the interior $M_+$ equals $-Sc^{\go}/\big(d(d-1)\big)$. 
 \end{lemma} 

\begin{proof}
The proof is elementary, see for example~\cite{Goal}.
\end{proof}

Paying homage to its tractor calculus origins~\cite{BEG,GN}, we have denoted the above smooth extension of $-Sc^{\go}/\big(d(d-1)\big)$ to $M$  
by~$I^2$.
 Indeed,~$I^2$  is non-vanishing along the zero locus $\partial M$ of~$\sigma$
for   conformally compact structures with hypersurface boundary and is therefore positive in a neighborhood thereof. In fact, without loss of generality, we may assume this holds everywhere in $M$.
We will also need a 
 related 
  operator
 \begin{equation}\label{ID}
 I.D : \Gamma({\mathcal V}M[w])\ni v \stackrel g \mapsto 
 \big[ (d+2w-2) (\nabla^g_n + w \rho) - \sigma (\Delta^g + wJ^g)\big] v \in {\mathcal V}M[w-1]\, .
 \end{equation}
The above is defined for some $g\in \cc$. Also
 $n:=\ext \sigma$ and $\rho := -\frac 1d (\Delta^g  + J^g) \sigma$. 
 Its conformal invariance is a standard result in tractor calculus relying on a ``strong invariance property'' of the so-called Thomas $D$-operator, see for example~\cite{Goal,GW,Goann}. An explicit verification is also not difficult to carry out.
 Note that in the interior~$M_+$, 
 \begin{equation}\label{intID}
 I.D\stackrel {\go}=-\Delta^{\go}-\frac{2J^{\go}}d w(d+w-1)\, .
 \end{equation}
 Along  the boundary, the operator 
  \begin{equation}\label{Robin}
 \delta^{(1)}\hh  v :\stackrel g= ( \nabla_{\hat n} v -w H^g v)\big|_\Sigma\in{\mathcal V}\Sigma[w-1]
 \end{equation}
is a conformally invariant, first order, Robin-type operator generalizing a result of~\cite{Cherrier}.
Because along $\Sigma$ we have that $I.D\stackrel\Sigma = (d+2w-2) \delta^{(1)}$ and by virtue of Equation~\nn{intID},
we term the $I.D$ operator the  {\it (Yang--Mills twisted) Laplace--Robin operator}. The following property   this operator underlies the  ${\mathfrak sl}(2)$ solution generating algebra employed in~\cite{GW} to study comformally compact boundary problems.
 \begin{lemma}\label{sl2}
 Let $(M,\cc,\sigma)$ be a conformally compact structure and $\nu \in \Gamma({\mathcal V}M[w])$. Then
 $$
 I.D ( \sigma \nu) - \sigma I.D \nu = I^2 (d+2w) \nu\, .
 $$
 \end{lemma}
 
 \begin{proof}
 The proof follows, {\it mutatis mutandis}, that of~\cite[Lemma 3.1]{GW}. 
 \end{proof}
 
 \subsection{Conformal Variations and Basic Invariants}\label{LCVs}

We will need to study the conformal invariance of  tensors built from various combinations of hypersurface invariants coupled to some connection. 
The  trace-free part $\IIo_{ab}:=\II_{(ab)_\circ}$  of the second fundamental form 
$$\II_{ab}:=\nabla^\top_a \hat n_b^{\rm ext}\big|_\Sigma
$$
defines a section of $\odot_\circ T^*\Sigma[1]$.
Here $\hat n^{\rm ext}$ is any smooth extension of the unit conormal of determined by the boundary hypersurface embedding $\Sigma\hookrightarrow (M,g)$ to~$M$, for some $g\in \cc$. The tangential derivative $\nabla^\top$ is the operator defined along $\Sigma$ by 
$$\nabla^\top:\stackrel\Sigma{=} \nabla -\hat  n \nabla_{\hat n}\, ,$$
and the symbol $\top$ applied to any tensor along $\Sigma$ indicates orthogonal projection by
$\delta^a_b - \hat n^a \hat n_b$.
Note that the trace-free second fundamental form and unit conormal are both invariants of the boundary conformal embedding $\Sigma\hookrightarrow(M,\cc)$, while the mean curvature trace $$H:=\frac1{d-1} \II_a^a$$ of the second fundamental form is not conformally invariant and  thus only gives an invariant of the Riemannian embedding $\Sigma\hookrightarrow(M,g)$.

\medskip

To study when a Riemannian invariant is also 
conformally invariant, 
 it will suffice to consider its linearized conformal variation. Thus
we define an operator~$\delta$ on tensor-valued functionals $T[g]$ of a metric~$g$, by
\begin{equation}\label{delta}
\delta T[g] := \frac{\ext T[e^{2t\varpi} g]}{\ext t}\Big|_{t=0}\, ,
\end{equation}
where $\varpi\in C^\infty M$. 
Thus conformal variations of the Levi-Civita connection, acting on a vector $v\in \Gamma(TM)$ 
and a covector $\mu\in \Gamma(T^*M)$,
obey
\begin{eqnarray}
\nabla^{\Omega^2 g}_a v^b
-
\nabla_a^g v^b\hh\; &=&\:\:\:\,
\Upsilon_a v^b +\delta^a_b v_\Upsilon - \Upsilon^b g_{ac} v^c \;\hh\hh\hh=\;\delta (\nabla_a^g v^b)\, ,
\nonumber\\[-2mm]
\label{deltanablas}
\\[-2mm]
\nabla^{\Omega^2 g}_a \omega_b - \nabla^g_a \omega_b^{\phantom a}
&=&
-\Upsilon_a \mu_b - \Upsilon_b \mu_a + g_{ab} \Upsilon^c \mu_c
\;=\;\delta(\nabla_a^g \mu_b)
\, .\nonumber
\end{eqnarray}
where $\Upsilon :=\ext \varpi$.

An important result of~\cite{Branson} is that $T[g]$ takes values in  conformally  densities of weight~$w$ if and only if its {\it linearized conformal variation} $\delta T[g]$ obeys
$$
\delta T[g]=w\hh \varpi T[g]\, .
$$
The argument of~\cite{Branson} can be applied {\it mutatis mutandis} to  invariants $T[g]$ of the  boundary conformal embedding $\Sigma\hookrightarrow(M,\cc)$, for which we must require 
$$
\delta  T[g]\stackrel\Sigma=w \hh \bar\varpi T[g]\, ,
$$ where $\bar \varpi = \varpi|_\Sigma$.
In the above $T[g]$ is any tensor defined along $\Sigma$ from the ambient metric~$g$. In particular $T$ may take values in the normal bundle and tensor powers thereof.
For example, if $s$ is any defining function for $\Sigma$, then $\hat n=\ext s/|\ext s|_g|_\Sigma$~and
$$
\delta \hat n = \delta\Big(\frac{\ext s}{|\ext s|_g}\Big|_\Sigma\Big)
=
\frac{\ext }{\ext t}
\Big(\frac{\ext s}{|\ext s|_{e^{2\varpi}g}}\Big|_\Sigma\Big)
\Big|_{t=0}
=\frac{\ext }{\ext t}(e^{t\bar\varpi} \hat n)|_{t=0}=\bar\varpi \hat n\, .$$
Hence
$$
\delta \bar g = \delta(g|_\Sigma- \hat n \otimes \hat n) = 2\bar \varpi\bar g\, .
$$
The above two displays establish that $\hat n\in \Gamma(T^*\Sigma[1])$, $\bar g \in \Gamma(\otimes^2T^*\Sigma[2])$, and that both $\hat n$ and $\bar g$ are hypersurface invariants. The induced metric~$\bar g$ is an intrinsic invariant of $\Sigma$. 

\medskip

By way of notational convenience, since we only deal with tensors of a definite conformal weight, when acting on a tensor $T[g]$ along $\Sigma$,
we define
$$\delta_{w} := \delta - w \bar \varpi\, ,$$
so that, for example, $\delta_{1}\hat n= 0 = \delta_{2}\bar g$, and
\begin{equation}\label{varyH}
\delta_{-1} H = 
 \Upsilon_{\hat n}\, .
 \end{equation}
  In general for any scalar $X$,
$$
\delta_{k} \bar \nabla_a X = \bar \nabla_a \delta_{k} X + k X \bar \Upsilon_a\, .
$$
 It follows that
 \begin{equation}\label{varyH1}\delta_{-1} \bar \nabla_a H = \bar \nabla_a \Upsilon_{\hat n}-H\bar \Upsilon_a 
 \Rightarrow \delta_{-3}\bar
 \Delta H
 =\bar \Delta \Upsilon_{\hat n}+\big((d-5)\bar\nabla_{\bar \Upsilon}
- (\bar \nabla^a \bar \Upsilon_a)
 \big) H
 \, ,
\end{equation}
where  $\bar \Upsilon := \Upsilon^\top|_\Sigma\in \Gamma(T^*\Sigma)$. 

The well-known conformal transformation law for  the Schouten tensor
\begin{equation}\label{shouting}
P^{\Omega^2 g}_{ab}-P^g_{ab} = -\nabla^g_a \Upsilon_b + \Upsilon_a\Upsilon_b -\frac12 g_{ab}|\Upsilon|_g^2\, .
\end{equation}
implies another set of useful variations
\begin{eqnarray}\label{Pnnvary}
\delta_{-2}P_{\hat n\hat n}\: \:&\stackrel\Sigma=&
-\hat n^a  (\nabla_{\hat n} \Upsilon_a)\, ,
\\
\label{gradPnnvary}
\qquad
\delta_{-2} \bar \nabla_a P_{\hat n\hat n}&\stackrel\Sigma=&
-2 \bar \Upsilon_a P_{\hat n\hat n}- \bar \nabla_a\big(
\hat n^b  (\nabla_{\hat n} \Upsilon_b)\big)\, ,
\\
\label{varyJdot}
\delta_{-3} \nabla_{\hat n} J \:\:&\stackrel\Sigma=& 
-\hat n^a \hat n^b \hat n^c\nabla_a \nabla_b \Upsilon_c
-\bar \Delta \Upsilon_{\hat n}
-H\big ((d-1)  \hat n^a \hat n^b \nabla_a  \Upsilon_b
-2\bar \nabla^a \bar \Upsilon_a\big)\, ,
\\\nonumber
&&
+\big((d-3) P_{\hat n\hat n} - \bar J+\tfrac32(d-1) H^2\big) \Upsilon_{\hat n}
-(d-3) \bar \nabla_{\bar \Upsilon}H\, .
\end{eqnarray}
The Schouten tensor variation also implies that
$$
\delta C_{abc} =- W_{abc\Upsilon}\, .
$$

\medskip

In addition we shall need 
 transformations of connection-coupled quantities.
The transformation of the boundary Yang--Mills current is
\begin{equation}\label{prairie_dog1}
\delta_{-2}\bj_a = (d-5) \bar F_{\bar\Upsilon a}\, . 
\end{equation}
An interesting  result is the following.
\begin{proposition}\label{walnut}
Let  $\bar \nabla$ be a connection on ${\mathcal V}\Sigma^{d-1}$. Then there is  a conformally  invariant operator
$$
\boxast:\Gamma(T^*\Sigma[2-\tfrac{d-1}2]\otimes \End{\mathcal V}\Sigma)
\to 
\Gamma(T^*\Sigma[-\tfrac{d-1}2]\otimes \End{\mathcal V}\Sigma)\, ,
$$
given by
$$
\boxast V_b
:
\stackrel g=
\bar\Delta V_b
-\tfrac4{d-1}\bar\nabla_b \bar \nabla^a V_a
-2\bar P^a{}_b  V_a
-\tfrac{d-3}2\bar J V_b \, .
$$
\end{proposition}

\begin{proof}
This is a connection coupling of an operator discovered by  Branson, Deser and Nepomechie~\cite{Branson0,DN}. As it is only quadratic order in the connection,  the additional coupling does not spoil its invariance.
\end{proof}

A useful fact closely linked to the above is that
\begin{align*}
\boxast_{7}\hh  V_c 
&\stackrel g{:=} 
\bar\Delta V_c
-\frac23\bar\nabla_c \bar \nabla^a V_a
-
\tfrac{
2(d-4)}3\bar P^a{}_c  V_a
-2\bar J V_c\, ,\nonumber\\[-5mm]
\\
\boxast_{9}\hh  V_c 
&\stackrel g{:=} 
\bar\Delta V_c
-\frac12\bar\nabla_c \bar \nabla^a V_a
-
\tfrac{
d-5}2\bar P^a{}_c  V_a
-3\bar J V_c\, ,\nonumber
\end{align*}
have  conformal variations
\begin{align}
\delta_{-3}\big(
\boxast_{7} V_c \big)
&=(d-7)\big( \bar \nabla_{\bar \Upsilon} V_c
-\tfrac 23 \bar \Upsilon^a \bar \nabla_c  V_a\big)\, ,\nonumber\\[-5mm]
\label{varybox7}\\
\delta_{-4}\big(
\boxast_{9} V_c \big)
&=(d-9)\big( \bar \nabla_{\bar \Upsilon} V_c
-\tfrac 12 \bar \Upsilon^a \bar \nabla_c  V_a\big)
\, ,\nonumber
\end{align}
when $\delta_{-1} V$ and $\delta_{-2}V$ are respectively set to zero. Note that $\boxast= \boxast_7$ when $d=7$ and $\boxast= \boxast_9$ when $d=9$.
Also note that, given a choice of metric,  we will use  the operator~$\boxast$ defined by the second display in Proposition~\ref{walnut} at weights where it is not invariant.
\noindent

\section{The Conformally Compact Yang--Mills Equation}\label{CCYME}

Let $(M,\cc,\sigma)$ be a conformally compact structure and $\nabla^A$ a gauge connection  
on a vector bundle ${\mathcal V}M$. 
Then the (source-free) Yang--Mills equation on the interior~$M_+$ is given by the divergence-free condition
\begin{equation}\label{ccYM}
\go^{ab}\nabla^{\go }_a F_{bc}^{\phantom{\go}} = 0 \, .
\end{equation}
Recall $\nabla^{\go }_a$ is the Levi-Civita connection of the conformally compact metric $\go$ twisted by the connection $A$ while $F$ is the $\End{\mathcal V}M$-valued curvature two-form of the latter.

Remembering that on $M_+$ the conformally compact metric  $\go=\bg/ \sigma^2$,
by re-expressing the Levi--Civita connection of $\go$ in terms of that of $\bg$, we may rewrite the left hand side of Equation~\nn{ccYM}  as $\sigma\bg^{ab}\big( \sigma\nabla_a F_{bc}-(d-4) (\nabla_a \sigma) F_{bc}\big)=\sigma j_c$.
Here $ j_c\in \Gamma(T^*M\otimes \End{\mathcal V}M[-1])$ is termed the {\it  conformally compact Yang--Mills current}, and is defined by
\begin{equation}\label{YMC}
{ j_c}[A]:\stackrel g{=} \sigma \nabla^a F_{ac}-(d-4) F_{nc}\, ,\qquad n :\stackrel g{=}\ext \sigma\, .
\end{equation}
Critically, because the connection~$A$ is defined on all of~$M$, we have that $j_c$ is well-defined on~$M$,  
rather than just its interior~$M_+$,  and is a functional of connections $\nabla^A$. Therefore the equation
\begin{equation}\label{YM}
{ j}_c[A]=0\, 
\end{equation}
naturally extends the  Yang--Mills Equation~\nn{ccYM} to  the boundary.
We  term this the {\it conformally compact Yang--Mills equation}. Clearly any solution to Equation~\nn{YM} also solves~\nn{ccYM}.

\medskip
The following result will be important when studying
solutions to the conformally compact Yang--Mills equations by using distinguished coordinate systems---see Section~\ref{connexp}.

\begin{lemma}
Let $(M,\cc,\sigma)$ be a conformally compact structure with  a smooth connection~$\nabla^A$. Then (given $g\in \cc$)
\begin{equation}\label{BI}
\sigma \nabla^a j_{a}[A]\stackrel g=(d-3)\, j_n[A]\, . 
\end{equation}
\end{lemma}

\begin{proof}
Throughout this proof we compute in some scale $g\in\cc$.
The result of the lemma is a direct consequence of the Bianchi/Noether identity 
$$
\nabla^a \nabla^b F_{ab}=0\, .
$$
In detail, observe that using this and Equation~\nn{YMC},
$$
\nabla^b j_b[A]= n^b \nabla^a F_{ab}
-(d-4) \nabla^b F_{nb}=-(d-3)\nabla^b F_{nb} \, .
$$ 
The last equation holds because $\nabla_{[a} n_{b]}=0$. 
Multiplying the above by $\sigma$ and using again the definition of $j_a[A]$ we have
$$
\sigma \nabla^a j_a[A]=(d-3) \sigma n^b \nabla^a F_{ab} =(d-3)n^a(j_a[A]+(d-4) F_{na})\, .
$$
Clearly $F_{nn}=0$.
\end{proof}

\subsection{Maxwell System on Hyperbolic Space}\label{Maxwell}

An instructive model problem is that of Maxwell's equations on the Poincar\'e ball~${\mathbb B}^d$. 
In this situation a Fourier analysis allows a detailed description of solutions. Such a study was originally carried out in early work on the Anti de Sitter--conformal field theory correspondence in the physics literature~\cite{Gubser}. 
Here we give a more careful mathematical account focussing on the Dirichlet-to-Neumann map. Our starting point is the 
  upper-half space  model ${\mathbb H}^d$ for the Poincar\'e ball. In particular we study solutions which make (global) sense  on the one-point compactification of ${\mathbb H}:=\{ (y,\vec x)\hh|\hh  y> 0, \vec x\in {\mathbb R}^{d-1}\}$. 
The (singular) metric in this model is given by 
$$
\go=\frac{\ext{}y^2+\ext{}\vec x^{\hh 2}}{y^2}=:\frac g{y^2}\, .
$$
In the above formula, and in what follows, we  use standard vector calculus notations for vector-valued quantities on ${\mathbb R}^{d-1}$.
Also, $g$ is the upper half plane Euclidean  metric.
Away from the  $y\to \infty$ ``north pole'' of the Poincar\'e ball,  the defining density $\sigma$ for this conformally compact geometry is given by $\sigma=[g;y]$.

\medskip
In its potential form, the Maxwell system is
$$
\mathring{\delta} \ext{} A = 0\, ,
$$
where $A\in \Gamma(T^*{\mathbb H})$ and $\mathring{\delta}$ is the divergence operator of  the singular metric $\go$;  this operator extends smoothly to the boundary. In terms of the compactified metric $g$, the above display becomes
\begin{equation}\label{y-ified}
y^{d-3} \delta \big(y^{4-d} \ext A\big)=0\, ,
\end{equation}
where $\delta$ is now the divergence operator of $g$.
Indeed, in the interior ${\mathbb H}$, the left hand side is $y\delta \ext A -(d-4) \iota_{\frac\partial{\partial y}} \ext A$ which extends to  the boundary $\partial {\mathbb H}$. Calling $F:=\ext A$, we have
$$
y\delta F -(d-4) \iota_{\frac\partial{\partial y}} F=0\, .
$$
The above is precisely Equation~\nn{YM}.
Note that the vector field $\frac\partial{\partial y}$ is defined everywhere on the Poincar\'e ball 
as well as  all its boundary points except the north pole.

Since we are considering potentials, we must deal with gauge degeneracies. In particular,
given 
any solution $A$ to the Maxwell system, then $A+ \ext \alpha$, where $\alpha$ is any smooth function, is also a solution. In particular let us suppose that a solution $A$ obeys the Dirichlet condition 
$$
A|_\Sigma=\bar A \in \Gamma\big(T^*{\mathbb H}\big|_{\partial {\mathbb H}}\big)\, .
$$ 
Henceforth, we shall study solutions whose boundary data is compactly supported away from the north pole
and view  $(y,\vec x)$ as coordinates on  a patch of the Poincar\'e ball that excludes this point.  
Then, given a solution $A$  in this coordinate patch, and setting
$$
\alpha=-y \beta\, ,
$$
where the function $\beta\stackrel {\partial {\mathbb H}}  = \iota_{\frac\partial{\partial y}} A 
$, it follows that
\begin{equation}\label{bbc}
\iota_{\frac\partial{\partial y}} (A+\ext{}\alpha)\stackrel {\partial {\mathbb H}}  =  0\, .
\end{equation}
In view of our compact support assumption, the above condition can be extended to hold  along the boundary  $\partial {\mathbb B}^d$ of the Poincar\'e ball.
Recalling that the orthogonal projection of~$T^*{\mathbb H}\big|_{\partial {\mathbb H}}$ with respect to the normal vector~$\frac{\partial}{\partial y}\big|_{\partial {\mathbb H}}$ is isomorphic to $T^*\partial {\mathbb H}$,   we set ourselves the problem of solving
$$
\left\{
\begin{array}{c}
\mathring{\delta} \ext{} A = 0\, ,\\[1mm]
A|_\Sigma = \bar A \in \Gamma(T^*\partial {\mathbb H})\, ,
\end{array}
\right.
$$
where the compactly supported boundary one-form $\bar A$ constitutes the Dirichlet data.

\smallskip

To solve this problem, it is useful to impose a ``temporal'' 
type gauge condition, meaning given a solution $A$, we search for $\alpha\in C^\infty {\mathbb H}$ such that the condition 
$$
\iota_{\frac\partial{\partial y}} (A+\ext{}\alpha)  =  0
$$
holds globally (rather than just along the boundary).
This condition  makes sense for any chart excluding the north pole.
We have already established in Equation~\nn{bbc}
that $\iota_{\frac\partial{\partial y}} (A+\ext{}\alpha)$
vanishes on all of ${\partial {\mathbb H}}$, and thus also at the north pole,
so the above requirement can be made globally well-defined.
Noting that $\iota_{\frac\partial{\partial y}}\ext{}\alpha  =
\frac{\partial\alpha}{\partial y}$, we see that this can be solved by setting $\alpha = -\int^y \iota_{\frac\partial{\partial y}} A$, so long as $\iota_{\frac\partial{\partial y}} A$ is $L^1$ integrable in $y$.
Thus we henceforth study potentials
\begin{equation}\label{Coulomb}
A=\vec A(y,\vec x)\cdot \ext{}\vec x \, ,
\end{equation}
where our boundary Dirichlet data is the one-form $\bar A=\vec A(0,\vec x)\cdot \ext{}\vec x$.  
Let us now assume that this is compactly supported and expressed as an inverse Fourier transform ${\mathcal F}^{-1}$ over ${\mathbb R}^{d-1}$ 
$$
\vec A(0,\vec x)={\mathcal F}^{-1}\big[\vec a\big(0,\vec k\big)\big]\, ,
$$
so that  
$$\Delta^{\!{\mathbb R}^{d-1}}{\mathcal F}^{-1}\big[\vec a\big(0,\vec k\big)\big]={\mathcal F}^{-1}\big[-\vec k^2\hh\vec a\big(0,\vec k\big)\big]\: \mbox{ and }\:
\vec \nabla {\mathcal F}^{-1}\big[\vec a\big(0,\vec k\big)\big]={\mathcal F}^{-1}\big[ i\vec k\otimes \vec a\big(0,\vec k\big)\big]\, .
$$
Then we look for a solution to the Maxwell problem of the form 
$$
\vec A(y,\vec x)={\mathcal F}^{-1}\big[\vec a\big(y,\vec k\hh \big)\big]\, .
$$
Note that solutions to the Maxwell system without performing a Fourier transform are studied in~\cite{Witten}. Indeed the inverse Fourier transform of the below Bessel-type solutions can be performed explicitly, see for example~\cite{Deser}.
Decomposing Equation~\nn{y-ified} into its parts normal and parallel to $\frac\partial{\partial y}$, gives a pair of equations for the Fourier coefficients~$\vec a(y,\vec k)$ 
\begin{align*}
0&=\ y \Big(\frac{\ext{}^2 \vec a}{\ext{}y^2}-\vec k^{\hh 2}\hh \vec a+ \vec k \hh\hh \vec k.\vec a\Big) -(d-4)\hh \frac{\ext{} \vec a}{\ext{}y}\, ,\\
0&=\ y\frac{\ext{}(\vec k. \vec a)}{\ext{}y}\, .
\end{align*}
The second equation implies that $i\vec k.\vec a = \varphi$ where $\varphi$ is some $y$-independent function of~$\vec k$.
However, note there is still the freedom to make further gauge choices, by searching for~$\beta$ independent of $y$ such that
$$
\vec \nabla.(\vec A + \vec \nabla \beta)=0\, .
$$
For that we must consider $\Delta^{{\mathbb R}^{d-1}} \beta
=-{\mathcal F}^{-1}[\varphi]$. Solving this Poisson equation, 
we may
proceed assuming that $\vec k.\vec a=0$. In other words, the solution~$\vec A$ is transverse, meaning~$\vec \nabla . \vec A=0$. It then remains to solve
\begin{equation}\label{bessellike}
0= y \Big(\frac{\ext{}^2 \vec a}{\ext{}y^2}-\vec k^{\hh 2}\hh \vec a \Big) -(d-4)\hh \frac{\ext{} \vec a}{\ext{}y}\, .
\end{equation}
The above-displayed ordinary differential equation is essentially Bessel's equation, so the results given below follow from classical special function theory (see, for example~\cite{Watson}).

The cases $d$ even and $d$ odd  exhibit differing behaviors.
We start with the former simpler case. For that we wish to solve  Equation~\nn{bessellike} for smooth $\vec a(y,\vec k\hh)$ as function of~$y\geq 0$, and subject either to a 
 Dirichlet or Neumann boundary condition:
 $$
 \vec a(0,\vec k)=\vec b_{\vec k}\, \qquad\mbox{ or }\qquad \frac1{(d-3)!} \, \frac{\ext{}^{n-2} \vec a}{\ext{}y^{d-3}}\Big|_{y=0}=\vec e_{\vec k}\, .
 $$
 These  boundary behaviors follow  from a standard characteristic exponent analysis by performed by rewriting  Equation~\nn{bessellike} as $\big(X(X-d+3)-k^2 y^2\big)\vec a=0$, where $X=y\frac{\ext{}}{\ext{}y}$ is the Euler vector field. 
The formal asymptotic solutions about $y=0$ are uniquely determined for each of the above boundary behaviors upon assuming that their expansions are, respectively, even or odd functions of $y$. The first few terms are 
$$
\vec a_{\rm D}(y)\stackrel{y\to 0}\sim \left(1 +\tfrac{1}{d-1} \tfrac{y^2\hh \overline \Delta}{2}  + \tfrac{1}{2!(d-1)(d-3)}\left(\tfrac{y^2\hh \overline \Delta}{2}\right)^{\! 2}
 + \tfrac{1}{3!(d-1)(d-3)(d-5)} \left(\tfrac{y^2\hh \overline \Delta}{2}\right)^{\! 3}  +\cdots\right) \vec b_{\vec k} $$
and
$$
\vec a_{\rm N}(y)\stackrel{y\to 0}\sim
y^{d-3} 
\left(1 -\tfrac{1}{d-1} \tfrac{y^2\hh \overline \Delta}{2}  + \tfrac{1}{2!(d-1)(d+1)}\left(\tfrac{y^2\hh \overline \Delta}{2}\right)^{\!2}\!
 - \tfrac{1}{3!(d-1)(d+1)(d+3)} \left(\tfrac{y^2\hh\overline \Delta}{2}\right)^{\!3}  +\cdots\right) \vec e_{\vec k}\, ,
$$
where we have used the suggestive notations $k:=|\vec k|$ and $ \overline \Delta :=-k^2$. (The above expressions are examples of solution generating operators 
of the type studied  in wider generality in~\cite{GW}.)
Both the above expansions can be encoded by generating functions: 
$$f_{\rm D}(y):= P\big((ky)^2\big)\cosh (ky)-ky\,  Q\big((ky)^2\big)\sinh (ky)$$
and
$$
 f_{\rm  N}(y):=\frac{(-1)^{\frac{n-3}2}(n-2)!! (n-4)!!}{k^{n-2}}\Big[P\big((ky)^2\big)\sinh (ky)-ky\, Q\big((ky)^2\big)\cosh (ky)\big)\Big]\, .$$
In the above $P$ and $Q$ are polynomials of degree $\lfloor (n-2)/4 \rfloor$ and $\lfloor (n-4)/4 \rfloor$, respectively, of the form
$$
P(z) = 1 + \tfrac{1}{2!}\tfrac{n-5}{n-4} \, z + \cdots\, ,\qquad
Q(z) =  1 + \tfrac{1}{3!}\tfrac{n-7}{n-4} \, z + \cdots\, .
$$
The higher order terms are unimportant in the current context, but generating expressions for these exist as well. For example, when $d=9$, we have a general asymptotic solution given by
\begin{multline*}
\vec a(y,\vec k)\stackrel{y\to 0}\sim\Big[(1+\frac13 (ky)^2)\cosh (ky)-(ky)\sinh (ky)\Big]\vec b_{\vec k}\\+
\frac{45}{k^5}\Big[\big((1+\frac13 (ky)^2)\sinh (ky)-ky\cosh (ky)\big)\Big]\vec e_{\vec k}\, .
\end{multline*}
The above is an asymptotic solution around $y\to 0$ for {\it any} choice of initial datum $\vec e_{\vec k}$ and $\vec b_{\vec k}$. However, observe that neither the coefficient function of $\vec b$ nor $\vec e$ is well-defined in the limit $y\to \infty$, but for 
the choice $\vec e_{\vec k}=-\frac1{45} k^{5}\hh \vec b_{\vec k}$
the exponentially growing behavior cancels.
In general, 
because the underlying problem is linear, one can search for a global solution~$\vec a_{\rm G}(y,\vec k)$  to the Dirichlet problem that is a smooth function of  $y\in {\mathbb R}^+$, 
 given by
$$
\vec a_{\rm G}(y,\vec k)=f_{\rm D}(y) \hh \vec b_{\vec k} + f_{\rm N}(y)\hh  \vec e_{\vec k}(\vec b_{\vec k})\, ,
$$
where the Neumann data $\vec e_{\vec k}$ is determined in terms of the Dirichlet data $\vec b_{\vec k}$. The  map $\vec b_{\vec k} \mapsto \vec e_{\vec k}\big(\vec b_{\vec k}\big)$ is termed a {\it Dirichlet-to-Neumann} map.
For this particular problem, the Dirichlet-to-Neumann map is given by the linear relation 
$$
\vec e_{\vec k}\big(\vec b_{\vec k}\big)=\frac{(-1)^{\frac{d-2}2}}{(d-3)!! (d-5)!!}{} (-\hh \overline \Delta)^{\frac{d-3}2}\, \vec b_{\vec k}\, ,
$$
which gives the global solution 
$$\vec a_{\rm G}(y) = \big[P\big((ky)^2\big)+ky \, Q\big((ky)^2\big)\big]\exp(-ky)\vec b_{\vec k}\, .$$

Now we consider the more subtle $d$ odd  case. For that we again use that Equation~\nn{bessellike} is essentially Bessel's equation, and thus from classical theory~\cite{Watson} has a pair of solutions
$$
\left\{
\begin{array}{l}
\vec a_{\rm G}(y,\vec k) = \frac{2}{(\frac{d-5}2)!}\big(\frac{ky}2\big)^{\frac {d-3}2} K_{\frac {d-3}2}(ky)\hh  \vec b_{\vec k}
\\[3mm] 
\phantom{f(y) } 
\!\!\! \stackrel{y\to 0}\sim \left(1
 +\tfrac{1}{d-5} \tfrac{y^2\hh \overline \Delta}{2}  + \tfrac{1}{2!(d-5)(d-7)}\left(\tfrac{y^2\hh \overline \Delta}{2}\right)^{\! 2}
+\cdots
+{\mathcal O}(y^{d-3})  +{\mathcal O}(y^{d-3}\log y) \right)\vec b_{\vec k}\, ,
\\[6mm]
\vec a_{\rm U}(y,\vec k) = 
  \big(\frac {d-3}2\big)!\big(\frac{2y}k\big)^{\frac {d-3}2} I_{\frac {d-3}2}(ky) \hh
\vec e_{\vec k} \\[3mm]
\phantom{f(y) } 
\!\!\! 
\stackrel{y\to 0}\sim
y^{d-3} 
\left(1 -\tfrac{1}{d-1} \tfrac{y^2\hh \overline \Delta}{2}  + \tfrac{1}{2!(d-1)(d+1)}\left(\tfrac{y^2\hh \overline \Delta}{2}\right)^{\!2}+\cdots \right) \vec e_{\vec k}\, .
\end{array}
\right.
$$
Any linear combination of the above two solutions gives a polyhomogeneous asymptotic solution around $y\to 0$. However,
only the Bessel function of type $K$ is finite as $y\to \infty$ so just the first of the two  above solutions 
is global. 
The terms of order $y^{d-3}$ and $y^{d-3}\log y$ in the expansion of the global solution $\vec a_{\rm G}$  are known and are given by
$$
\frac{(-1)^{\frac {d-5}2}\, y^{d-3}}{2^{d-4}\big(\tfrac {d-5}2\big)! 
\big(\tfrac {d-3}2 \big) !} \, 
\Big(
\log y
+\log k +\log \frac12 + \gamma
-\frac12\Big[1 + \frac 12 + \frac 13 + \cdots +\frac 1{\hh\frac {d-3}2\hh}\Big]
\Big)\, k^{d-3} \, \vec b_{\vec k}\, .
$$
The coefficient $$
{\sf P}:=-\frac{ \hh \overline \Delta^{\frac {d-3}2}}{2^{d-4}\big(\tfrac {d-5}2\big)! 
\big(\tfrac {d-3}2 \big) !}
$$
appearing next to the $y^{d-3} \log y$ term may, upon Fourier transformation, be viewed as the analog of the conformally invariant Laplacian  powers of Graham, Jennes, Mason, and Sparling (GJMS)~\cite{GJMS} that were understood in a scattering  context by Graham and Zworski~\cite{GZ}. Its appearance as the coefficient of the first log term is universally determined by the intrinsic boundary geometry, even in the case of the fully non-linear system on general conformally compact manifolds. Rearranging the global solution to read
\begin{multline*}
a_{\rm G}(y,\vec k)=
\Bigg(1
 +\tfrac{1}{d-1} \tfrac{y^2\hh \overline \Delta}{2}  + \tfrac{1}{2!(d-1)(d-3)}\left(\tfrac{y^2\hh \overline \Delta}{2}\right)^{\! 2}
+\cdots
+ \tfrac{1}{2^{\frac{d-5}2}\big(\big(\frac{d-5}2\big)!\big)^2}\left(\tfrac{y^{2}\hh 
\overline \Delta}{2}\right)^{\! \frac{d-5}2}\Bigg)\vec b_{\vec k}
\\[1mm]
\hspace{2cm}
y^{d-3}\log y \hh {\sf P}\hh\vec b_{\vec k}
+y^{d-3}\
\vec e_{\vec k }(\vec b_{\vec k})
+{\mathcal O}(y^{d-2}\log y)+{\mathcal O}(y^{d-2})
\, ,
\end{multline*}
the terms multiplying $y^{d-3}$,
$$
\vec e_{\vec k }(\vec b_{\vec k})= \Big(
\log \frac k2  + \gamma
-\frac12\Big[1 + \frac 12 + \frac 13 + \cdots +\frac 1{\hh\frac {d-3}2\hh}\Big]
\Big)\, {\sf P}\, \vec b_{\vec k}\, ,
$$
capture
 the global nature of the solution
 and constitute the Dirichlet-to-Neumann map.
It is interesting to study whether there exist boundary operators that 
extract this non-local data in both dimension parities.

\section{Boundary Problems and Asymptotics}\label{Apollosrevenge}

We now study the boundary asymptotics of  solutions $\nabla^A$ to the conformally compact Yang--Mills Equation~\nn{YM} on a conformally compact manifold $(M,\cc,\sigma)$. The allowed boundary behaviors of formal asymptotic solutions can be determined by a non-linear analog of a characteristic exponent analysis. This  will lead to the magnetic and electric problems discussed earlier.

 To study characteristic exponents, we consider the curvature ansatz
  \begin {equation}\label{ansatz}F^A=\sigma^\alpha \big(G + {\mathcal O}(\sigma)\big)\, ,\end{equation}
 where $0\neq G\in\Gamma(\wedge^2 T^*M\otimes \End{\mathcal V}M[-\alpha]) $ and $\alpha$ is some real number. Note that $\Gamma(BM[w])\ni X={\mathcal O}(\sigma^k)$ means $X=\sigma^k Y$ where $Y\in \Gamma(BM[w-k])$.
 Any solution for a curvature  must obey both a Bianchi identity and the conformally compact Yang--Mills equation
$$
\nabla_{[a} F_{bc]}=0 = \sigma \nabla^a F_{ab}-(d-4)F_{nb}\, ,
$$
so that Ansatz~\nn{ansatz} yields 
$$
\alpha n_{[a} G_{bc]}= {\mathcal O}(\sigma) = (\alpha-d+4) G_{nb}\, .
$$
When $d>4$, this implies two---at least at the level of formal asymptotics---boundary problems:\begin{eqnarray}
\label{magnetic} 
\alpha=0\, &\mbox{ and }&G_{\hat n b}\stackrel\Sigma=0\, ,\\
\label{electric} \alpha=d-4\, 
&\mbox{ and }& \hat n_{[a} G_{bc] }\stackrel\Sigma=0\, .
\end{eqnarray} 
Here Equation~(\ref{magnetic}) 
is the {\it magnetic problem} dealt with in Section~\ref{MP}, while~(\ref{electric}) is the {\it electric problem} of Section~\ref{EP}.

\subsection{Magnetic Problem}
\label{MP}

Let us consider connections on ${\mathcal V}M$ that extend smoothly to the boundary as connections. To that end, here and throughout the following, we assume that the manifold $M$ is equipped with a vector bundle $\pi : {\mathcal V}M\to M$ and so $\pi^{-1}\Sigma={\mathcal V}M|_\Sigma$ defines a vector bundle ${\mathcal V}\Sigma$ over the boundary $\Sigma$. Connections on~${\mathcal V}\Sigma$ are denoted by $\bar A$. 
Observe that when a connection $\nabla^A$ on ${\mathcal V}M$ extends smoothly to the boundary, 
so too does its curvature.
When $\bar \nabla^{\bar A}$ 
has non-vanishing curvature and 
is the boundary condition for a solution,
the boundary problem has 
  the  $\alpha=0$ behavior of  Equation~(\ref{magnetic}) above.
We now prove Theorem~\ref{recurse}
concerning the existence of solutions to the magnetic Problem~\ref{yeswehavethem}.

\begin{proof}[Proof of Theorem~\ref{recurse}]
Let us, as usual, choose some $g\in \cc$  and trivialize density bundles accordingly.
First we will establish  that  if 
\begin{equation}\label{nameme}
 j_a[A]\stackrel g=\sigma \nabla^b F_{ba} -(d-4) F_{na}=\sigma^\ell  k_a\, ,
\end{equation}
for any $\ell\in {\mathbb Z}_{\geq 0}$ when $d=3$, and any $d-3\neq \ell\in {\mathbb Z}_{\geq 0}$ for $d\geq 4$,
 then
$$\hat { n}^a { k}_a \stackrel\Sigma=0\, ,$$
where $\hat n = (n/|n|)|_\Sigma$.

\smallskip

The case $\ell=0$ is obvious from the definition of the Yang--Mills current $j$ in Equation~\nn{nameme}. Now, contracting Equation~\nn{nameme} with  $n$  gives
$$
 \sigma n^a\nabla^b F_{ba}= \sigma^\ell  k_n\, .
$$
Using again Equation~\nn{nameme} gives 
\begin{multline*}
(d-4)\hh\sigma n^a\nabla^b F_{ba}=
-(d-4)\hh\sigma \nabla^a F_{na}\\
={\sigma}\nabla^a\Big(
\sigma^\ell k_a -\sigma \nabla^b F_{ba}
 \Big)
=\ell\sigma^\ell k_n-\sigma n^a \nabla^b F_{ba}  + {\mathcal O}(\sigma^{\ell+1})
\, .
\end{multline*}
The first equality used that $\nabla_{[a}n_{b]}=0$
and the third relied on the Noether identity $\nabla^{a}\nabla^{b} F_{ba}=0$.
We now have the system of equations $$
\begin{pmatrix}\label{system}
 \ell& 3-d\\
1&-1
\end{pmatrix}
\begin{pmatrix}
\sigma^\ell k_n\\
\sigma n^a \nabla^b F_{ba}
\end{pmatrix}
={\mathcal O}(\sigma^{\ell+1})\, .
$$
Thus, whenever $\ell\neq d-3$, we conclude that 
$
\sigma^\ell k_n={\mathcal O}(\sigma^{\ell+1})\, ,
$
so the claim follows save for the special case $d=3$, which already holds from the $\ell=0$ analysis above.

\medskip

For $d\geq 4$ we now recursively construct a connection $A$ that satisfies
$$
j[A]=\sigma^{d-4} k\, ,
$$
for some $k\in \Gamma(T^*M[5-d]\otimes \End{\mathcal V}M)$.
For that, suppose $A_0$ is any connection  that obeys
$$\nabla^{A_0}_v\stackrel \Sigma  = \nabla_v^{{\bar A}}\, ,$$
for all $v\in \Gamma(T\Sigma)$.
Clearly such a connection always exists. Then, by its definition, the Yang--Mills current obeys
$$
j_a[A_0]=-(d-4) F_{na}+{\mathcal O}(\sigma)\, . 
$$
If $d=4$ we set $A=A_0$ and are done, else we study an improved connection
$$
\nabla^{A_1}=\nabla^{A_0} + \sigma a_1\, ,
$$
where the section $a_1$ of
$T^*M\otimes\End{\mathcal V}M$ is to be determined.
For that first note the general curvature identity for connections $\nabla^{A'}$ and $\nabla^A$ which differ by a one-form valued endomorphism $\sigma^\ell a\hh$:
$$
F_{ab}[A']=F_{ab}[A]+2 \nabla^A_{[a}(\sigma^\ell a^{\phantom{A}}_{b]})+\sigma^{2\ell} [a_a,a_b]\, .$$
Thus, using Equation~\nn{nameme} and the above display,
$$
j^a[A_1]=j^a[A_0]-(d-4) (n^2 a^{a}_1-n^a n_b a^b_1) +{\mathcal O}(\sigma)\, .
$$
From the $\ell=0$ version of the claim above we have 
$
j_a[A_0]=k_a
$
with $\hat n^a k_a\stackrel\Sigma=0$ so that~$k\stackrel\Sigma=k^\top$, where for any vector $v\in \Gamma(TM)$ we denote $v^\top:=(v-\hat n\hh  v_{\hat n})|_\Sigma$.
Thus, for $d\neq 4$, 
we may impose
$$
a_1^\top \stackrel\Sigma = \frac1{(d-4)|n|^2_g}\,  j[A_0] \, ,
$$
which achieves $j[A_1]={\mathcal O}(\sigma)$ for $\nabla^{A_1}=\nabla^{A_0}+\sigma a_1$.

Having seeded the recursive solution, we now suppose that 
$A_{\ell-1}$ obeys $$j[A_{\ell-1}]=\sigma^{\ell-1} k\, .$$
Again $k\in \Gamma(T^*M\otimes \End{\mathcal V}M)$ and by our above claim obeys $k\stackrel\Sigma= k^\top$.
Also, a similar computation to above shows that
\begin{equation}\label{diffo}
j_a[A+\sigma^\ell a] = j_a[A]   + \sigma^{\ell-1} \ell (\ell-d+3) (n^2 a_a - n_a a_n)+{\mathcal O}(\sigma^\ell)\, .
\end{equation}
Hence, if $0<\ell\leq d-4$, we can achieve
$j[A_\ell]={\mathcal O}(\sigma^\ell)$ by setting 
$\nabla^{A_\ell}=\nabla^{A_{\ell-1}}+\sigma^{\ell} a_\ell$,
where
\begin{equation}\label{solve}
a_\ell^\top\stackrel\Sigma = \frac1{\ell(d-3-\ell)|n|^2_g} \, j[A_{\ell-1}] \, .
\end{equation}

Finally, for the case $d=3$, the same argument as above holds, except there is no restriction on when  Equation~\nn{diffo} can be solved for the improvement term~$\sigma^\ell a$.

\end{proof}

\medskip
Our next result shows that any non-uniqueness of the asymptotic solution to~\ref{yeswehavethem} is modulo gauge equivalence.

\begin{theorem}\label{gaugeeq}
Let $\nabla^A$ and $\nabla^{A'}$ be asymptotically Yang--Mills connections with boundary data ${\bar A}$. Then there exists an invertible element $U\in \Gamma(\End{\mathcal V}M)$ such that 
$$
U|_\Sigma = \operatorname{Id}\, ,
$$
and
\begin{equation}\label{gaugey}
\nabla^{A'}=U^{-1} \circ \nabla^A \circ U +{\mathcal O}(\sigma^{d-3})\, .
\end{equation}

\end{theorem}

\begin{proof}
We proceed by induction and  pick a scale $g$ to trivialize density bundles. Suppose that
$$
\nabla^{A'}=\nabla^{A}+\sigma^\ell a
$$
for some $a\in \Gamma(T^*M\otimes\End{\mathcal V}M)$ and some integer
$\ell > 0$. We want to find $U_\ell$ such that
$$U_\ell\circ\nabla^{A'}= \nabla^A \circ U_\ell +{\mathcal O}(\sigma^{\ell+1})\, .
$$
Because $\nabla^{A}$ and $\nabla^{A'}$ 
share the same boundary data, 
it follows that 
$$
\nabla^{A'}=\nabla^{A} + n \beta_0 + {\mathcal O}(\sigma)\, ,
$$
for some $\beta_0 \in \Gamma(\End{\mathcal V}M)$ and, as usual,  $n:=\ext \sigma$ in the scale $g$.
Thus,
$$
\nabla^{A}\circ (\operatorname{Id} + \sigma \beta_0) =
(\operatorname{Id}+\sigma \beta_0 )\circ  (\nabla^A + n \beta_0) +  {\mathcal O}(\sigma)\, .
$$
Thus choosing $U_0= \operatorname{Id} + \sigma \beta_0$
in a suitable collar neighborhood  of $\Sigma$ establishes the $\ell=1$  base case. 

For $0<\ell\leq d-4$,
 because both $\nabla^A$ and $\nabla^{A'}$ are 
asymptotically Yang--Mills connections, 
it follows from Equation~\nn{diffo}  that
$$\sigma^{\ell-1} \ell(\ell-d+3) (n^2 a-n a_n)={\mathcal O}(\sigma^\ell)$$
so that
$$
\nabla^{A'} =\nabla^{A} + \sigma^\ell n \beta_{\ell}+{\mathcal O}(\sigma^{\ell+1})\, ,
$$
for some $\beta_\ell \in \Gamma(\End{\mathcal V}M)$.
A similar calculation to the base case shows that $$U_\ell=\operatorname{Id} + \frac{\sigma^{\ell+1}\beta_\ell}{\ell+1}\,  $$
gives the required gauge transformation.
\end{proof}

An important corollary of the above theorem concerns the uniqueness of $\bar k:=k|_\Sigma$ as given in Equation~\nn{obst}.  
\begin{corollary}\label{forgotten-soap}
Let $d\geq 5$, $(M,\cc,\sigma)$ be a conformally 
compact structure, and $\bar A$ be a vector bundle  connection on $\Sigma$. 
Then there is a canonical map
$$
(M,\cc,\sigma,{\bar A})
\mapsto \bar k
 \in
\Gamma(T^*\Sigma\otimes\End{\mathcal V}\Sigma[3-d]) \, .
$$
\end{corollary}

\begin{proof}
According to Theorem~\ref{gaugeeq}, the data $(M,\cc,\sigma,{\bar A})$
determines a conformally compact Yang--Mills asymptotic connection $A$ up to gauge transformations as in Equation~\nn{gaugey}.
From Equation~\nn{diffo} we have that
$$j[U^{-1}\circ A\circ U+{\mathcal O}(\sigma^{d-3})]=j[U^{-1}\circ A\circ U]+{\mathcal O}(\sigma^{d-3})\, .$$
The Yang--Mills current (as constructed in
Equation~\nn{YMC}) is clearly gauge covariant, meaning 
$$
j[U^{-1}\circ A\circ U]=U^{-1}\circ j[A]\circ U \, .
$$
But $j[A]=\sigma^{d-4} k$ and $U|_\Sigma=\operatorname{Id}$, so the result follows.
%
%
%
\end{proof}
When $\bar k\neq 0$, the boundary regularity of solutions to the conformally compact Yang--Mills equation of magnetic type must be relaxed to include log behavior. Details are discussed in Section~\ref{lets_get_loggy}.
The image $\bar k$ of the above map  is the obstruction current, so-called because   it is the obstruction to solving  smoothly to the boundary the  conformally compact Yang--Mills equation.

\medskip

Given an asymptotically Yang--Mills connection $\nabla^A$, the Dirichlet boundary datum
$$
\bar B=F^A\big|_\Sigma\in \Gamma(\wedge^2 \Sigma\otimes \End{\mathcal V}\Sigma)
$$
necessarily obeys
$
\nabla^{{\bar A}}_{[a} \bar B^{\phantom A}_{bc]}=0
$.
Indeed $\bar B$ 
may be viewed
as the boundary values of a non-abelian  ``magnetic field''. 
There exists a formal Neumann analog of this boundary problem with ``electric field'' boundary data; this is discussed in Section~\ref{EP}. First, however, we deal with asymptotics beyond order~$\sigma^{d-4}$.

\subsection{Log Solutions}\label{lets_get_loggy}

Given an asymptotically Yang--Mills connection $\nabla^A$ such that
$$
j[A]= \sigma^{d-4}k\, ,
$$
where $k\neq 0$, it is  possible to look for higher order asymptotic solutions by relaxing the boundary regularity to allow for a polyhomogeneous boundary expansion. Specifically we wish to solve the following problem.
\begin{problem}\label{logprob}
Let $\nabla^A$ be an asymptotically Yang--Mills connection and 
let a true scale $0<\tau\in \Gamma(\ce M[1])$ be given. Find an improved connection  
$$
\nabla^{A'} = \nabla^A  + \sigma^{d-3}\big(K\log(\sigma/\tau)\, 
+L \big)\, ,$$
on $M^+=M\setminus \Sigma$,
such that
$$j[A']={\mathcal O}\big(\sigma^{d-3}\log(\sigma/\tau)\big)+{\mathcal O}\big(\sigma^{d-3}\big)\, .$$
Here the smooth endomorphism-valued one-forms $K,L\in \Gamma(T^*M\otimes \End{\mathcal V}M[3-d])$,
and ${\mathcal O}\big(\sigma^{d-3
}\log(\sigma/\tau)\big)
$ denotes terms of the form $\sigma^{d-3
}\log(\sigma/\tau) G$ for some smooth $G\in \Gamma(\wedge^2M\otimes \End{\mathcal V}M[3-d])$. 
\end{problem}

\begin{remark}
The ratio of the weight one densities
 $\sigma$ and $\tau$ appearing in the above problem is a defining function for the boundary $\Sigma$. Hence the logarithm of $\sigma/\tau$ is well-defined in the interior. Note that there does exist a well-defined notion of a log-density such as $\log \sigma$ (see~\cite{GW}), but this will not be needed here. 
\hfill$\blacklozenge$
\end{remark}

Using the method of proof for the following solution to Problem~\ref{logprob}, it is also possible to solve a more general  problem concerning higher asymptotics involving products of logarithms; but here we content ourselves with the leading log asymptotics.

\begin{proposition}\label{logsol}
Let $d\geq 4$ and  $\nabla^A$ be an asymptotically Yang--Mills connection such that~$j[A]= \sigma^{d-4}k$. Then
Problem~\ref{logprob} is solved by setting
\begin{equation}\label{OK}
K\stackrel \Sigma= \frac1{d-3}\,  \bar k\, ,
\end{equation}
while $L$ is left undetermined.
\end{proposition}

\begin{proof}
In the proof of Theorem~\ref{recurse} we established that $k$ as determined by Equation~\nn{obst} obeys
 $$\hat n^a k_a\stackrel\Sigma=0\, .$$
We begin by comparing the curvatures of the connection $\nabla^A$ and that of the improved connection  $\nabla^{A'}$ (we also denote $\nabla:=\nabla^A$ for brevity and choose some scale $g$):
\begin{multline*}
F_{ab}\Big[A+\sigma^{d-3}\big(K\log(\sigma/\tau)\, 
+L \big)\Big]-F_{ab}[A]\\=
2\nabla_{[a}\big(\sigma^{d-3}(K_{b]}\log(\sigma/\tau)\, 
+L_{b]} )\big)
+\sigma^{2d-6}\, \Big[K_{a}\log(\sigma/\tau)+L_{a},
K_{b}\log(\sigma/\tau)+L_{b}\Big]\\
=
2n_{[a}  K_{b]} \sigma^{d-4}
+2 (d-3) n_{[a} K_{b]} \sigma^{d-4} \log(\sigma/\tau)
+2 (d-3) n_{[a} L_{b]} \sigma^{d-4}
\\
+{\mathcal O}\big(\sigma^{d-3
}\log(\sigma/\tau)\big)
+{\mathcal O}(\sigma^{d-3})
\, .
\end{multline*}
Whence (calling $\nabla'=\nabla^{A'}$ and similarly for its curvature) \begin{multline*}
\sigma \nabla'^a F'_{ab}=
\sigma \nabla^a F_{ab}+
2(2d-7)n^a n_{[a}K_{b]} \sigma^{d-4}+
2(d-3)(d-4)n^a n_{[a}L_{b]} \sigma^{d-4}
\\[1mm] 
+2 (d-3)(d-4) n^a n_{[a} K_{b]} \sigma^{d-4} \log(\sigma/\tau)+
{\mathcal O}\big(\sigma^{d-3}\log(\sigma/\tau)\big)
+{\mathcal O}\big(\sigma^{d-3}\big)
\end{multline*}
and
\begin{multline*}
-(d-4) F'_{nb}=
-(d-4) F_{nb}-2(d-4)
n^a\big(n_{[a}  K_{b]} \sigma^{d-4}
+ (d-3) n_{[a} L_{b]} \sigma^{d-4}
\big)\\[1mm]
-2 (d-3)(d-4) n^a n_{[a} K_{b]} \sigma^{d-4} \log(\sigma/\tau)
+{\mathcal O}\big(\sigma^{d-3
}\log(\sigma/\tau)\big)
+{\mathcal O}(\sigma^{d-3})\, .
\end{multline*}
Thus we have
$$
j[A']=
\sigma^{d-4} k -
(d-3)\sigma^{d-4} K^{\top^{\rm ext}}  
+{\mathcal O}\big(\sigma^{d-3
}\log(\sigma/\tau)\big)
+{\mathcal O}(\sigma^{d-3})\, ,
$$
where in the above $K^{\top^{\rm ext}}_a$ denotes $2n^a n_{[a}K_{b]}\stackrel\Sigma=K^\top_b$. Hence Equation~\nn{OK} gives the desired solution.
\end{proof}

When the obstruction current $\bar k$ vanishes, an all order asymptotic solution without logs is available when any  suitable choice of the boundary behavior of $L$ is provided; see Theorem~\ref{all-orders-magnetic}. Indeed, $L$ captures the data of an ``electric field'', as discussed in the next section.

\subsection{Formal Electric Problem}\label{EP}

\noindent
We now consider Problem~\ref{problemono}, for which
the curvature vanishes along $\Sigma$. However, the two-form and density-valued endomorphism
$
\big(\sigma^\alpha F^A\big)\big|_\Sigma
$
may now be viewed as boundary datum. Neither it nor its weight $\alpha$ may be chosen arbitrarily, but rather must obey the stipulations of the following lemmas.

\begin{lemma}\label{handofgod}
Let $\nabla^A$ be a connection on a conformally compact structure $(M,\cc,\sigma)$
such that
$$
j[A]=0
\quad
\mbox{ and }
\quad
F^A=\sigma^{-\alpha}G\, ,$$
where $G\in  \Gamma(\wedge^2 M\otimes\End{\mathcal V}M[\alpha])$ and $0>\alpha\in {\mathbb R}$. Then if $\alpha\neq 4-d$
$$G\big|_\Sigma=0\, ,$$
while when $\alpha=4-d<0$ 
$$
\hat n_{[a} G_{bc]}\stackrel\Sigma=0\, .
$$
\end{lemma}  
\begin{proof}
For brevity denote $\nabla:=\nabla^A$, $F:=F^A$ and  $G^\Sigma:=G\big|_\Sigma$. Then (away from $\Sigma$) the
 Bianchi identity gives
 $
 0=
 \sigma^{\alpha+1}
\nabla_{[a}F_{bc]}=-\alpha n_{[a} G_{bc]}
+{\mathcal O}(\sigma)
$
and thus, along~$\Sigma$,
$$
\hat n_{[a} G^\Sigma_{bc]}=0\, .
$$
Moreover
\begin{equation}\label{forlater}
j[A]=\sigma^{-\alpha} \big(-\alpha n^a G_{ab}-(d-4) n^a G_{ab}+{\mathcal O}(\sigma)\big)=0\, ,
\end{equation}
so that, along $\Sigma$,
$$
(\alpha+d-4)\, \hat n^a G^\Sigma_{ab}=0\, .
$$
There is no non-vanishing solution to the above two displays for $G^\Sigma$ unless $\alpha=4-d$.
\end{proof}

\noindent
Observe, that the above result matches the behavior in Equation~\nn{electric} given in the characteristic exponent analysis at the beginning of Section~\ref{Apollosrevenge}.

\smallskip
The next lemma shows that the boundary data for Problem~\ref{problemono} is encoded by an ``electric field'' subject to a non-abelian Gau\ss\ law (Equation~\nn{GL}).

 \begin{lemma}\label{fingerofgod}
Let $\nabla^A$ be a connection on a  $d\geq 5$ conformally compact structure $(M,\cc,\sigma)$ with boundary condition $\bar \nabla^{\bar A}$.
If 
$$
j[A]={\mathcal O}(\sigma^\infty)
\quad
\mbox{ and }
\quad
F^A=\sigma^{d-4}G\, ,
$$
for some smooth $G\in  \Gamma(\wedge^2 M\otimes\End{\mathcal V}M[4-d])$, 
then $G^\Sigma:=G\big |_{\Sigma}$ must obey
\begin{equation}\label{nwedgeE}
G^\Sigma_{ab}=\hat n_{a} E_{b}-\hat n_b E_a\, , 
\end{equation}
where $E\in  \Gamma(T^* \Sigma\otimes\End{\mathcal V}\Sigma[3-d])$ is subject to
\begin{equation}\label{GL}
\bar \nabla^{ {\bar A}}_a E^a=0 \, .
\end{equation}
\end{lemma}

\begin{proof}
We need to study the $\sigma^{d-4}{\mathcal O}(\sigma)$ terms in $j[A]$ that were  not computed in Equation~\nn{forlater} in the proof of Lemma~\ref{handofgod}. These are
$$
j_b[A]=\sigma^{d-3} \nabla^a G_{ab}\, .
$$
Now, because $\hat n_{[a} G_{bc]}\stackrel\Sigma=0$, it follows  that
$$
G^\Sigma_{ab} =\hat n_{a} E_{b}-\hat n_b E_a \, ,
$$ 
where $E\in  \Gamma(T^* \Sigma\otimes\End{\mathcal V}\Sigma[3-d])$. Note that  
$E_b=G^\Sigma_{\hat n b} $ because $\hat n^a E_a=0$. Using the result of the display before last, dividing by~$\sigma^{d-3}$ in the interior,   then extending to $\Sigma$ and contracting with a unit conormal  yields
$$
0\stackrel\Sigma=\hat n^b \nabla^a G_{ab}\stackrel\Sigma=
\hat n^b (\nabla^a-\hat n^a \nabla_{\hat n}) G_{ab}
\stackrel\Sigma=\hat n^b (\nabla^a-\hat n^a \nabla_{\hat n})(\hat n_{a} E_{b}-\hat n_b E_a)
\stackrel\Sigma=-\bar \nabla_a E^a\, .
$$
Here we used antisymmetry of $G$ to achieve the 
second equality. 
The third equality used that the operator $\nabla^a-\hat n^a \nabla_{\hat n}$ is tangential.
For the fourth we used that
$\hat n^b (\nabla_a-\hat n_a \nabla_{\hat n})\hat n_{b}
\stackrel\Sigma=0\stackrel\Sigma=\hat n^a E_a$. 
\end{proof}

\begin{remark}
It is easy to see that there exist connections $A_0$ satisfying 
$F^{A_0}=\sigma^{d-4} G$  where $G$ obeys Equations~\nn{nwedgeE} and~\nn{GL}. For example,  in the Fefferman--Graham expansion explained in Section~\ref{connexp}, one can take  the connection coefficients to  vanish up until order~$r^{d-3}$, with then the coefficient  $E/(d-3)$.
\hfill$\blacklozenge$
\end{remark}

Given a connection $\nabla^A$, whose curvature obeys $F^A=\sigma^{d-4} G$,
 the  data~$E_b=\hat n^a G^\Sigma_{ab}\in \Gamma( T^* \Sigma\otimes\End{\mathcal V}\Sigma[3-d])$
 may be viewed
as the boundary values of a non-abelian  ``electric field''. Indeed, Lemma~\ref{fingerofgod}
shows that 
 the electric field $E$  of 
solutions to
Problem~\ref{problemono}  
 must obey the Gau\ss\ law~\nn{GL}.  
We are now ready to prove Theorem~\ref{electrictheorem} 
which gives a recursive solution to Problem~\ref{problemono} for connections with  boundary electric data.

\begin{proof}[Proof of Theorem~\ref{electrictheorem}]
The proof mimics that of Theorem~\ref{recurse}.
From Lemmas~\ref{handofgod} and~\ref{fingerofgod} and the boundary conditions in the statement of the theorem,
it follows that 
$$
j[A_0]=\sigma^{d-3} k^{(1)}
$$
for some $k^{(1)}\in \Gamma(T^*M\otimes \End{\mathcal V}M[2-d])$. The power $d-3$ of $\sigma$ in the above display is precisely the value for which 
the system of equations~\nn{system} degenerates 
and thus does not imply that $\hat n^a k_a^{(1)}\stackrel\Sigma=0$. Nonetheless, this condition still holds by virtue of 
   the computation presented in the  last display proof of Lemma~\ref{fingerofgod} combined with the fact that
 the boundary electric field obeys the Gau\ss\ Law~\nn{GL}. Thus the recursion used to 
prove Theorem~\ref{recurse} can be applied at all orders since Equation~\nn{solve} is needed only for values of $\ell\geq d-2$.
\end{proof}

We can use the methods of the above proof to prove the all orders result of Theorem~\ref{all-orders-magnetic}. This gives a smooth to the boundary,  asymptotic solution to the magnetic problem~\ref{yeswehavethem} in the case that 
 the obstruction current $\bar k$ vanishes and additional electric field boundary data is supplied.
  \begin{proof}[Proof of Theorem~\ref{all-orders-magnetic}]
 Because the obstruction current $\bar k=0$ we have that 
  $$
 j[A_0]=\sigma^{d-3} k^{(1)}\, ,
 $$
 for some $k^{(1)}\in \Gamma(T^*M\otimes \End{\mathcal V}M[2-d])$. 
 This order is precisely encountered in the proof of Theorem~\nn{recurse} where the corresponding right hand side of the above display is  {\it not} forced to obey  $\hat n^a k^{(1)}_a\stackrel\Sigma=0$.  
 To achieve a condition on $\hat n^a k_a^{(1)}$, we must partially use the undetemined-ness of the data $L$ (see  Proposition~\ref{logsol} for another instance of this). Thus
we now extend Equation~\nn{diffo}
to include terms of order $\sigma^{d-3}$ 
$$
j_b[A_0+\sigma^{d-3} L]\!=\sigma^{d-3}k^{(1)}_b+\sigma^{d-3} n^a (\nabla_a L_b
-\nabla_b L_a)
+(d-3)\sigma^{d-3} \nabla^a (n_a L_b - n_b L_a)
+{\mathcal O}(\sigma^{d-2})\, .
$$
Here we have denoted by $\nabla$ the connection $\nabla^{A_0}$.
Contracting the above with $n^b$ gives
\begin{equation}\label{ncont}
n^b j_b[A_0+\sigma^{d-3} L]=\sigma^{d-3}n^b k^{(1)}_b+
(d-3) \sigma^{d-3}
n^b \nabla^a (n_a L_b
-n_b L_a)
+{\mathcal O}(\sigma^{d-2})\, ,
\end{equation}
where we used that $n^a n^b (\nabla_a L_b
-\nabla_b L_a)=0$.
Now define
$$
E^{\rm ext}_a=\sqrt{n^2 -\frac{2\sigma}{d}(\Delta^g \sigma + J \sigma)}\,  L_a \stackrel \Sigma= |n|_g L_a\, .
$$
The factor under the square root is a conformally invariant extension of $|n|_g\big|_\Sigma$ to $M$ (indeed it equals $\sqrt{-\frac{Sc^{\go}}{d(d-1)}}\hh$; see Section~\ref{theenergy} for further discussion of this point), so $E^{\rm ext}\in \Gamma(T^*M\otimes \End{\mathcal V}M[3-d])$. Now, since $L$ only appears in the combination $n_a L_b - n_b L_a$
in Equation~\nn{ncont}, 
we may further assume that $L$, and hence $E^{\rm ext}$ are both extensions of sections of  $T^*\Sigma\otimes \End{\mathcal V}\Sigma [3-d]$ so that $\hat n^a E_a^{\rm ext}\big|_\Sigma =0$ where $E = E^{\rm ext}\big|_\Sigma$. Thus
$$
n^b \nabla^a (n_a L_b
-n_b L_a)\stackrel\Sigma=
n^b \nabla^a{}^\top \big(\hat n_a E_b^{\rm ext} - \hat n_b E_a^{\rm ext}\
\big)
\stackrel\Sigma=- |n| \nabla^a{}^\top E_a^{\rm ext}
\stackrel\Sigma=- |n| \bar \nabla_aE^a\, .
$$
Using that the electric field $E$ obeys the Gau\ss\ law~\nn{GL} we now have that
 $$
 j[A_0+\sigma^{d-3}L]=\sigma^{d-3} j^{(1)}\, ,
 $$
 for some $j^{(1)}\in \Gamma(T^*M\otimes \End{\mathcal V}M[2-d])$ such that
$\hat n^a j^{(1)}_a\stackrel\Sigma=0$.
The remainder of the proof follows that of Theorem~\ref{recurse}, {\it mutatis mutandis}.
\end{proof}
 
 \begin{remark}
In the case that the obstruction current $\bar k$ vanishes (a special case of which includes the purely electric solution to Problem~\ref{EP}), we have established the existence of a smooth all order  asymptotic solution. This depends on the data of a boundary connection $\bar \nabla$ and a boundary electric field $ E\in  \Gamma(T^* \Sigma\otimes\End{\mathcal V}\Sigma[3-d])$ whose divergence is constrained by Equation~\nn{GL}. However, we do not expect the data $(\bar \nabla,E)$ to be independent in the case of a global solution, but rather---viewing these respectively as   Dirichlet and  Neumann data---they ought be related by a (non-linear) Dirichlet-to-Neumann map; see Section~\ref{FGrescale}. 
 \hfill$\blacklozenge$
\end{remark}


\section{Poincar\'e--Einstein Manifolds}\label{PEM}

A key problem handled by our methods is the study of the Yang--Mills equations on Poincar\'e--Einstein manifolds. 
Here we gather  technical results for Poincar\'e--Einstein structures and their
asymptotically Yang--Mills connections. 
Recall that a conformally compact structure $(M,\cc, \sigma)$ is called {\it Poincar\'e--Einstein} when 
the singular metric, defined on  $M_+=M\backslash\Sigma$ by
$$
\go=\sigma^{-2}{\bm g}\, ,
$$
 has vanishing trace-free Schouten tensor
$$
P^{\go}_{(ab)_\circ}=0\, .
$$
The moniker {\it asymptotically Poincar\'e--Einstein} is used when
\begin{equation}\label{Moniker}
P^{\go}_{(ab)_\circ}={\mathcal O}(\sigma^{d-3})\, .
\end{equation}
It is  useful to rewrite the above condition in a 
way that extends to the boundary~$\Sigma$.
A standard computation on $M^+$ shows that
 $$P^{\go}_{(ab)_\circ}
\stackrel g=
\sigma^{-1}\big(\nabla_{(a} \nabla_{b)_\circ}  + P_{(ab)_\circ}\big)\hh \sigma  \, .
$$
Multiplying this by $\sigma$  we have  that $(M,\cc,\sigma)$ is asymptotically Poincar\'e--Einstein when
\begin{equation}\label{gradn}
\nabla_a n_b + \sigma P_{ab} + g_{ab} \rho = {\mathcal O}(\sigma^{d-2})\, . 
\end{equation}
Here $n:=\ext \sigma$ and 
$\rho:=-\frac1d (\Delta + J)\sigma$. Importantly the left hand side of the above display extends smoothly to the boundary $\Sigma$. 

\smallskip

The scalar curvature of $\go$ can be also expressed in terms of $g$ and $n$, indeed
\begin{equation}\label{icecream}
Sc^{\go}\stackrel g=-d(d-1)\big(|n|_g^2 + 2 \rho \sigma\big)\, .
\end{equation}
This shows that the scalar curvature of the singular metric 
 extends smoothly to the boundary.
The above display implies that the Poincar\'e--Einstein metric $\go$ has {\it negative} constant scalar curvature; by convention we choose the normalization $Sc^{\go}=-d(d-1)$.

\smallskip

The conformal boundary embedding $\Sigma\hookrightarrow(M,\cc)$ for a Poincar\'e--Einstein structure must be umbilic~\cite{Goal,LeBrun}, meaning that the trace-free part $\IIo_{ab}$  of the second fundamental form 
vanishes.
The following lemma gives identities for the boundary embedding of a Poincar\'e--Einstein structure, 
many of which follow directly from the umbilic property.
Note that bars are used to indicate boundary quantities so, for example, $\bar P_{ab}$ is the boundary Schouten tensor for the Levi-Civita connection $\bar \nabla$ of the induced boundary metric $\bar g_{ab}\stackrel\Sigma= g_{ab}-\hat n_a \hat n_b$. Also, we use the same abstract indices to denote 
boundary tensors, relying on the Gau\ss\ isomorphism between $T\Sigma$ and $(TM|_\Sigma)^\perp$ where $\perp$ denotes the  bundle projection orthogonal to the unit conormal $\hat n$.

\begin{lemma}\label{thelemma}
Let $d\geq 3$ and $(M,\cc,\sigma)$ be an asymptotically  Poincar\'e--Einstein structure. Also, let $g\in \cc$ and 
 $\bar v^{\rm ext}$ be any smooth extension of $\bar v\in \Gamma(T^*\Sigma)$ to $\Gamma(T^*M)$. Then
\begin{equation}
\label{top2bar}
\nabla^\top_a \bar v^{\rm ext}_b \stackrel\Sigma= \bar \nabla_a \bar v_b -H\hat n_b \bar v_a\, . 
\end{equation}
Moreover the following identities hold.
\begin{eqnarray}
 \label{Wn}
 W_{\hat n bcd}&\stackrel\Sigma=&0\, ,
 \\
 \label{J2Jbar}
 J\: &\stackrel\Sigma=&{\bar J}+P_{\hat n \hat n}-\frac{d-1}2 H^2\, ,\\
 \label{Fialkow}
 P^\top_{ab}\, &\stackrel\Sigma =& \bar P_{ab} -\frac 12 \bar g_{ab} H^2\, ,
 \\
 \label{mainardi}
\bar \nabla_a H &\stackrel\Sigma=& -P_{\hat n a}^\top\, ,
\\
\label{Mainardi}
\bar \nabla_{[a} P^\top_{b]\hat n}&\stackrel\Sigma=&0\, ,
\\
\label{zanky}
\big( \hat n^a \nabla_{\hat n} P_{ab}\big)^\top
 &\stackrel{\Sigma}=& \bar \nabla_b P_{\hat n\hat n} +2  H \bar \nabla_b H\, ,
 \\
 \hat n^a \hat n^b \nabla_{\hat n} P_{ab}\:\:
&\stackrel\Sigma=&
\nabla_{\hat n } J 
-(d-1)HP_{\hat n\hat n}
+(\bar \Delta+\bar J) H
-\tfrac{d}{2} H^3
 \label{baloo}
\, .
\end{eqnarray}
In dimensions $d\geq 4$,
\begin{equation}
\label{W2W}
W_{abcd}\stackrel\Sigma=
W^\top_{abcd}
\stackrel \Sigma= \bar W_{abcd}\stackrel{d=4}=0\, .
\end{equation}
When $d\geq 5$, we also have
\begin{equation}\label{cotton}
C_{abc}\stackrel\Sigma=C_{abc}^\top\stackrel\Sigma=
\bar C_{abc} \, .
\end{equation}
\end{lemma}
\begin{proof}
These are all specializations of known hypersurface embedding results to Poincar\'e--Einstein structures.
By umbilicity of the  conformal boundary embedding~\cite{Goal,LeBrun}  we have that
$$
\II_{ab} = \bar g_{ab} H\, .
$$
Recalling that the second fundamental form relates hypersurface and ambient derivatives according to $\nabla^\top_a \bar v^{\rm ext}_b \stackrel\Sigma= \bar \nabla_a \bar v_b -\hat n_b\II_a^c \bar v_c$, the above display implies 
Equation~\nn{top2bar}.

\smallskip
Taking a gradient of Equation~\nn{gradn}
and then skewing over a pair of indices gives
$$
R_{abc\hat n}+2\hat n_{[a} P_{b]c}+2g_{c[b}\nabla_{a]}\rho\stackrel\Sigma=0\, .
$$
The (boundary) trace-free part of the above recovers Equation~\nn{Wn}.

Equations~(\ref{J2Jbar}-\ref{Mainardi}) are standard hypersurface identities following directly from the equations of Gau\ss, Codazzi, Mainardi and Ricci, see for example~\cite{Will1}, and the vanishing of~$\IIo$ for Poincar\'e--Einstein embeddings.
Note also that Equation~\nn{Mainardi} is a direct consequence of Equation~\nn{mainardi}.

The key to establishing Equation~\nn{zanky} is to first rewrite $\hat n^a \hat n^c \stackrel\Sigma= g^{ac} - \bar g^{ac}$ and then apply $\nabla^a P_{ab}=\nabla_b J$. The details are as follows:
 \begin{eqnarray*}
 \hat n^a \nabla_{\hat n} P_{ab}
 &\stackrel{\Sigma,\top}=&\bar \nabla_b \bar J + \bar \nabla_b P_{\hat n\hat n} -(d-1) H \bar \nabla_b H
 -\nabla^\top_a P^a{}_b\\
 &\stackrel{\Sigma,\top}=&\bar \nabla_b \bar J + \bar \nabla_b P_{\hat n\hat n} -(d-1) H \bar \nabla_b H
 -\bar\nabla^a \big(\bar P_{ab}-\tfrac12 \bar g_{ab} H^2\big)+d H\bar \nabla_b H\\
 &\stackrel{\Sigma,\top}=&
 \bar \nabla_b P_{\hat n\hat n} +2  H \bar \nabla_b H\, \, .
 \end{eqnarray*}
In the above $\stackrel\top =$ denotes projection of the index $b$  to directions tangential to $\Sigma$.
A similar computation gives Equation~\nn{baloo}:
\begin{eqnarray*}
\hat n^a \hat n^b \nabla_{\hat n} P_{ab}&\stackrel\Sigma=&
\hat n^a (g^{bc}-\bar g^{bc}) \nabla_c P_{ab}
\stackrel\Sigma=\nabla_{\hat n } J - 
\hat n^a
\nabla^\top_b P_a{}^b\\
&\stackrel\Sigma=&
\nabla_{\hat n } J - 
\hat n_a \nabla^\top_b 
(
\bar P^{ab} -\tfrac12 \bar g^{ab} H^2
- \hat  n^a \bar \nabla^b H 
- \hat  n^b \bar \nabla^a H
+\hat n^a \hat n^b P_{\hat n\hat n}) \\
&\stackrel\Sigma=&
\nabla_{\hat n } J 
+H\bar J
-\tfrac{d}{2} H^3
+\bar \Delta H
-(d-1)HP_{\hat n\hat n}
\, .
\end{eqnarray*}

Note  that, in dimensions $d\geq 4$, the tensor $(d-3)(P^\top_{ab}-\bar P_{ab} +H\IIo_{ab} -\frac12 \bar g_{ab} H^2)$ gives a  (weight zero) conformal hypersurface invariant~\cite{Fialkow} termed the Fialkow tensor in~\cite{Stafford,Vyatkin}. Thus Equation~\nn{Fialkow} shows that the Fialkow tensor vanishes for asymptotically  Poincar\'e--Einstein structures. The Fialkow tensor and trace-free second fundamental forms are low-lying examples of the conformal fundamental forms of~\cite{Blitz} which measure obstructions to given conformal embeddings admitting Poincar\'e--Einstein metrics.
The tensors $W^\top_{abcd}$ and $\bar W_{abcd}$ agree when  both the Fialkow tensor and trace-free second fundamental form vanish;  this gives 
Equation~\nn{W2W}.

To establish the second equality in Identity~\nn{cotton}. we compute
$$
\bar \nabla_{[a} \bar P_{b]c}
\stackrel\Sigma=\bar \nabla_{[a} P_{b]c}^\top - \bar g_{c[b}H\bar \nabla_{a]} H=\big(
\nabla^\top_{[a} P_{b]c}^\top
+H \hat n_{[b} P_{a]c}^\top
- \bar g_{c[b}H\bar \nabla_{a]} H\big)^\top
=(\nabla_{[a} P_{b]c})^\top\, .
$$  
The first equality above used Equation~\nn{Fialkow} while the second relied on Equation~\nn{top2bar} and that the operation $\top$ is a projector. 
The last equality employs a combination of Equations~\nn{mainardi} and the fact that for asymptotically  Poincar\'e--Einstein structures we have
\begin{equation}\label{nablan2H}
\nabla_a^\top \hat n_b^{\rm ext} \stackrel\Sigma= \bar g_{ab} H\, ,
\end{equation}
for any smooth extension of $\hat n$ to $M$.
The result then follows from the definition of the Cotton tensor.

It remains to prove the equalities~\nn{cotton}. For that we note the result (see~\cite{Goal}, for example) 
$$
n^a W_{abcd}= \sigma C_{cdb}+ {\mathcal O}(\sigma^{d-3})\, .
$$
This gives an alternative derivation of Equation~\nn{Wn}
 because  the Weyl tensor vanishes identically in $d=3$ and the right hand side is zero along $\Sigma$ when
   $d\geq 4$. Hooking $n^b$ into the above display and then taking a normal derivative gives  $n^b C_{cbd}|_\Sigma=0$ (when $d\geq 5)$. Remembering that the Cotton tensor obeys $C_{[abc]}=0=C_{(ab)c}$ then gives the result.
\end{proof}

\medskip
For many of our computations, the extension of the unit conormal $\hat n^{\rm ext}$ (the adornment ${\rm ext}$ is written as a subscript when the metric is used to produce the corresponding normal vector $\hat n^a_{\rm ext}$) given by $n=\nabla \sigma$, where $(M,\cc,\sigma)$ is an asymptotically Poincar\'e--Einstein structure, 
plays an important {\it r\^ole}; the following result  gives pertinent identities for this. 

 \begin{lemma}\label{69}
Let   $(M^d,\cc,\sigma)$ be an asymptotically Poincar\'e--Einstein structure.
Then  the extension $n_a = \nabla_a \sigma$ of the unit conormal $\hat n_a$ obeys the following identities:
 \begin{equation}\label{I2} n^2 = \nabla_n \sigma=1-2 \rho \sigma\stackrel\Sigma=1\, ,\end{equation}
 where 
 $
 \rho:=-\frac1d (\Delta \sigma + J\sigma)
 $. 

\bigskip 
 When $d\geq 3$
  \begin{equation}\label{nablan}
 \nabla_a n_b\stackrel\Sigma= g_{ab} H\, ,
 \end{equation}
 \begin{equation}\label{nablann}
 \nabla_n n_a\stackrel\Sigma= H \hat n_a\,.
 \end{equation}
\smallskip

 When $d\geq 4$
\begin{equation}\label{nablanrho}
\nabla_{\hat n}\rho\stackrel\Sigma=P_{\hat n\hat n}\, ,
\end{equation}
\begin{equation}\label{nablannn}
\nabla_n^2 n_a\stackrel\Sigma=
\big(H^2-2P_{\hat n\hat n})\hat n_a
-P_{\hat n a}^\top \, .
\end{equation}
\begin{equation}\label{PP}
\nabla_{\hat n}\nabla_a n_b\stackrel\Sigma=-P_{ab}-g_{ab}P_{\hat n\hat n}\, ,
\end{equation}
\begin{equation}\label{PPP}
\nabla_a \nabla_b n_c\stackrel\Sigma = 
-\hat n_a P_{bc} -g_{bc} (P_{\hat n a}^\top +\hat n_a P_{\hat n\hat n})\, .
\end{equation}
\smallskip

When $d\geq 5$,
\begin{equation}\label{nablannrho}
\nabla_n^2 \rho
\stackrel\Sigma = 
-\bar\nabla^a P_{\hat n a}^\top
-(d-2) HP_{\hat n\hat n }
+(\nabla_{\hat n} +H)J\, ,
\end{equation}
\begin{equation}\label{P}
\nabla_n^3 n_a
\stackrel\Sigma=
\Big(H^3
+(3d-10)HP_{\hat n\hat n}
+3\bar\nabla^a P_{\hat n a}^\top
-3(\nabla_{\hat n} +H)J
\Big)\hat n_a
-2\bar \nabla_a P_{\hat n\hat n} -H P^\top_{\hat n a}\, .
\end{equation}
\smallskip

 \end{lemma}

\begin{proof}
Because $(M,\cc,\sigma)$ is asymptotically Poincar\'e--Einstein, it follows that the scalar curvature $Sc^\go$  equals $-d(d-1)$ along $\Sigma$. Hence the first identity in Equation~\nn{I2} 
follows directly from Equation~\nn{icecream}.

To establish Equation~\nn{nablan}, we note that 
Equation~\nn{gradn} gives 
$
\nabla_a n_b \stackrel\Sigma= -g_{ab} \rho$, so we need to show that This that
 $\rho \stackrel\Sigma= -H$.
 This was  first proved in~\cite{Goal}, the argument is as follows:
\begin{multline*}
-\rho \stackrel\Sigma=\frac1d  \Delta \sigma  \stackrel\Sigma= 
 \frac1d \big(\nabla_a^\top n^a+ \hat n_a \nabla_{\hat n} n^a\big)
  \stackrel\Sigma= 
 \frac 1d\big( \nabla_a^\top \hat n^{a}_{\rm ext}+\frac1{2}  \nabla_{\hat n} n^2\big)
\\
\stackrel\Sigma= 
\frac{d-1}d H
+\frac1{2d} \nabla_{\hat n}\big(1-2\rho\sigma+{\mathcal O}(\sigma^2)\big)
\stackrel\Sigma=
 \frac{d-1}d H
-\frac\rho d
  \, .
\end{multline*}
In the above display, the first equality used the definition of $\rho$
while the second relied on the identity $\nabla_a n^a=\Delta \sigma$. The third used 
that $n$ is an extension of $\hat n$. The fourth equality relied on Equation~\nn{nablan2H} and that $n^2 + 2\rho \sigma = 1 + {\mathcal O}(\sigma^{d-1})$. The latter fact can be established by applying the  identity $\nabla^a P_{ab} = \nabla_b J$ to Equation~\nn{Moniker}. (In fact, as explained in~\cite{Will1}, it is possible to find an  improved defining density $\sigma$ that achieves ${\mathcal O}(\sigma^d)$, but this is unimportant here.)  The final equality used that~$\nabla_{\hat n} \sigma  \stackrel\Sigma=1$.

\smallskip

 Equation~\nn{nablann} follows directly from Equation~\nn{nablan}. To establish Equation~\nn{nablanrho} we compute as follows:
\begin{multline*}
\nabla_{\hat n}\rho \stackrel\Sigma=
-\frac1d \nabla_{\hat n} \big(\nabla.n + J \sigma)
\stackrel\Sigma=
-\frac Jd -\frac1d n^a (\nabla_b \nabla_a n^b +R_{ab}{}^b {}_n)
\\
\stackrel\Sigma=
-\frac Jd 
+\frac1d Ric_{\hat n\hat n}-\frac1d n^a \Delta n_a
\stackrel\Sigma=\frac{d-2}d P_{\hat n\hat n}
-\frac1d \Delta n^2 +\frac1d  \big[\Delta, n^a\big]n_a
\\
\stackrel\Sigma=
\frac{d-2}d P_{\hat n\hat n}
+\frac1d \Delta \big(\rho\sigma + {\mathcal O}(\sigma^{d-1})\big)+H^2
\stackrel\Sigma=
\frac{d-2}d P_{\hat n\hat n}+ \frac{2}d \nabla_{\hat n}\rho\, .
\end{multline*} 
The first equality uses the definition of $\rho$, the second relies on $\nabla_{\hat n} \sigma\stackrel\Sigma=1$ while the third uses that the one-form $n$ is exact.
The fourth is the definition of the commutator of operators while the fifth equality follows upon noting that $[\Delta, n_a] n^a=(\nabla_b n_a)\nabla^b n^a +\frac 12 \Delta n^2$ and then applying Equation~\nn{nablan}. The sixth equality uses $\Delta (\rho \sigma)\stackrel\Sigma = 2 \nabla_{\hat n} \rho+d H^2$ and that $\Delta {\mathcal O}(\sigma^{d-1})= {\mathcal O}(\sigma^{d-3})$. The result then follows after some trivial algebra.

\smallskip
To establish Equation~\nn{nablannn} we compute along similar lines:
\begin{equation*}
\nabla_n^2 n_a =
\nabla_n (\frac 12 \nabla_a n^2)
\stackrel\Sigma=\!
-\nabla_n \nabla_a (\rho\sigma)
\stackrel\Sigma=\!
-\nabla_a \rho
-\nabla_n (n_a \rho)
\stackrel\Sigma=\!
\bar \nabla_a H
-\hat n_a P_{\hat n\hat n}
+\hat n_a H^2 -\hat n_a P_{\hat n\hat n}\, .
\end{equation*}
The result follows upon application of 
Equation~\nn{mainardi}. 

Equations~\nn{PP} and~\nn{PPP} are direct consequences of Equation~\nn{gradn}.
Equation~\nn{nablannrho} may be obtained from Lemma 3.10 of~\cite{Will2} upon specialization to asymptotically Poincar\'e structures, while for Equation~\nn{P} we calculate as below, again using the same methods as above:
\begin{multline*}
\nabla_n^3 n_a \stackrel\Sigma=
-\nabla_n^2 \nabla_a\big(\rho\sigma+{\mathcal O}(\sigma^{d-1})\big)
\stackrel\Sigma=
-\nabla_n \nabla_a \rho
-\nabla_n((1-2\rho\sigma) \nabla_a \rho)
-\nabla_n^2(n_a \rho)
\\
\stackrel\Sigma=-2\nabla_n \nabla_a \rho
-2H\nabla_a \rho
-\hat n_a H \nabla_n \rho
-\nabla_n((\nabla_n n_a) \rho)
-\hat n_a \nabla_n^2\rho
\\
\stackrel\Sigma=
-2\nabla_a \nabla_n\rho
-2\hat n_a H \nabla_n \rho
+H\nabla_n^2 n_a 
-\hat n_a \nabla_n^2\rho
\\
\stackrel\Sigma=
-2(\bar \nabla_a+\hat n_a H) P_{\hat n\hat n}
-3 \hat n_a \nabla_n^2 \rho
+H\nabla_n^2 n_a\, .
\end{multline*}
The result then follows upon employing Equations~\nn{nablannn} and~\nn{nablannrho}.

\end{proof}

It is also useful to relate the bulk and boundary Laplace operators.
\begin{lemma}\label{lap2lap}
Let   $(M,\cc,\sigma)$ be an asymptotically Poincar\'e--Einstein structure and $g\in \cc$.
Also let $f\in C^\infty M$ such that $f|_\Sigma= \bar f\in C^\infty \Sigma$. Then
$$
\Delta f|_\Sigma = \bar \Delta \bar f + \hat n^a \hat n^b \nabla_a \nabla_b f|_\Sigma 
+(d-1) H \nabla_{\hat n} f\, .
$$
\end{lemma}

\begin{proof}
The proof can be found (in a more general context) in~\cite[Lemma A.2]{Will2}.
\end{proof}

\subsection{Canonical Expansions}\label{canexp}

The boundary asymptotics
of  a quantity $Q$ on a conformally compact structure~$(M,\cc,\sigma)$ are described by computing 
$Q_{(k)}$ such that
$$
Q=Q_{(k)}+\sigma^{k+1} R=Q_{(k)}+{\mathcal O}(\sigma^{k+1})\, ,
$$
where $R$ is smooth on $M$ and $k={\mathbb Z}_{\geq 0}$. No coordinate choice has been made in the above expression.
Upon picking 
coordinates $(x,r)$ for a collar neighborhood of the boundary such that $r$ itself gives a defining function thereof, one can then study a formal asymptotic series $$Q\stackrel{r \to 0} \sim Q^{(0)}+Q^{(1)} r \, +  Q^{(2)} r^2 +\cdots\, .$$ 
When $Q$ is a function,  one may ask that  the coefficients~$Q^{(k)}$ are $r$ independent. For a canonical
coordinate $r$,  the coefficients~$Q^{(k)}$ are of independent interest. To extract them 
from $Q_{(k)}$, one picks the scale $g\in \cc$ 
such that  $\sigma = [g;r]$ and then computes~$\frac1{k!} \frac{\partial^k Q_{(k)}}{\partial r^k}\Big|_{r=0}$ in that scale. There  exist canonical, conformally invariant boundary operators that perform this computation in  rather general settings~\cite{GPt,Blitz}.

\medskip
 
In the case $Q$ is the singular metric $g^o$, 
   given a choice of boundary metric representative~$\bar h\in \cc_{\Sigma}$, there always exists a canonical Graham--Lee normal form~\cite{GLee}  on $M_+$ such that
\begin{equation}\label{normal}
\go=r^{-2}\big(\ext r^2+h(x,r)\big)\, ,
\end{equation}
in a collar neighborhood of the boundary $\Sigma$, where
$\iota_{\frac\partial{\partial r}} h = 0$
and $h(x,0)=\bar h$.
The coordinate $r$ is canonical as it measures the normal geodesic distance to the boundary with respect to the compactified metric $g\in \cc$ on $M$ given by
\begin{equation}\label{block}
g=\ext r^2+h(x,r)\, .
\end{equation}
Moreover, the embedding $\Sigma \hookrightarrow(M,g)$ is minimal. Here and above $x$ denotes 
some choice of coordinates along $\Sigma$
which also give coordinates along hypersurfaces of  constant $r$ by pulling back the exponential map. 
The boundary $\Sigma$ is the hypersurface at $r=0$.
In~\cite{FG-ast} Fefferman and Graham established that asymptotic solutions to the Einstein condition for $\go$ exist such that
\begin{equation}\label{FG}
\left\{
\begin{array}{cc}
\textstyle
h(x,r)\stackrel{r \to 0}{\sim}
h^{\rm ev}
+
\sum_{\ell=d-1}
h^{(\ell)}(x) r^{\ell}  
\, ,
& d {\rm \ even}\, ,\\[3mm]
h(x,r)\stackrel{r \to 0}{\sim}\
h^{\rm ev}
+ \beta(x)r^{d-1}\log r +  h^{(d-1)}(x)r^{d-1}
 \, ,
& d {\rm \ odd}\, ,
\end{array}
\right.
\end{equation}
where
$$
h^{\rm ev}:=
\sum_{\ell=0}^{\big\lfloor\!\frac{d-2}{2}\!\big\rfloor}
h^{(2\ell)}(x) r^{2\ell} . 
$$
In the above $h^{(0)}$ gives the choice of boundary metric in $\cc_{\Sigma}$.  The absence of a term linear in $r$ further implies that this embedding is totally geodesic. When $d$ is even, the Dirichlet data $h^{(0)}$ uniquely determines
the coefficients~$h^{(1)}, \ldots , h^{(d-2)}$, while higher coefficients are uniquely determined upon specifying suitable Neumann data $h^{(d-1)}$. Evenness of the expansion in $r$ thus breaks at this order. Of course a global solution for $\go$, when such is available, determines the quantity $h^{(d-1)}$. When $d$ is odd, 
the Dirichlet data~$h^{(0)}$ uniquely determines
the coefficients~$h^{(1)}, \ldots , h^{(d-2)}$. The coefficient $\beta(x)$ of the log term  gives the  so-called Fefferman--Graham obstruction tensor~\cite{FG-ast,thebookFG}, and only depends on the conformal class of boundary metrics~$\cc_\Sigma$, while the  Neumann data (undetermined by the asymptotic expansion) is again given by suitable $h^{(d-1)}$; see~\cite{GraOb}.

Higher terms in the expansion for odd $d$ and non-vanishing Fefferman--Graham obstruction tensor are of polylogarithmic type involving higher powers  of $\log r$. The expansions above solve the Einstein condition asymptotically, in the sense
$$
\left\{
\begin{array}{cc}
\textstyle
\mathring P^{\go}(x,r)
={\mathcal O}(r^\infty)
\, ,
& d {\rm \ even}\, ,\\[3mm]
\mathring P^{\go}(x,r)={\mathcal O}(r^{d-2})+ \,  {\mathcal O}(\log r,r^{d-2})\, ,
& d {\rm \ odd}\, ,
\end{array}
\right.
$$
where $\mathring P^{\go}$ is the trace-free Schouten tensor of $\go$ and the notation ${\mathcal O}(\log r,r^{d-2})$ denotes a polyhomogeneous expansion in $r$ and $\log r$ where all terms are at least order $r^{d-2}$.

There  exist conformal fundamental forms~\cite{Blitz} that measure the failure of a conformally compact structure to admit asymptotic solutions to the Einstein condition. These both capture the coefficients $h^{{(k)}}$ determined by the 
boundary Dirichlet data, as well as the leading Neumann data~\cite{Blitzedagain} (assuming special conditions on the boundary conformal class when $d$ is odd).

\subsection{Scalar Poincar\'e--Einstein Boundary Operators}\label{no}

In Section~\ref{models} we construct the Yang--Mills boundary operators discussed in the introduction. Here we construct scalar boundary operators that exhibit salient properties of their Yang--Mills counterparts.
Firstly, the conformally invariant
Robin operator~\cite{Cherrier}, defined for any 
conformally compact structure and given for some choice of scale by
$$
\delta^{(1)}:\stackrel\Sigma{=}\nabla_{\hat n}-w H\, ,
$$
 maps bulk densities  $\Gamma(\ce M[w])$ to boundary densities $\Gamma(\ce \Sigma[w-1])$. 
The above operator is a special case of Equation~\nn{Robin}. It takes one normal derivative, so is said to have transverse order one. In general a smooth operator $\delta$ mapping $\Gamma(\ce M[w])$ to $\Gamma(\ce \Sigma[w'])$ is said to have {\it transverse order $k\in {\mathbb Z}_{\geq 0}$} if $\delta\circ \sigma^k \neq 0=\delta \circ\sigma^{k+1}$, where $\sigma$ is a defining density for the boundary here viewed as a multiplication operator mapping  $\Gamma(\ce M[w-1])$ to $\Gamma(\ce M[w])$.
The general theory of~\cite{GPt} (see also~\cite{BG0}) leads in particular to  transverse order two and three scalar boundary operators:
\begin{proposition}\label{wearenotnormal}
Let $(M_+,\go)$ be a dimension $d$ Poincar\'e--Einstein structure with conformal infinity~$\Sigma$. Then the operators  defined for any $g\in \cc$ by
 $$
\check\delta^{(2)}
:\stackrel\Sigma=
 (d + 2 w - 3)\Big[ \hat n^a \hat n^b\nabla_a \nabla_b
  -2 (w - \tfrac12) H \nabla_{\hat n}
  + w   P_{\hat n\hat n}
   +  w (w - \tfrac 12)H^2
  \Big]
   - (\bar \Delta+w\bar J)
  $$
 and
 \begin{multline*}
\check\delta^{(3)}:\stackrel\Sigma=
(d+2w-5)\Big[\hat n^a \hat n^b \hat n^c \nabla_a \nabla_b \nabla_c
-3(w - 1)H \hat n^a \hat n^b  \nabla_a \nabla_b 
\\
\qquad
+ (3 w - 1) P_{\hat n\hat n} \nabla_{\hat n}
+ \tfrac 32 (2 w - 1) (w - 1) H^2\nabla_{\hat n}\hspace{.3cm}
\\
\qquad\qquad\qquad\qquad\qquad
 + w \big( (\nabla_{\hat n}J) 
  +    (\bar \Delta  H)
  - (d + 3 w - 4) P_{\hat n\hat n}  H
    -\tfrac 12  (2 w^2  + d - 3 w)H^3 
  +\bar J H  \big)
  \\
 -(\bar \nabla^a H) \bar \nabla_a
\Big]\hspace{4.7cm}
\\[1mm]
\quad
-3\big[\bar \Delta +(w-1) \bar J\hh \big]\delta^{(1)}\, ,
\hspace{9.5cm}
\end{multline*}
respectively map $\Gamma(\ce M[w])$ to $\Gamma(\ce \Sigma[w-2])$ and $\Gamma(\ce \Sigma[w-3])$.
\smallskip

When $d=3$, the operator
$$
\delta^{(2_3)}:\stackrel\Sigma=\hat n^a \hat n^b\nabla_a \nabla_b
  + H \nabla_{\hat n}
  $$
  maps $
  C^\infty_{\ker \bar \Delta} M\to   \Gamma(\ce \Sigma[-2])
$ where $
  C^\infty_{\ker \bar \Delta} M:=\big\{f\in C^\infty M\, \big|\,  \bar \Delta (f|_\Sigma)=0\big\}$.

 \end{proposition}
 
 \begin{proof}
 The operators $\check \delta^{(2)}$ and $\check \delta^{(3)}$ are both given,  modulo lower order conformally invariant terms and in a more general context than here, in~\cite{GPt}.
 The operator $\delta^{(2_3)}$ is new. It is a simple exercise to verify its claimed conformal variation property when acting on functions whose restriction to the boundary is harmonic.
 \end{proof}

\begin{remark}\label{3}
When $w$ is $\frac{3-d}2$ and $\frac{5-d}2$, respectively, the operators 
$$
\check \delta^{(2)}=-\bar \square
\:\mbox{ and }\:
\check \delta^{(3)}=
-3\hh \bar\square \circ \delta^{(1)}\, ,
$$
where $\bar \square := \bar \Delta + \frac{3-d}2 \bar J$
is the conformally invariant boundary Yamabe operator mapping $\Gamma({\mathcal E}\Sigma[\frac{3-d}2])\to \Gamma({\mathcal E}\Sigma[\frac{1-d}2])$. 
Their respective  transverse orders are then $0$ and $1$ rather than $2$ and $3 $ in those cases.
\hfill$\blacklozenge$
\end{remark}

The operators $\check\delta^{(1)}:=(d+2w-2)\delta^{(1)}$, $\check \delta^{(2)}$ and $\check \delta^{(3)}$ control the behaviour of (generalized) Laplacian eigenfunctions and harmonics. For the latter case consider the following harmonic problem for Poincar\'e--Einstein manifolds.
\begin{problem}\label{DNPE}
Let $(M_+,g^o)$ be a dimension-$d$ Poincar\'e--Einstein manifold and   $B\in C^\infty \Sigma$ be given. Find $f$, continuous on $M$ and such that $f\in C^\infty M_+$, which obeys
$$\Delta^{g^o} f = 0 = f|_\Sigma -B\, .$$
\hfill$\blacksquare$
\end{problem}
\noindent
It is well known~\cite{Mazzeo,MM,GZ} that Problem~\ref{DNPE}
has a unique solution for $f\in C^\infty M_+$ for $d=3,5,\ldots$ and for $f\in C^\infty M$ for $d=2,4,6,\ldots$. For $d$ even there is a 
canonical {\it Dirichlet-to-Neumann} map
$$
{\mathscr N}: C^\infty \Sigma \to \Gamma({\mathcal E}\Sigma[1-d])\, ,
$$
whose input is the Dirchlet data $B$ in Problem~\ref{DNPE}.
The output is  
defined by expanding the unique solution $f$ in the canonical Fefferman--Graham coordinate $r$ and extracting the coefficient $\frac1{(d-1)!} \frac{\ext^{d-1} f}{\ext r^{d-1}}\big|_{r=0}$
of the odd order $r^{d-1}$ term. 
Evenness of the asymptotic expansion  
of the solution $f$ to order $r^{d-2}$
can be used to show that this then
defines a weight $1-d$ boundary density~\cite{GZ}.


The situation for odd dimensions $d$ is more subtle. There, the unique, smooth global solution for $f$ on $M_+$ has a polyhomogeneous asymptotic series expansion~\cite{GZ}
\begin{multline}\label{pollywantacracker}
f(x,r) = B(x) + r^2 f_2(x) + \cdots  + r^{d-3} f_{d-3}(x)
\\
+\tfrac{(-1)^d[(d-2))!!]^2}{(d-1)!(d-2)!}
r^{d-1} \log r
\, {\sf P}_{d-1}B(x)
+ {\mathscr N}_{\bar h} \big(B(x)\big) r^{d-1} +
{\mathcal O}(r^d)+
\log r\, {\mathcal O}(r^d)\, . 
\end{multline}
Here $\sf P_{d-1}$ denotes a GJMS operator, namely a conformally invariant Laplacian power of the form $\bar \Delta^{^{\lfloor\frac{d-1}2\rfloor}}+{\rm LOTs}$ where ``${\rm LOTs}$'' denotes the lower derivative order terms required for conformal invariance~\cite{GJMS}. 
In the above expansion, the Dirichlet-to-Neumann map $${\mathscr N}_{\bar h}: C^\infty \Sigma \to C^\infty \Sigma\, ,$$ defined by the coefficient of the $r^{d-1}$ term, depends on the choice of boundary metric representative $\bar h \in \cc_\Sigma$~\cite{GZ}.
The boundary operators of Proposition~\ref{wearenotnormal} can be used to extract the image of the Dirichlet-to-Neumann map:

\begin{proposition}
Let $f$ a solution to Problem~\ref{DNPE}
be given for a dimension $d$ Poincar\'e--Einstein manifold with Dirichlet data $B\in C^\infty \Sigma$. Then, when
\begin{enumerate}[(i)]
\item\label{one} $d\geq 3$, $\check\delta^{(1)} f = 0$,
\item \label{two}$d\geq 4$, $\check\delta^{(1)} f =\check\delta^{(2)} f = 0$,
\item \label{three}$d\geq 5$, $\check\delta^{(1)} f =\check\delta^{(2)} f =\check\delta^{(3)} f = 0$.
\end{enumerate}
\medskip

Moreover, when $d=2,4$ respectively, 
$$
\delta^{(1)} f={\mathscr N}(B)\:\mbox{ and }\:
-\tfrac1{3!}\check \delta^{(3)} f={\mathscr N}(B)\, .
$$
Also, when $d=3$ and $\check \delta^{(2)}f=\bar \Delta B= 0$, then $f\in C^\infty M$. 
In this case the Dirichlet-to-Neumann map ${\mathscr N}_{\bar h}: C^\infty \Sigma \cap \ker \bar \Delta\to   \Gamma(\ce \Sigma[-2])$ and is given by the operator $\frac1{2!} \delta^{(2_3)}$.
\end{proposition}
\begin{proof}
First we establish that the operators $\check \delta^{(1)}$, $\check \delta^{(2)}$ and $\check \delta^{(3)}$ annihilate solutions in the claimed dimensions. An elegant proof approach relates these operators to 
powers of the Laplace--Robin operator restricted to $\Sigma$, and then uses that the harmonic condition on  functions is equivalent to lying in the kernel of that operator. Else, one may use that acting on $k\in C^\infty M_+$,
the singular metric  Laplacian obeys
$$
\sigma^{-1}\Delta^{g^o}k=
\sigma\Delta k
-(d-2)\nabla_n k\stackrel\Sigma=-(d-2)\delta^{(1)}k\, .
$$
The last equality used that $n=\ext \sigma\stackrel\Sigma=\hat n$. Taking $k$ to be a solution $f$, this establishes Point~{\it (\ref{one})}, and in also that $\delta^{(1)} f=\nabla_{\hat n}f|_\Sigma = 0$ 
for $d\geq 3$.

Acting with $\nabla_n$ on $\sigma^{-1} \Delta^{g^o}f$, applying Lemma~\ref{lap2lap}, and using that Equation~\nn{pollywantacracker} implies $f\in C^{d-2} M$ for odd $d\geq 5$, gives
\begin{multline}\label{doit}
0=
\sigma \nabla_n \Delta f + |n|^2 \Delta f - (d-2) \nabla_n^2 f
\\ \stackrel\Sigma = 
\big(\bar \Delta -(d-3)
\hat n^a \hat n^b \nabla_a \nabla_b +
 H\nabla_{\hat n}\big) f 
 = 
 -\check \delta^{(2)} f 
 + H\check \delta^{(1)} f\, .
\end{multline}
Hence, when $d\geq 4$, we have $\check\delta^{(1)} f =\check\delta^{(2)} f = 0$, as required for Point~{\it (\ref{two})}.

To prove Point~{\it(\ref{three})}, 
we may use $\check \delta^{(1)}f=0=\check\delta^{(2)}f$ so that
$$
\nabla_{\hat n} f|_\Sigma=0
\mbox{ and }
\hat n^a \hat n^b \nabla_a \nabla_b f|_\Sigma
=\tfrac1{d-3}\bar \Delta f\, .
$$
Hence
\begin{equation}\label{35}
\check\delta^{(3)}f=(d-5)
\Big[ 
\hat n^a \hat n^b \hat n^c \nabla_a\nabla_b \nabla_c f+\tfrac{3}{d-3} H\bar \Delta f-(\bar \nabla^a H)\bar\nabla_a f
\Big]
\, .
\end{equation}
To show that this vanishes for $d>5$ we apply $\nabla_n$ once again to Equation~\nn{doit}, giving
$$
0\stackrel\Sigma=
2 \nabla_n \Delta f+
2H\Delta f - (d-2) \nabla_n^3 f\, .
$$
Here we used $\nabla_n|n|^2\stackrel\Sigma = 2H$.
Now  we develop each term above separately and find that
\begin{align*}
\nabla_n \Delta f
&\stackrel\Sigma
=\hat n^a (\bar g^{bc} + \hat n^b \hat n^c)\nabla_a \nabla_b \nabla_c f
\\[1mm]
&=\hat n^a  \hat n^b \hat n^c\nabla_a \nabla_b \nabla_c f
+\bar g^{bc}R_{\hat n b c}{}^d \nabla_d f
+\hat n^a \bar g^{bc} \nabla^\top_b \nabla_c \nabla_a f
\\[1mm]
&=\hat n^a  \hat n^b \hat n^c\nabla_a \nabla_b \nabla_c f
-Ric_{\hat n d} \bar \nabla^d f-
H \nabla^\top_a \nabla^a f +
\nabla_b^\top (\hat n^a \nabla^b \nabla_a f)
\\[1mm]
&=\hat n^a  \hat n^b \hat n^c\nabla_a \nabla_b \nabla_c f
-(d-2)P_{\hat n a}\bar \nabla^a f
-H \bar \Delta f 
-\bar \nabla_a (H \bar \nabla^a f)
+ \nabla^\top_b
(\hat n^a \hat n^b  \nabla_{\hat n} \nabla_a f)
\\[1mm]
&=\hat n^a  \hat n^b \hat n^c\nabla_a \nabla_b \nabla_c f
-\tfrac{d-5}{d-3}H\bar \Delta f
+(d-3)(\bar \nabla _a H)\bar \nabla^a f\, ,\\[2mm]
\Delta f\: \:&\stackrel\Sigma = \tfrac{d-2}{d-3} \bar \Delta f\, ,
\end{align*}
and
\begin{align*}
\nabla_n^3 f 
&\stackrel\Sigma=
\hat n^b \nabla_{\hat n} \nabla_b \nabla_n f
+ H \nabla_{\hat n} \nabla_n f
\\
&
=
\hat n^b \nabla_{\hat n}  \nabla_n\nabla_b f
+\hat n^b\nabla_{\hat n}\big((\nabla_b n^c )\nabla_c f\big)
+ H \hat n^a \nabla_{\hat n} \nabla_af
\hspace{4.45cm}
\\
&
=\hat n^a  \hat n^b \hat n^c\nabla_a \nabla_b \nabla_c f+
\tfrac{3}{d-3}H  \bar \Delta f
+  (\bar \nabla_a H)\bar \nabla^a f 
\, .
\end{align*}
For the second equality we used Equation~\nn{lap2lap} and the last line above relied on Equation~\nn{nablan}.
Hence Equation~\nn{doit} now implies
$$
-(d-4)\Big[\hat n^a \hat n^b \hat n^c\nabla_a \nabla_b \nabla_c f+\tfrac{3}{d-3} H  \bar \Delta f
-(\bar \nabla_a H)\bar \nabla^a f 
\Big]\stackrel\Sigma=0\, .
$$

Having established Points {\it (\ref{one}-\ref{three})}, we turn to the Dirichlet--Neumann maps. For that we need to study the operators $\check \delta^{(k)}$ ($k=1,2,3$)
in the Fefferman--Graham coordinate system. Note that the embedding $\Sigma\hookrightarrow (M,g)$ for the compactified metric $g$ of Equation~\nn{FG} has zero mean curvature. Also, acting on scalars, the operator $\nabla_n = \frac{\partial}{\partial r}$.
That $
\delta^{(1)} f={\mathscr N}(B)
$ when $d=2$ now follows immediately. 
When $d=4$, we  now have
$$
\check \delta^{(3)}f \stackrel g=
-\hat n^a \hat n^b \hat n^c \nabla_a \nabla_b \nabla_c f\, .
$$
Using our above computation for  $\nabla_n^3=\frac{\partial^3}{\partial r^3}$ in terms of $\hat n^a \hat n^b \hat n^c \nabla_a \nabla_b \nabla_c f$, it readily follows that
$-\tfrac1{3!}\check \delta^{(3)} f={\mathscr N}(B)$.

The smoothness statement for the $d=3$  case was already established in~\cite{GZ}, so we must now consider the operator $\delta^{(2_3)}$. The same considerations show that this operator now equals $\frac{\partial^2}{\partial r^2}$ and therefore extracts $2!$ times the coefficient ${\mathscr N}_{\bar h}$ of the $r^2$ term in the expansion~\nn{pollywantacracker}.
\end{proof}

\begin{remark}
When $d$ is odd, solutions $f$ to Problem~\ref{DNPE} may not be smooth on $M$. For the $d=3$ case,   the obstruction to  smoothness of solutions is  $\check \delta^{(2)} f^{\rm ext}=-\bar \Delta B=-{\sf P}_2 B$, where $f^{\rm ext}$ is {\it any} smooth extension of the boundary Dirichlet data $B$ to $M$.
Also, the vanishing results for the operators $\check \delta^{(k)}$ on solutions can be used to extract  
the coefficients in the expansion~\nn{pollywantacracker}, in particular, in odd dimensions $d\geq 5$ we have 
$$
f(x,r)= B(x) +\tfrac1{2!(d-3)} r^2 \bar \Delta B(x) + {\mathcal O}(r^4)+
 {\mathcal O}(r^{d-1}\log r)\, .
$$

When $d=5$, the invariant operator $\check \delta^{(3)}
=-3\hh \bar\square \circ \delta^{(1)}$ does annihilate solutions to 
Problem~\ref{DNPE}, however it only has transverse order one. Instead, Equation~\nn{35} 
suggests  the operator
$$\delta^{(3_5)}:=\hat n^a \hat n^b \hat n^c \nabla_a\nabla_b \nabla_c f+\tfrac{3}{2} H\bar \Delta f-(\bar \nabla^a H)\bar\nabla_a f\, .$$ 
In fact $\delta^{(3_5)}$ defines a
conformally invariant operator
acting on $C^\infty M \cap\ker \delta^{(1)}$ that also annihilates solutions.
\hfill$\blacklozenge$
\end{remark}

\noindent
Non-linear analogs of the scalar boundary operators for  the Poincar\'e--Einstein filling problem were given in~\cite{Blitzedagain}.
Non-linear 
 Yang--Mills analogs are given in Section~\nn{models}.

\subsection{Poincar\'e--Einstein Asymptotically Yang--Mills Connections}\label{PEYM}
We  now provide some technical results concerning asymptotically Yang--Mills connections.


\begin{lemma}\label{gradF}
Let $(M^d,\cc)$ be an  asymptotically Poincar\'e--Einstein structure with $d\geq 6$, equipped with an  asymptotically Yang--Mills connection $\nabla$. 
Then, for any $g\in \cc$, 
\begin{equation}\label{first}
\nabla_{\hat n} F_{nb}\stackrel\Sigma=\frac{1}{d-5}\, 
\bj_b
\, .
\end{equation}
Also
\begin{equation}
\label{nablanF}
\nabla_{\hat n} F_{ab}\stackrel\Sigma=\frac{2}{d-5}\,  \hat n_{[a}{\bj}_{b]}-2H \bar F_{ab}
\end{equation}
and
\begin{equation}\label{when}
\nabla^a F_{ab} \stackrel\Sigma= \frac{d-4}{d-5}
\hh \hh\hh \bj_b\, .
\end{equation}

\end{lemma}
\begin{proof}
Because $\nabla$ is asymptotically Yang--Mills,  we have $(d-4) F_{nb}= \sigma \nabla^a F_{ab} + {\mathcal O}(\sigma^{d-4})$, so acting on this relation with $\nabla_{n}$  we obtain 
$$
(d-4) \nabla_{ n} F_{nb}=\nabla_{ n}(\sigma \nabla^a F_{ab})
+{\mathcal O}(\sigma^{d-5})
\stackrel\Sigma=
\nabla^\top_a \bar F^a{}_b+
\hat n^a \nabla_{\hat n} F_{ab}
=\bj_b+\nabla_{\hat n} F_{nb}\, .
$$
The second equality above used Equation~\nn{I2} and that $F_{ab}\stackrel\Sigma=\bar F_{ab}$. The third equality relied on the tensor analog of Equation~\nn{top2bar}, Equation~\nn{nablann}
as well as $ \bar F_{\hat nb}=0$. This yields Equation~\nn{first}.

To obtain Equation~\nn{nablanF} we employ the Bianchi identity to compute
$$
(d-5) \nabla_n F_{ab}=-2(d-5)n^c \nabla_{[a} F_{b]c}
\stackrel\Sigma=-2(d-5)\nabla_{[a} F_{b]n}+2(d-5)H F_{ba}$$
$$=
2\hat n_{[a}\bj_{b]}
-2(d-5)H \bar F_{ab}\, .
$$
Here we used Equation~\nn{first} as well as $\nabla_a n_b\stackrel\Sigma = H g_{ab}$ and $\nabla_{[a}^\top F_{b]\hat n}^{\phantom \top}\stackrel\Sigma=0$.

Finally, Equation~\nn{when} follows from a simple computation
$$
(d-5)\nabla^a F_{ab}\stackrel\Sigma=(d-5)\nabla^\top_a F^a{}_b + (d-5)\hat n^a \nabla_{\hat n} F_{ab}
\stackrel\Sigma =(d-4)\bj_b\, ,
$$
where the last step used Equation~\nn{nablanF}.
\end{proof}

Identities involving two normal derivatives of the curvature of an asymptotically Yang--Mills connection are more involved and are recorded below.
\begin{lemma}\label{d666}
Let $(M^d,\cc)$ be an  asymptotically Poincar\'e--Einstein structure with $d\geq 7$, equipped with an  asymptotically Yang--Mills connection $\nabla$. 
Then, for any $g\in \cc$,
\begin{equation}\label{d6}
\nabla_{\hat n} \nabla^a F_{ab}
\stackrel\Sigma=
-(d-4)\Big(P_{\hat n}{}^c \bar F_{cb}
+\frac
{2}{d-5}\, H \bj_b
\Big)\, ,
\end{equation}
\begin{multline}\label{d66again}
\hat n^c \nabla_{\hat n}\nabla_c F_{ab}
\stackrel\Sigma=
 \frac2{d-5} 
\hh\bar \nabla_{[a}\bj_{b]}
-2\bar P_{[a}{}^c \bar F_{b]c}
+(2P_{\hat n\hat n}+5H^2)\bar F_{ab}
+2\hat n_{[a}\Big(
P_{\hat n}{}^c\bar F_{b]c}
-\frac{5}{d-5}\hh H
\bj_{b]}
\Big)
\, ,
\end{multline}
and
\begin{equation}\label{d66againagain}
\hat n^a\hat n^b \nabla_{\hat n}\nabla_a F_{bc}
\stackrel\Sigma=
P_{\hat n}{}^a\bar F_{ca}
-\frac{5}{d-5}\hh H
\bj_{c}
\, .
\end{equation}
\end{lemma}

\begin{proof}
To obtain Equation~\nn{d6} we compute as follows
 \begin{align*}
\nabla_{\hat n}\nabla^a F_{ab}\
\stackrel\Sigma=\ &
\nabla^a \nabla_{n} F_{ab}
-Ric_{\hat n}{}^c \bar F_{cb}
+R_{\hat n}{}^a{}_b{}^c \bar F_{ac}
-(\nabla^a n^c)\nabla_c F_{ab}\\
\stackrel\Sigma=\ &
-2 \nabla^a (n^c \nabla_{[a}F_{b]c})
-(d-3)P_{\hat n}{}^c \bar F_{cb}
-H \nabla^a F_{ab}\, .
\end{align*}
The second equality used Equations~\nn{Wn} and~\nn{nablan} as well as $\hat n^a \bar F_{ab}=0$.
Thus, because $n^c \nabla_{[a}F_{b]c}=\nabla_{[a}F_{b]n}
-(\nabla_{[a}n^c) F_{b]c}$ and $(d-4)F_{an}=\sigma\nabla^b F_{ba} +{\mathcal O}(\sigma^{d-4})$, we have in dimensions $d\geq 7$ that 
\begin{align*}
(d-4)(\nabla_{\hat n}+H)\nabla^a F_{ab}
\stackrel\Sigma=\ &
-2 \nabla^a  \nabla_{[a}(\sigma \nabla^c F_{b]c})
\\&
+(d-4)\Big(2(\nabla^a\nabla_{[a}n^c) F_{b]c}
+2H \nabla^a F_{ba}
-(d-3)P_{\hat n}{}^c \bar F_{cb}\Big)\, .
\end{align*}
Thus
\begin{align*}
(d-4)(\nabla_{\hat n}+3H)\nabla^a F_{ab}
\stackrel\Sigma=\ &
-2 \hat n^a  \nabla_{[a} \nabla^c F_{b]c}
-2 \nabla^a  (n_{[a} \nabla^c F_{b]c})
\\&
+(d-4)\Big(-2\big(
\hat n^a P_{[a}{}^c +\delta_{[a}^c P_{\hat n}{}^a\big)
 \bar F_{b]c}
-(d-3)P_{\hat n}{}^c \bar F_{cb}\Big)\\
\stackrel\Sigma=\ &
2\nabla_{\hat n} \nabla^a F_{ab}
+\hat n^a \nabla_b \nabla^c F_{ac}
+\hat n_b \nabla^a \nabla^c F_{ac}
-2 H \delta^a_{[a}
 \nabla^c F_{b]c}
 \\&
 -(d-4)(d-6)P_{\hat n}{}^c \bar F_{cb}
\, .
\end{align*}
Here the first equality relied on Equation~\nn{PPP} while the second employed Equation~\nn{nablan}.
Thence, using  $2\nabla^a \nabla^c F_{ac} = [F^{ac},F_{ac}]=0$, we have
\begin{multline*}
(d-6)\nabla_{\hat n} \nabla^a F_{ab}
+\hh\hat n_b \hat n^a \nabla_{\hat n} \nabla^c F_{ca}
+(2d-11) H \nabla^a F_{ab}
\\
\stackrel\Sigma=\ 
\hat n^a \nabla^\top_b \nabla^c F_{ac}
-(d-4)(d-6)P_{\hat n}{}^c \bar F_{cb}\, .
\end{multline*}
But
$
(d-4)\hh \hat n^b \nabla^a F_{ab}\stackrel\Sigma=
\nabla^a (\sigma \nabla^c F_{ac})
$
so $(d-3) \hat n^b \nabla^a F_{ab}\stackrel\Sigma=0$, whence contracting the above display with $\hat n^b$ we find 
$
\hat n^a \nabla_{\hat n}\nabla^c F_{ca}\stackrel\Sigma=0
$.
Thus  we learn that
\begin{align*}
(d-5)(d-6)\nabla_{\hat n} \nabla^a F_{ab}
&\stackrel\Sigma=\ 
(d-4)\big(- \hat n^a \nabla_b^\top {\bj}_a^{\rm ext}
-(2d-11)H \bj_b
-(d-5)(d-6)P_{\hat n}{}^c \bar F_{cb}\big)
\\
&\stackrel\Sigma=
-(d-4)(d-6)\big(
2H \bj_b
+(d-5)P_{\hat n}{}^c \bar F_{cb}\big)
\, .
\end{align*}
In first equality above we used Equation~\nn{when}. The second used that $\nabla^\top_a \hat n_b^{\rm ext}\stackrel\Sigma=\bar g_{ab} H$. Since $d\geq 7$, the claimed result holds.

\medskip
 
To establish Equation~\nn{d66} we compute that
 \begin{align*}
\nabla_n^2 F_{ab}\
\stackrel\Sigma=\ &
-2\nabla_{\hat n} (n^c \nabla_{[a}F_{b]c})\\
\stackrel\Sigma=\ &
-2\nabla_{\hat n} \big(
\nabla_{[a}  F_{b]n}-
(\nabla_{[a} n^c) F_{b]c}
\big)\\
\stackrel\Sigma=\ &
-\frac2{d-4} \nabla_{\hat n}\nabla_{[a}(\sigma \nabla^c F_{b]c})+
2(\nabla_{\hat n} 
\nabla_{[a} n^c) \bar F_{b]c}
+2H \nabla_{\hat n} F_{ba}\, .
\end{align*}
The first equality above  relies on the Bianchi identity for the curvature $F$ while the third employed the asymptotically Yang--Mills condition in dimensions $d\geq 7$ as well as  Equation~\nn{nablan}.
Hence 
 \begin{align*}
\nabla_n^2 F_{ab}
\stackrel\Sigma=\ &
-\frac2{d-4} \big(\nabla_{[a}\nabla^c F_{b]c}
+\nabla_{\hat n}(n_{[a} \nabla^c F_{b]c})
\big)
\\[1mm]&
-2(P_{[a}{}^c+P_{\hat n\hat n}\delta_{[a}^c ) \bar F_{b]c}
-2H \big(\frac{2}{d-5}\hat n_{[a}\bj_{b]} -2 H\bar F_{ab}
\big)
\\[2mm]
\stackrel\Sigma=\ &
\frac2{d-5} (\nabla^\top_{[a}-2H\hat n_{[a})\bj_{b]}
-\frac2{d-4} \hat n_{[a} (2\nabla_{\hat n}+H)\nabla^c F_{b]c}
\\&
-2\big(P_{[a}{}^c \bar F_{b]c}-
 P_{\hat n\hat n} \bar F_{ab}\big)
 +4 H^2 \bar F_{ab}
\\[2mm]
\stackrel\Sigma=\ &
\frac2{d-5} 
 \bar \nabla_{[a}\bj_{b]}
 -4\hat n_{[a} 
 \Big(P_{\hat n}{}^c \bar F_{|c|b]}
+\frac
{2}{d-5}\, H \bj_{b]}
\Big)
\\&
-2\big(P_{[a}{}^c \bar F_{b]c}-
 P_{\hat n\hat n} \bar F_{ab}\big)
 +4 H^2 \bar F_{ab}\\
 \stackrel\Sigma=\ &
 \frac2{d-5} 
\Big(\bar \nabla_{[a}\!-4\hh H\hat n_{[a}
\Big)\bj_{b]}
-2\big(\bar P_{[a}{}^c \bar F_{b]c}
- \hat n_{[a} P_{\hat n}{}^c \bar F_{b]c}
- P_{\hat n\hat n} \bar F_{ab}\big)
 +3 H^2 \bar F_{ab}\, .
 \end{align*}
 The first equality above used Equations~\nn{PP} and \nn{nablanF}. The second equality used Equation~\nn{when}. The third equality used Equations~\nn{top2bar},~\nn{when} and~\nn{d6}.
 The last equality relied on Equation~\nn{Fialkow}. Rearranging the above we have
\begin{multline}\label{d66}
\nabla_n^2 F_{ab}
\stackrel\Sigma=
 \frac2{d-5} 
\hh\bar \nabla_{[a}\bj_{b]}
-2\bar P_{[a}{}^c \bar F_{b]c}
+(2P_{\hat n\hat n}+3H^2)\bar F_{ab}
+2\hat n_{[a}\Big(
P_{\hat n}{}^c\bar F_{b]c}
-\frac{4}{d-5}\hh H
\bj_{b]}
\Big)
\, .
\end{multline}
An  application of Equations~\nn{nablann} and~\nn{nablanF}
to the left hand side above yields the quoted result. Finally, Equation~\nn{d66againagain} follows directly from Equation~\nn{d66again} upon contraction with a unit normal vector.
 
 \end{proof}

\bigskip

\section{Renormalization, Anomalies and Energies}\label{renorm}

Integrals over conformally compact manifolds 
run the risk of being ill-defined since the metric is singular along the boundary.
Nonetheless, useful information can often be 
extracted from the asymptotics of integrals over families of regions  suitably approaching the boundary: 
Let $(M,\cc,\sigma)$ be conformally compact  with singular metric 
$$
\go =  \frac{\bg}{ \sigma^2}\, ,
$$
on its interior $M_+$. We wish to  construct a one-parameter $\varepsilon>0$ family of  ``regulating manifolds'' $M_\varepsilon$,
for each of which $(M_\varepsilon,\cc,\sigma)$ is a Riemannian structure with non-singular metric $\go$. 
For that, given any true scale $0< \tau\in \Gamma(\ce M[1])$, we  define 
\begin{equation}\label{ME}
M_\varepsilon:=\{p\in M\, | \,  \sigma(p)/ \tau(p) > \varepsilon\}\subset M_+\, . 
\end{equation}
Observe that $0\leq  \sigma/ \tau \in C^\infty M $ and $\Sigma=\partial M$ is the zero locus of the smooth defining function~$\sigma/\tau$.
Moreover $\sigma - \tau\varepsilon$ is a defining density for $\Sigma_\varepsilon = \partial M_\varepsilon$. 
Note that
in the zero $\varepsilon$ limit we recover $M_+$ from $M_\varepsilon$. Also, the restriction $\bar \tau:= \tau|_\Sigma$ to $\Sigma$ of
the density~$\tau$ determines a boundary true scale, and in turn a metric~${\bar g}={\bar{\bm g}}/ \bar \tau\in \cc_\Sigma $, where $\bar{\bg}$ is the conformal metric on $\Sigma$.
We call the data~$ \tau$ a  {\it choice of regulator}.

\medskip
We are primarily interested in the asymptotics of  the Yang--Mills functional on $M$.
Thus we define the {\it regulated Yang--Mills action functional} determined by a true scale~$ \tau$,
\begin{equation}
\label{Sepsilon}
S^\varepsilon[A]  :=-\frac14 \int_{M^\varepsilon} 
\ext {\rm Vol}(\go)\,{\rm Tr}(
\go^{ab}\go^{cd} F_{ac}F_{bd})\, .
\end{equation}

\begin{lemma}
The regulated  Yang--Mills functional is  extremal with respect to
compactly supported variations of $A$ when
the conformally compact  Yang--Mills current  satisfies 
$$
 j_a[A]=0\, .
$$
\end{lemma}

\begin{proof}
First note that
$$
S^\varepsilon[A]  = -\frac14\int_{M^\varepsilon} 
\frac{\ext {\rm Vol}(\bm g)}{ \sigma^{d-4}}\,
\operatorname{Tr}(F^{ab} F_{ab})\, ,
$$
where indices are raised and lowered with the conformal metric $\bm g$. 
Let $A_t$ be a one-parameter family of gauge connections on $M$ such that $A_t=A$ everywhere outside a compactly supported region and such that $A_0=A$. Also define the endomorphism-valued  one-form
$$
\delta A:= \frac{\ext A_t}{\ext t}\Big|_{t=0}\, .
$$
Then
\begin{align*}
\frac{\ext S^\varepsilon[A_t]}{\ext t}
\Big|_{t=0} &=-
\frac12\int_{M^\varepsilon}
\frac{\ext {\rm Vol}(\bm g)}{ \sigma^{d-4}}\,
\operatorname{Tr}(F^{ab} \delta F_{ab})\\
&=
\int_{M^\varepsilon}
\ext {\rm Vol}(\bm g)\,\operatorname{Tr}\big(
\nabla_a ( \sigma^{4-d} F^{ab}) \delta A_b\big)\, .
\end{align*}
Here  the operator $\delta$ denotes
$\frac{\ext }{\ext t}
\hh\hh\pdot\hh\hh
\big|_{t=0} $. Note that $\delta F_{ab}=2\nabla_{[a} \delta A_{b]}$. The compact support assumption was used in the second line. 
The proof is completed by verifying that
\begin{equation}\label{howtogetj}
\nabla_a ( \sigma^{4-d} F^{ab})= \sigma^{3-d} j^b\, .
\end{equation}
\end{proof}


 We now prove  Theorem~\ref{Laurent}  giving the asymptotics of $S^\varepsilon[A]$  of Equation~\nn{Sepsilon}.

\begin{proof}[Proof of Theorem~\ref{Laurent}]
The proof employs the methods of~\cite{Vol}. 
To begin with, the regulated Yang--Mills action functional can be expressed in terms of an integral over $M$ whose integral is weighted by a Heaviside step function
\begin{equation}\label{HeavyReg}
S^\varepsilon[A]  = -\frac14
 \int_{\tilde M} 
\frac{\ext {\rm Vol}(\bm g)}{\sigma^{d-4}}\, \theta\Big(\frac{\sigma}{\tau} - \varepsilon\Big)\,
\operatorname{Tr}F^2\, .
\end{equation}
In the above, $\tilde M$ denotes an extension of $M$ beyond $\partial M= \Sigma$ by any collar neighborhood. This does not change the above integral since its support is on $M^\varepsilon$.
This technical maneuver is necessary for the following distribution-based computations, and is described in  detail in~\cite{Vol}.

Using the distributional identity 
$
\theta'(x)=\delta(x)$, we now study
\begin{align*}
-\varepsilon^{d-4}
\frac{\ext S^\varepsilon[A]}{\ext \varepsilon}&  = -\frac{\varepsilon^{d-4}}4\!\!
 \int_{\tilde M} 
\frac{\ext {\rm Vol}(\bm g)}{\sigma^{d-4}}\, \delta\Big(\frac{\sigma}{\tau} - \varepsilon\Big)
\operatorname{Tr}F^2
\\[2mm]
&=\:\:\;
-\frac14 \int_{\tilde M} 
\frac{\ext {\rm Vol}(\bm g)}{\tau^{d-5}}\, \delta\big({\sigma}-{\tau}  \varepsilon\big)\operatorname{Tr} F^2
\, .
 \end{align*}
To obtain the last equality recall that, for any test function $f$,  under an integral $f(x)\delta(x)=f(0)\delta(x)$, and also $\delta(x/f(x)) = f(x) \delta(x)$ so long as $f(x)>0$. Also note  this implies that a quantity such as~$\delta(\sigma)$ is a distribution-valued density of weight $-1$.

The right hand side of the above display is a regular function of $\varepsilon$ around $\varepsilon=0$, so we expand it as a Taylor series in $\varepsilon$ to obtain
\begin{align*}
 \int_{\tilde M} 
\frac{\ext {\rm Vol}(\bm g)}{\tau^{d-5}}\, \delta\big({\sigma}-{\tau}  \varepsilon\big)\,
\operatorname{Tr}F^2
&=
 \int_{\tilde M} 
\frac{\ext {\rm Vol}(\bm g)}{\tau^{d-5}}\, \delta\big({\sigma}\big)\,
\operatorname{Tr}F^2
\\[1mm]
&
+ \varepsilon\int_{\tilde M} 
\frac{\ext {\rm Vol}(\bm g)}{\tau^{d-6}}\, \delta'\big({\sigma}\big)\,
\operatorname{Tr}F^2+\cdots \\
&+ 
\frac{ \varepsilon^{d-5}}{(d-5)!}\int_{\tilde M} 
\ext {\rm Vol}(\bm g)\, \delta^{(d-5)}\big({\sigma}\big)\,
\operatorname{Tr}F^2
+{\mathcal O}\big(\varepsilon^{d-4}\big)\, .
\end{align*}
The following identity for differentiated delta functions of a defining density is proved in~\cite{Vol},
$$
\Big(\frac1{I^2}\hh I.D\Big)^\ell
\delta(\sigma)= (d-\ell-1)\cdots(d-3)(d-2)\delta^{(\ell)}(\sigma)\, ,\qquad \ell\leq d-1\, .
$$
In the same reference, it is shown that the Laplace--Robin operator 
is formally self-adjoint. 
Moreover, if $ f$ is any weight $-d+1$ density, then 
\begin{equation}
\label{b2b}
\int_{\tilde M} \ext {\rm Vol}(\bm g) \delta(\sigma)\sqrt{I^2} \hh  f=
\int_\Sigma \ext {\rm Vol}(\bm {\bar g})  f\, .
\end{equation}
Orchestrating the above observations, 
it follows that
\begin{multline*}
 \int_{\tilde M} 
\frac{\ext {F^A\rm Vol}(\bm g)}{\tau^{d-5}}\, \delta\big({\sigma}-{\tau}  \varepsilon\big)\,
\operatorname{Tr}F^2=\\
 \int_{\Sigma} 
\frac{\ext {\rm Vol}(\bm {\bar g})}{\sqrt{I^2}\hh \tau^{d-5}}\,
\operatorname{Tr}F^2
+ \frac{\varepsilon}{d-2}\int_{\Sigma} 
\frac{\ext {\rm Vol}(\bm {\bar g})}{\sqrt{I^2}}\, 
I.D\Big(
\frac{\operatorname{Tr}F^2}{I^2 \hh \tau^{d-6}}
\Big)
+\cdots\qquad \qquad\\+ 
\frac{ 3!\, \varepsilon^{d-5}}{(d-5)!(d-2)!}\int_{\Sigma} 
\ext {\rm Vol}(\bm {\bar g}){\sqrt{I^2}}\, \hh
\Big(\frac1{I^2} I.D\Big)^{d-5}\,
\Big(\frac{\operatorname{Tr}F^2}{I^2}\Big)
+{\mathcal O}\big(\varepsilon^{d-4}\big)\, .
\end{multline*}
Notice that the last-displayed integral does not depend on the choice of regulator $\tau$ and hence is an invariant of the underlying structure.  Dividing the above display by $-\varepsilon^{d-4}$ and then integrating over $\varepsilon$ gives the quoted result.
\end{proof}

In quantum field theory vernacular, the coefficients $v_\ell/\ell$ of the $1/\varepsilon^\ell$ poles above are termed {\it divergences}. Theorem~\ref{Laurent} establishes that the divergences of the regulated Yang--Mills functional are boundary integrals over local quantities. They are however dependent on the {\it choice of regulator} ${\tau}$. 
The $\varepsilon$ independent contribution $S^{\rm ren}$ is termed the {\it renormalized action} and is not expressed as a
local boundary integral. Instead it probes global aspects of the system. Theorems~\ref{wehaveenergy} and~\ref{evenvarySren} give its dependence on the choice of regulator. The proofs are as follows:

\begin{proof}[Proof of Theorem~\ref{wehaveenergy}]
We begin by considering the regulated Yang--Mills functional for the choice of regulator $e^{t\omega}\tau$ where $t\in {\mathbb R}$ and $\omega\in C^\infty M$. Using Equation~\nn{HeavyReg}, this is given by
$$
S^\varepsilon[A;e^{t\omega}\tau]  =-\frac14 
 \int_{\tilde M} 
\frac{\ext {\rm Vol}(\bm g)}{\sigma^{d-4}}\, \theta\Big(\frac{\sigma}{\tau} - e^{-t\omega}\varepsilon\Big)\,
\operatorname{Tr}F^2\, .
$$
Differentiating with respect to $t$ and recycling the
computation used to prove  Theorem~\ref{Laurent} gives
$$
\frac{\ext S^\varepsilon[A;e^{t\omega}\tau]}{\ext t} 
=-\frac1{4\hh\varepsilon^{d-5}}\hh
 \int_{\tilde M} 
\frac{\ext {\rm Vol}(\bm g)}{\tau^{d-4}}\, \delta\Big(\frac{\sigma}{\tau} - e^{-t\omega}\varepsilon\Big)\,\omega e^{(d-5)t\omega}
\operatorname{Tr}F^2\, .
$$
The above can be expressed as a Laurent series in $\varepsilon$. We need  to compute the coefficient of $\varepsilon^0$ to study how $S^{\rm ren}[A]$ responds to changes of the regulator.
This is given by 
\begin{multline*}
\frac1{(d-5)!}\hh
 \int_{\tilde M} 
\ext {\rm Vol}(\bm g)\, 
\delta^{(d-5)}\big({\sigma}\big)\,\omega \hh
\operatorname{Tr}F^2
\\
=\frac{3!}{(d-5)!(d-2)!}\int_\Sigma 
\ext {\rm Vol}(\bm {\bar g})\, {\sqrt{I^2}}\, \hh
\Big(\frac1{I^2} I.D\Big)^{d-5}\,\Big(\frac{
\operatorname{Tr}F^2}{I^2}\hh\omega\Big)\, .
\end{multline*}
The result follows upon observing that this is independent of $t$.
\end{proof}


\begin{proof}[Proof of Theorem~\ref{evenvarySren}]
This is an direct consequence of the evenness property of the Feffer\-man--Graham type expansions developed in Section~\ref{connexp}; further details are discussed there.
\end{proof}

\begin{remark}
Theorem~\ref{wehaveenergy} can also be easily extended to  dimension $d=5$ when the expansion in Theorem~\ref{Laurent} has no negative powers of $\varepsilon$ and instead $\log \varepsilon$ is its leading term.
Specializing 
Theorem~\ref{wehaveenergy}
to constant $\omega=\log \Lambda$ (say) one has
$$
S^{\rm ren}[A;\Lambda \tau]-S^{\rm ren}[A;\tau]=\Lambda\,  {\rm En}[A]\, .
$$
Hence, the energy functional ${\rm En}[A]$ is termed an {\it anomaly} since it encodes the anomalous response of the renormalized action to
constant rescalings of the metric $g_{\scalebox{.7}{$\tau$}}:={\bm g}/{\tau^2}\in \cc$ determined by the regulator $\tau$.
\hfill$\blacklozenge$
\end{remark}

Another key property of the energy ${\rm En}[A]$ 
is the relation between its functional gradient and 
the obstruction to smoothly solving the Yang--Mills equations on the conformally compact structure $(M,\cc,\sigma)$ given in Theorem~\ref{fruit=vary}, whose proof is as follows:

\begin{proof}[Proof of Theorem~\ref{fruit=vary}]
Without loss of generality, we may assume that $A[\bar A_t]$ has compact support in the extended manifold $\tilde M$ (as defined in the proof of Theorem~\ref{Laurent}) because~$A[{\bar A_t}]$ is, in any case, determined only asymptotically. Also, since Theorem~\ref{gaugeeq} shows that any differing  asymptotic Yang-Mills connections are related by a gauge transformation (at least to sufficiently high order), the energy $\operatorname{En}[A[\bar A_t]]$ is independent of this difference.

Now, 
from Theorem~\ref{Laurent}, the coefficient of $\log \frac1\varepsilon$ in  $ \ext S^\varepsilon\big[A[\bar A_t]\big]/\ext t\big|_{t=0}$ 
equals the variation of the energy ${\rm  En}\big[A[{\bar A}]\big]$.
Let us denote $\delta A:=\ext A\big[{\bar A}]/\ext t\big|_{t=0}$. Hence, starting from Equation~\nn{HeavyReg}, we compute
\begin{align*}
\ext S^\varepsilon\big[ A[\bar A_t ]\big]/\ext t\big|_{t=0}
&=
-\int_{\tilde M} 
\frac{\ext {\rm Vol}(\bm g)}{\sigma^{d-4}}\, \theta\Big(\frac{\sigma}{\tau} - \varepsilon\Big)\,\operatorname{Tr}(
F^{ab}\, \nabla_a \delta A_b)
\\&
=
 \int_{\tilde M} 
\frac{\ext {\rm Vol}(\bm g)}{\sigma}\, \theta\Big(\frac{\sigma}{\tau} - \varepsilon\Big)\, 
\operatorname{Tr}(k^b \,
\delta A_b)
\\&
+\frac{1}{\varepsilon^{d-4}}
\int_{\tilde M} 
\frac{\ext {\rm Vol}(\bm g)}{\tau^{d-5}}\, \delta\big({\sigma}-{\tau}  \varepsilon\big)\, \operatorname{Tr}\Big(F^{ab}\,\Big(\nabla_a \hh \frac{\sigma}{\tau}\Big)\hh 
\delta A_b\Big)\, ,
\end{align*}
where again indices are raised and lowered with
the conformal metric.
Also $\nabla$ is the connection of $A[{\bar A}]$. In the second line we integrated by parts discarding boundary terms thanks to  the compact support of $A[{\bar A}_t]$, and then used Equations~(\ref{howtogetj},\ref{obst}) 
plus the distributional calculus identities explained earlier.

The second term of the above display is easily seen to be regular in $\varepsilon$ when multiplied by~$\varepsilon^{d-4}$, and therefore cannot produce a term proportional to $\log\varepsilon$. The small~$\varepsilon$ asymptotics of first term on the right hand side of the above display  can be determined
by exactly the same methods  used to establish 
Theorem~\ref{Laurent}. Its $\log\frac1\varepsilon$ coefficient can be computed 
from that of the $-\frac1\varepsilon$ contribution to  the first $\varepsilon$ derivative of this term. This gives
$$
 -\int_{\tilde M} 
{\ext {\rm Vol}(\bm g)}\, \delta\big({\sigma}\big)\, \operatorname{Tr}(k^a \,
\delta A_a)=-\int_\Sigma
\ext {\rm Vol}(\bm {\bar g})\, \frac{\operatorname{Tr}(k_a\, \hh 
\delta A^a_\Sigma)}{\sqrt{I^2}}
$$
as the coefficient of the $\log\frac1\varepsilon$ term. For the last  equality  above we employed
Equation~\nn{b2b} and  the fact that
$$
\delta {\bar A} :=\frac{ \ext {\bar A}}{\ext t}\Big|_{t=0}
=
\frac{\ext A\big[{\bar A}]}{\ext t}\Big|_{t=0,\Sigma}\, .
$$
This used that an asymptotically Yang--Mills connection obeys $A[{\bar A}]={\bar A}^{\rm ext} + {\mathcal O}(\sigma)$, where~${\bar A}^{\rm ext}$ is any smooth extension of ${\bar A}$ to $M$.
The display before last establishes that the obstruction current $\bar k=k|_\Sigma$ is proportional to the functional gradient of the energy ${\rm En}(A[{\bar A}])$.
\end{proof}


\begin{remark}
It is also possible to directly verify that the energy ${\rm En}[A]$ is  extremal with respect to arbitrary compactly supported variations of $A$ whenever the Yang--Mills current ${ j}[A]$ obeys 
$$
{ j}[A]={\mathcal O}(\sigma^{d-3})\, .
$$
To see this, note that from Theorem~\ref{Laurent}, in dimensions $d\geq 5$ the holographic energy formula in Equation~\nn{HolFor}
can  equivalently be written as $${\rm En}[A]=-
\frac{1}{4(d-5)!} 
\int_{\tilde M} {\ext {\rm Vol}(\bm g)}\delta^{(d-5)}(\sigma) 
\operatorname{Tr}F^2
\, .$$
Thus, because $\nabla_a \delta^{(k)}(\sigma)=n_a \delta^{(k+1)}(\sigma)$, it follows that
$$
\frac{\ext{\rm  En}\big[A_t\big]}{\ext t}\Big|_{t=0}=-\int_{\tilde M} {\ext {\rm Vol}(\bm g)}
\operatorname{Tr}\left(\Big[\delta^{(d-4)}(\sigma) \, n_a F^{ab} +\delta^{(d-5)}(\sigma) \hh \nabla_a F^{ab}\Big]\delta A_b\right)\, .
$$
The identity $\delta^{(d-5)}(\sigma)=-\frac1{d-4}\sigma\hh \delta^{(d-4)}(\sigma) $ then gives
$$
\frac{\ext{\rm  En}\big[A_t\big]}{\ext t}\Big|_{t=0}=\frac1{d-4}\int_{\tilde M} {\ext {\rm Vol}(\bm g)}\delta^{(d-4)}(\sigma) \, 
\operatorname{Tr}\big( { j}^a\delta A_a\big)\, .
$$
The claim follows by noting, for any positive integer $\ell$, that ${\sigma}^{\ell+1} \delta^{(\ell)}({\sigma})=0$.
\hfill$\blacklozenge$
\end{remark}

\subsection{The Energy}\label{theenergy}


Recursions that compute both the energy $\operatorname{En}[A[\bar A]]=:\overline{\operatorname{En}}[\bar A]$ and the obstruction 
can be constructed using only knowledge of
the existence of asymptotically Yang--Mills connections $A[\bar A]$ determined to sufficient order by $\bar A$.
The latter property was established in 
  Theorem~\ref{recurse}.
 Here we exhibit such recursions in dimensions $d=5,6,7$. 
There is no in-principle obstacle to performing these computations for any conformally compact structure, but for reasons of both simplicity and maximal relevance we focus on Poincar\'e--Einstein structures. 


\subsubsection{The Energy for $d=5,6,7$}
The holographic formula in Equation~\nn{HolFor} and identities of Section~\nn{PEM} allow us to compute the energy  for low dimensional Poincar\'e--Einstein structures.

\begin{proof}[Proof of Theorem~\ref{blabla}]
Theorems~\ref{recurse} and~\ref{gaugeeq}  establish that a boundary connection $\bar A$ uniquely determines an asymptotically Yang--Mills connection $A[\bar A]$ to order ${\mathcal O}(\sigma^{d-3})$. However it is not difficult, upon employing Lemma~\ref{sl2}, to check that $\Big(\frac1{I^2} I.D\Big)^{d-5}
\Big(\frac{-\operatorname{Tr}F^2}
{4I^2}\Big)$ only depends on the connection $A$ to order ${\mathcal O}(\sigma^{d-4})$. Gauge invariance of the holographic energy formula~\nn{HolFor} shows that ${\operatorname{En}}[A]$ is completely determined by the boundary data~$\bar A$.

In the remainder of the proof we rely heavily on the fact that an asymptotically Einstein structure 
dimension obeys
$\nabla_a n_b + \sigma P_{ab}+\rho g_{ab} = {\mathcal O}(\sigma^{d-2})$ and $I^2 \stackrel g= |n|_g^2 +2 \rho \sigma = 1 +{\mathcal O}(\sigma^{d-1})$.
Consulting Equation~\nn{HolFor}, we see that when $d=5$  
we have
 $$Q[{\bf g},A]= -\tfrac14 \operatorname{Tr} F^2|_\Sigma =  -\tfrac14
 \operatorname{Tr}
 {\bar F}^2\, . $$

For the case $d=6$ the same equation tells us to compute
$$
\frac 14\hh  I.D\operatorname{Tr} (F_{ab} F^{ab})
\stackrel\Sigma = -(\nabla_{\hat n} +4H)\operatorname{Tr}(F_{ab} F^{ab})
=\operatorname{Tr}\big(-2 \bar F^{ab}\nabla_{\hat n} F_{ab} - 4H {\bar F}^2\big)=0\, .
$$ 
In the last step we used Lemma~\ref{nablanF} and $F_{\hat n a}\stackrel\Sigma=0$, as well as the definition of the $I.D$ operator acting on the weight $-4$ density $\operatorname{Tr}(F_{ab} F^{ab})$ along $\Sigma$; see Equation~\nn{Robin}.

\smallskip
It remains to handle the case $d=7$ for which, 
again consulting Equation~\nn{HolFor}, we see that
we must  compute $\frac1{40}I.D^2 \operatorname{Tr} F^2$. 
Notice that  because $\operatorname{Tr}F^2$ is a density of weight~$-4$, from Equation~\nn{ID} we have (remembering $\nabla_n \sigma = |n|^2 \stackrel\Sigma=1$)
\begin{multline*}
I.D^2 \operatorname{Tr}F^2 \stackrel\Sigma=(d-12) (\nabla_n-5\rho) \big[(d-10)(\nabla_n-4\rho)-\sigma (\Delta -4 J)\big]\operatorname{Tr}F^2\\
\stackrel{d=7}=5\Big[\Delta-4J+3(\nabla_n-5\rho)(\nabla_n-4\rho)\Big] \operatorname{Tr}F^2\, .
\end{multline*}
Using $\rho\stackrel\Sigma= -H$ and $\nabla_n\rho\stackrel\Sigma=\Rho_{\hat n\hat n}$ (see Lemma~\ref{69} and its proof) we have
\begin{multline*}
\frac15
I.D^2 \operatorname{Tr}F^2 \stackrel\Sigma=
\Big[\Delta
+3\nabla_n^2
+27 H\nabla_n
 -12 \Rho_{\hat n\hat n}
+60H^2
-4J
\Big] \operatorname{Tr}F^2\\
\stackrel\Sigma=
\Big[\bar\Delta
+4\nabla_n^2
+32 H\nabla_n
 -12 \Rho_{\hat n\hat n}
+60H^2
-4J
\Big] \operatorname{Tr}F^2
\, .
\end{multline*}
In the above we used Lemma~\ref{lap2lap}
and the identity $\hat n^a  \hat n^b\nabla_{a} \nabla_b\stackrel\Sigma=\nabla_n^2  -H \nabla_n$ to relate interior and boundary Laplacians
according to 
$
\Delta f \stackrel\Sigma=\bar  \Delta \bar f +\nabla_n^2 f 
+5 H \nabla_n f
$.

Now, from Equation~\ref{nablanF} and $F\stackrel\Sigma = {\bar F}$, it follows that
$$
F^{ab}\nabla_n F_{ab}\stackrel\Sigma=
-2H{\bar F}^2\, .
$$
Similarly
$$
(\nabla_nF^{ab})\nabla_n F_{ab}
\stackrel\Sigma=
\frac12\hh {\bj}^2+
4H^2 {\bar F}^2\, .
$$
In addition note that
$
(\nabla_n F_{ab})^\top=
-2HF_{ab}^\Sigma
$,
and that (see Lemma~\ref{d6})
\begin{align*}
(\nabla_n^2 F_{ab})^\top
&\stackrel\Sigma=
\bar \nabla_{[a} {\bj}_{b]}
-2\bar P_{[a}{}^c \bar F_{b]c}
+2P_{\hat n\hat n} \bar F_{ab}
+3H^2 \bar F_{ab}
\, .
\end{align*}
Also Lemma~\ref{gradF} tells us that
$
\nabla^a F_{ab}\stackrel\Sigma=
\frac32\hh
\bj_b
$.

Orchestrating the above identities (using the  Ricci identity written in $d=7$ in terms of the Schouten tensor as in   Equation~\nn{Fialkow} and $J=\bar J + P_{\hat n\hat n} -3H^2$ from
 Equation~\nn{J2Jbar}), we have
\begin{align*}
\frac15
I.D^2 \operatorname{Tr}F^2 &\stackrel\Sigma=
\big(\bar \Delta
-16 P_{\hat n\hat n} 
+72H^2  -4\bar J \big)\operatorname{Tr}
{\bar F}^2
+4\nabla_n^2 \operatorname{Tr}F^2
+32 H\nabla_n \operatorname{Tr}F^2
\\
&
\stackrel\Sigma=
\big(\bar \Delta 
  -4\bar J \big)\operatorname{Tr}
{\bar F}^2
+4\operatorname{Tr}{\bj}^2
+8\operatorname{Tr}({\bar F}^{ab}
\bar \nabla_{[a} {\bj}_{b]})
-16\operatorname{Tr}({\bar F}^{ab}\bar P_{a}{}^c \bar F_{bc})\, .
\end{align*}

Finally, vanishing of $ \operatorname{En}[\bar A]$ in dimension $d=6$ follows directly from that of $Q[\bg,A]$. The vanishing property $ \operatorname{En}[\bar A]$ for $d\in\{8,10,\ldots\}$ follows from a simple evenness argument based on a  canonical  expansion of the connection; details are presented in Section~\ref{connexp}. (Note that  the same method can likely be used to establish vanishing of the energy integrand~$ Q[\bg,A]=0$ in those cases too.)
\end{proof}

\begin{remark}
Since  $\partial\Sigma=\emptyset$, upon integrations by parts, the $d=7$ energy integral takes the simple form
$$
\overline{\operatorname{En}}[\bar A]=
\frac18
\int_\Sigma {\ext {\rm Vol}(\bm {\bar g})}
 \operatorname{Tr} \big( 
{\bj}^2
+\bar J 
{\bar F}^2
+4{\bar F}^{ab}\bar P_{a}{}^c \bar F_{bc}
\big)\, .
$$
\hfill$\blacklozenge$
\end{remark}

\medskip

\subsection{The Obstruction}\label{obstruct}

Our  energy computations allow us to 
 employ the variational result  Theorem~\ref{fruit=vary} to efficiently compute the obstruction current~$\bar k$
of Equation~\nn{obcu}.
The for this result was recorded in Theorem~\ref{obsts} which is proved below.

\begin{proof}[Proof of Theorem~\ref{obsts}]
The first three results quoted (save for their overall normalizations) follow from  simple variational calculations: in $d=5$ the boundary energy is the standard Yang--Mills functional, when $d=6$ the energy vanishes. The $d=7$ case follows from a straightforward variational computation that we leave  for the reader. (We give an alternate recursion-based method for computing general obstruction currents and their normalizations by constants below.) The vanishing of $\bar k$ in $d=8,10,12,\ldots$ follows from vanishing of the energy in those dimensions as established in Theorem~\ref{blabla} along with
 the variational property of the obstruction current given in Theorem~\ref{fruit=vary}.
A direct proof based on  evenness properties of  the connection expansion is discussed 
in Section~\ref{connexp}. 
\end{proof}

\begin{remark}
Since the obstruction current is variational, a simple check of the above results is that the  divergence of the obstruction current $\bar \nabla^a \bar k_a$ vanishes identically.
\hfill$\blacklozenge$
\end{remark}

We now develop an algorithm for computing the obstruction current---in principle for  any dimension $d$---without relying on knowledge of the energy $\overline{\operatorname{En}}[\bar A]$. The basic idea is a direct generalization of the recursion developed in~\cite{Will2} for computing the obstruction to solving the singular Yamabe problem of finding a defining density such that the corresponding scalar curvature $Sc^{\go}=-d(d-1)$. To begin with, we observe from Equation~\nn{obst} that calculating the obstruction current $\bar k$ requires computation of   $d-4$ normal derivatives $\nabla_n^{d-4} j$ of the conformally compact Yang--Mills current along $\Sigma$. That requires knowledge of successive normal derivatives of the asymptotically Yang--Mills connection or its curvature. The following recursion is organized order by order in the transverse degree of the jets of the curvature and the dimension; see~\cite{Blitz} for a detailed definition of the transverse degree of tensors on conformally compact structures. 
Note that~$F_{nb}$ and~$\nabla^a F_{ab}$ 
have transverse degrees one and two respectively.
We now give a key technical result for normal derivatives of $j$.
\begin{proposition}\label{predict}
Let  $(M^d,\cc,\sigma)$ be  asymptotically Poincar\'e--Einstein with $d\geq 5$, and~$\nabla^A$ an asymptotically Yang--Mills connection. Also  let $1\leq \ell \leq d-4$. Then the corresponding conformally compact Yang--Mills current $j$ obeys
\begin{equation}\label{innocuous}
\nabla_n^{\ell} j_b\stackrel\Sigma=(\ell-d+4) \nabla_n^{\ell-1} \nabla^a F_{ab}
+{\rm LTOTs}\, ,
\end{equation}
where ${\rm LTOTs}$ denotes terms of lower transverse degree.

\end{proposition}

\begin{proof}
 We first compute
$$
\nabla_n^{\ell} j_b=
\nabla^{\ell}_n \big( \sigma \nabla^a F_{ab} -(d-4) F_{nb}\big)
\stackrel\Sigma=
{\ell} \nabla_n^{{\ell}-1} \nabla^a F_{ab}
-(d-4) \nabla_n^{\ell} F_{nb} +{\rm LTOTs}\, .
$$
Here we used that $F_{nb}$ and  $\nabla^a F_{ab}$ 
have respective transverse degrees one and two and that $\nabla_n^\ell\circ \sigma \stackrel\Sigma = \ell \nabla_n^{\ell-1}$. Also, the operator $\nabla_n$ increases the transverse degree of a tensor by one.
 Now observe that
$$
\nabla_{\hat n} \nabla_n^{{\ell}-1} F_{nb}\stackrel \Sigma = 
\hat n^a \hat n^c \nabla_a \nabla_n^{{\ell}-1} F_{cb} +{\rm LTOTs}
\stackrel \Sigma = \nabla^a  \nabla_n^{{\ell}-1} F_{ab} 
-\nabla^{ \hh \top}_a  \nabla_n^{{\ell}-1} F^a{}_{b} +{\rm LTOTs}
\, .
$$
In the above we have used that $\hat n^a \hat n^b \stackrel\Sigma= \bar g^{ab}-g^{ab}$ and the operator identity $\bar g^{ab} \nabla_b \stackrel\Sigma = \nabla^\top_b$. Noting that $\nabla^{ \hh \top}_a  \nabla_n^{{\ell}-1} F^a{}_{b} $ has lower transverse degree than $\nabla^a  \nabla_n^{{\ell}-1} F_{ab}$ as well as simple algebra, gives the result.
\end{proof}

The  seemingly innocuous result above generates a recursion, because when $\ell= d-4$, the terms denoted ``LTOTs'' in Equation~\nn{innocuous} give $(d-4)!$ times the obstruction current. However, since these are of lower transverse degree they will have been computed in a previous step. For $\ell=1,\ldots, d-5$, the left hand side of  Equation~\nn{innocuous} vanishes because~$\nabla^A$ is asymptotically Yang--Mills and $\nabla_n^\ell \circ\sigma^{d-4} \stackrel\Sigma=0$ when $\ell<d-4$. In that case, the terms LTOTs give an expression for $\nabla_n^{\ell-1} \nabla^a F_{ab}$.
This information is precisely that needed to compute the LTOTs in the following step. 

\medskip

Let us now perform  the recursion explicitly to give a non-variational proof of the $d=5,6,7$ results of Theorem~\ref{obsts}.
Although the following lemma is a special case of Proposition~\nn{obsts}, we separate it from the $d=6,7$ cases, to better illustrate the recursion.

\begin{lemma}\label{d=5O}
Let  $(M^d,\cc,\sigma)$ be  asymptotically Poincar\'e--Einstein with $d=5$ and~$\nabla^A$ an asymptotically Yang--Mills connection. Then
 the obstruction current $\bar k$
 equals the boundary Yang--Mills current~${\bj}$.
\end{lemma}
\begin{proof}
When $d=5$, Equation~\nn{obst} for an asymptotically Yang--Mills connection reads 
$$
\sigma \nabla^a F_{ab}-F_{nb}=\sigma k_b\, .
$$
The obstruction $\bar k$  is  $k|_\Sigma$. 
Hence, acting with $\nabla_n$ and restricting to $\Sigma$ we have
$$
\bar k_b\stackrel\Sigma= \nabla^a F_{ab} -\nabla_{\hat n} F_{nb}
 \stackrel\Sigma=
 \nabla^a F_{ab}+ (\bar g^{ac} - g^{ac})\nabla_a F_{cb}
 =\bar g^{ac} \nabla^\top_a F_{cb}\, ,
$$
where the second equality relies on  Equation~\nn{nablan} and $F_{nb}\stackrel\Sigma=0$.
The cancellation of terms proportional to $\nabla^a F_{ab}$ is precisely that predicted by Proposition~\nn{predict}. An application of Equation~\nn{top2bar} yields the result because $\bj_b=\bar \nabla^a \bar F_{ab}$.
\end{proof}

We now turn to the case $d=6$.

\begin{lemma}
Let  $(M^d,\cc,\sigma)$ be  asymptotically Poincar\'e--Einstein with $d=6$, and~$\nabla^A$ an asymptotically Yang--Mills connection. Then
 the obstruction current $\bar k$ vanishes.
 \end{lemma}


\begin{proof}
We are now in the $\ell=2$ case of Proposition~\ref{obsts} for which
$$
\bar k_b \stackrel\Sigma=\frac12\nabla_n^2 \big(\sigma \nabla^a F_{ab}-2 F_{nb}\big) 
\stackrel\Sigma= \nabla_{\hat n} \nabla^a F_{ab} 
+\frac12(\nabla_{\hat n} n^2)  \nabla^a F_{ab} - \nabla_{n}^2 F_{nb}\, .
$$
The first equality above uses Equation~\nn{obst}, while the second relies on $\nabla_n \sigma = n^2$.
Here  neither of Lemmas~\ref{d6} nor~\ref{d66} can be fruitfully  employed. Instead we manipulate the last term displayed above directly following the method of Proposition~\ref{obsts}:
\begin{align*}
\nabla_{\hat n} \nabla_n F_{nb}&\stackrel\Sigma=\ \hat n^a \hat n^c \nabla_a \nabla_{ n} F_{cb}+2H  \hat n^c \nabla_{\hat n} F_{cb}+ (\nabla_{ n}^2 n^c) \bar F_{cb}\\
&\stackrel\Sigma=\
(\nabla_a - \nabla_a^\top+2H\hat n_a)\hh  \nabla_{n} F^a{}_{b}
-P_{\hat n}{}^a \bar F_{ab}\, .
\end{align*}
Note that the first equality above used Equation~\nn{nablann} while the second relied on Equation~\nn{nablannn}.
Thus
\begin{align*}
\bar k_b\stackrel\Sigma=\ & [\nabla_{n}, \nabla^a]F_{ab}
+\nabla^{a\top}\big(\hat n_a \bj_b
-\hat n_b \bj_a -2H\bar F_{ab})
  +P_{\hat n}{}^a \bar F_{ab}
\\
 \stackrel\Sigma=\ &
 -Ric_{\hat n}{}^a {\bar F}_{ab}
 +R_{\hat n}{}^a{}_b{}^c \bar F_{ac}\!
 -(\nabla^a n^c)\nabla_c F_{ab}+2H \bj_b \!
 -\hat n_b \bar \nabla^a \bj_a\!
 -2(\bar \nabla^a H) \bar F_{ab} \!
 + P_{\hat n}{}^a \bar F_{ab}
 \\
 \stackrel\Sigma=\ &
-2(\bar \nabla^a H+P_{\hat n}{}^a) \bar F_{ab}
=0\, .
\end{align*}
The first equality above used Equations~(\ref{nablanF},\ref{when}) 
and $\nabla_{\hat n} n^2\stackrel\Sigma = 2H$ (see Equation~\nn{nablann}). The second relied on Equation~\nn{nablan2H}.
The penultimate equality needed the standard relation between Ricci and Schouten tensors, Equation~\nn{Wn}, Equation~\nn{nablan} and the Bianchi identity $\bar \nabla^a \bj_a =0$.
The last equality used the identity $P_{\hat n a}^\top\stackrel\Sigma=-\bar\nabla_a H$, valid for Poincar\'e--Einstein embeddings; see Equation~\nn{mainardi}.
\end{proof}

The $d=7$ case is considerably more involved but provides a nice check of the machinery developed here.

\begin{lemma}
Let  $(M^d,\cc,\sigma)$ be  asymptotically Poincar\'e--Einstein with $d=7$, and~$\nabla^A$ an asymptotically Yang--Mills connection. Then
 the obstruction current is given by
 $$\bar k=\frac12\bar \nabla^a
\big(
 \bar \nabla_{[a} {\bj}_{b]} 
 - 4\bar P_{[a}{}^{c} \bar F_{b]c}
-\bar J \bar F_{ab}
\big)
+\frac14 [\bj^a,\bar F_{ab}] 
 \, .
 $$
 \end{lemma}

\begin{proof}
Dimension $d=7$ is the  $\ell=3$ case of Proposition~\ref{obsts} so we must compute $\frac1{3!}\nabla_n^3 \big( \sigma \nabla^a F_{ab} -3 F_{nb}\big)$ along $\Sigma$, and thus begin with the first term
\begin{align*}
\nabla_n^3 \big( \sigma \nabla^a F_{ab} \big)
\stackrel\Sigma=&\ 
3\nabla_n^2 \nabla^a F_{ab}
-2\nabla_n( \rho\sigma \nabla_n \nabla^a F_{ab})
-2\nabla_n^2( \rho\sigma  \nabla^a F_{ab})\\
\stackrel\Sigma=&\ 
3\nabla_n^2 \nabla^a F_{ab}
+6H \nabla_{\hat n} \nabla^a F_{ab}
-2\big(\nabla_n^2( \rho\sigma) \big)  \nabla^a F_{ab}
\\
\stackrel\Sigma=&\ 
3\nabla_n^2 \nabla^a F_{ab}
-18H P_{\hat n}{}^a \bar F_{ab}
-12 H^2 \bj_b-6P_{\hat n \hat n}\bj_b
\, .
\end{align*}
The first line is a simple application of the Leibniz rule and $\nabla_n \sigma = n^2 = 1-2\rho \sigma + {\mathcal O}(\sigma^5)$, while the second line uses that $\rho\stackrel\Sigma=-H$.
For the third line  we used that
$\nabla_n^2(\rho\sigma)\stackrel\Sigma= \nabla_n \rho + \nabla_n(n^2 \rho)
\stackrel\Sigma=2P_{\hat n\hat n}-2H^2$ (see Equation~\nn{nablanrho}).

\smallskip
Now we must handle three normal derivatives of $F_{nb}$. Again, the key idea is that of Theorem~\ref{obsts}, namely to replace a pair of unit conormals with the difference of ambient and hypersurface metrics:
\begin{align*}
-3\nabla_n^3  F_{nb}\stackrel\Sigma=&\ 
-3\hat n^a \hat n^c \nabla_a\nabla_n^2 F_{cb}-3\big[\nabla_n^3,n^c\big] F_{cb}\\
\stackrel\Sigma=&\ 
-3 \nabla^a \nabla_n^2 F_{ab} + 3 \nabla_a^\top \nabla_n^2 F^a{}_b
-3 (\nabla_n^3 n^c) \bar F_{cb}
-9(\nabla_n^2 n^c)\nabla_{\hat n} F_{cb}
-9 \hat n^c H \nabla_n^2 F_{cb}
\\[1mm]
\stackrel\Sigma=&\ 
-3 \nabla^a \nabla_n^2 F_{ab} 
+3 \bar \nabla^a \big(
\hh\bar \nabla_{[a}\bj_{b]}
-2\bar P_{[a}{}^c \bar F_{b]c}
+(2P_{\hat n\hat n}+3H^2)\bar F_{ab}
\big)
\\&\
+15H \big(P_{\hat n}{}^c \bar F_{bc}-2H\bj_b\big)
-3\hat n_b \bar \nabla^a \big(P_{\hat n}{}^c \bar F_{ac}-2H\bj_a\big)
\\&\
+3\big(
2(\bar \nabla^a P_{\hat n\hat n}) +H P_{\hat n}{}^a\big)\bar F_{ab}
-9\Big(\tfrac12\big(H^2-2P_{\hat n\hat n})\bj_b
+\tfrac12 \hat n_b P_{\hat n}{}^a \bj_a
+2HP_{\hat n}{}^a \bar F_{ab}\Big)
\\&\
-9H\big(P_{\hat n}{}^a \bar F_{ba}-2H\bj_b\big)
\\[1mm]
\stackrel\Sigma=&\ 
-3 \nabla^a \nabla_n^2 F_{ab} 
+3 \bar \nabla^a \big(
\hh\bar \nabla_{[a}\bj_{b]}
-2\bar P_{[a}{}^c \bar F_{b]c}\big)
\\&\
+12 (\bar \nabla^a P_{\hat n\hat n}) \bar F_{ab}
+15  P_{\hat n\hat n} \bj_b
-39H P_{\hat n}{}^a \bar F_{ab}
-\frac{15}2 H^2 \bj_b
-\frac{27}2\hat n_b P_{\hat n}{}^a \bj_a  
\, .
\end{align*}
The second line above used Equation~\nn{nablann}.
To achieve the third line, we observe that Equation~\nn{d66}
may be employed under a tangential derivative $\nabla^\top$. That line also heavily relies on Lemma~\ref{69} as well as Equation~\nn{nablanF}.
The last line above uses that the divergence of the boundary curvature $\bar F$ gives the current $\bar j$ as well as Equation~\ref{mainardi}.

Now, by writing $-3 \nabla^a\nabla_n^2 F_{ab}\stackrel\Sigma=\ 
-3\nabla_n^2 \nabla^a F_{ab}
-3\big[\nabla^a, \nabla_n^2] F_{ab}$ all of  the $\nabla_n^2 \nabla^a F_{ab}$ terms of highest transverse degree cancel (again this is the general phenomenon established in Theorem~\ref{obsts}) but we must still compute the lower order commutator terms:
\begin{align*}
-3\big[\nabla^a,& \nabla_n^2] F_{ab}
\stackrel\Sigma=\ 
-\hh 3Ric_{\hat n }{}^c \nabla_{\hat n} F_{cb} 
-3R^a{}_{\hat n b}{}^c \nabla_{\hat n} F_{ac}
-3(\nabla^a n^c) \nabla_c\nabla_n F_{ab}
\\&\qquad\qquad\:\:
-3\nabla_{\hat n} \big(Ric{}_{ n }{}^c  F_{cb} 
+R^a{}_{ n b}{}^c  F_{ac}
+\big[F^a{}_{n},F_{ab}\big]+(\nabla^a n^c)\nabla_c F_{ab}\big)\\
\stackrel\Sigma=&\ 
-6Ric_{\hat n }{}^c \nabla_{\hat n} F_{cb} 
-6R^a{}_{\hat n b}{}^c \nabla_{\hat n} F_{ac}
-6H \nabla^a\nabla_n F_{ab}
\\&\
-3(\nabla_{\hat n} Ric_{n}{}^c)  \bar F_{cb}
-3 (\nabla_{\hat n} R^a{}_{ n b}{}^c)  \bar F_{ac}
-3\big[(\nabla_{\hat n} F^a{}_{n}), \bar F_{ab}\big]
\\[1mm]&\ 
-3(\nabla_{\hat n} \nabla^a n^c)\nabla_c F_{ab}
-3H [\nabla_n,\nabla^a] F_{ab}
\\[1mm]
\stackrel\Sigma=&\ 
-3( 5P_{\hat n}{}^a+J \hat n^a)\big(\hat n_a \bj_b - \hat n_b \bj_a -4H\bar F_{ab}\big)
\\&\
-3 P_{\hat n}{}^c
\big( \hat n_b \bj_c - \hat n_c \bj_b -4H\bar F_{bc}\big) 
-3P^a{}_b \hat n^c 
\big( \hat n_a \bj_c - \hat n_c \bj_a -4H\bar F_{ac}\big)
\\&\
-3 H\nabla_a^\top\big(\hat n^a \bj_b -\hat n_b \bj^a -4H\bar F^a{}_{b})
-6H\big(P_{\hat n}{}^c \bar F_{bc}-2H\bj_b\big)\\&\
-15H P_{\hat n}{}^c \bar F_{cb}
-15\hat n^a (\nabla_{\hat n} P_a{}^c) \bar F_{cb}
-3 H P_{\hat n}{}^c\bar F_{bc}
-3 \hat n^d (\nabla_{\hat n} R^a{}_{db}{}^c)\bar F_{ac}
\\&\
+\frac32 \big[ \bj^a,\bar F_{ab}\big]
\!+\!3\big(P^{ac}\!+\!g^{ac}P_{\hat n\hat n}\big)\nabla_c F_{ab}
+3H Ric_{\hat n}{}^a \bar F_{ab} 
\!-\!3H R_{\hat n}{}^a{}_b{}^c \bar F_{ac}\!
+\frac92H^2\bj_b
\\[1mm]
\stackrel\Sigma=&\ 
-3\bar J \hh \bj_b 
+3 \bar P^a{}_b \bj_a
+\frac32 \big[ \bj^a,\bar F_{ab}\big]
-\frac{21}2
P_{\hat n\hat n} \bj_b
+42
H P_{\hat n}{}^a \bar F_{ab}
+21
H^2 \bj_b
\\&\
+15\hat n_b
P_{\hat n}{}^a \bj_a
-15 \hat n^a (\nabla_{\hat n} P_a{}^c)\bar F_{cb}
-3 \hat n^d (\nabla_{\hat n} R^a{}_{db}{}^c)\bar F_{ac}
+3P^{ac} \nabla_c F_{ab}\, .
\end{align*}
The first line above used the Leibniz property of the operator $[\nabla_n,\pdot\hh]$ while the second line distributed the operator $\nabla_{\hat n}$ and used that $F|_\Sigma=\bar F$. The third equality used Equations~(\ref{first}, \ref{nablanF}, \ref{d66}, \ref{PP}), the results of which are collated in the last line.
It remains to develop  the terms above involving normal derivatives of ambient curvatures. Firstly 
$$
\hat n^a (\nabla_{\hat n} P_a{}^c)\bar F_{cb}
\stackrel\Sigma=
(\bar\nabla^c J)\bar F_{cb}
-(\nabla_a^\top P^{ac} )\bar F_{cb}
\stackrel\Sigma=
-2HP_{\hat n}{}^a \bar F_{ab}
+(\bar \nabla^a P_{\hat n\hat n}) \bar F_{ab}\, .
$$
Here again we have used the usual maneuver for handling pairs of conormals as well as Equations~\nn{Fialkow} and~\nn{mainardi}.
Along similar lines, we also compute
\begin{align*}
\hat n^d (\nabla_{\hat n} R^a{}_{db}{}^c)\bar F_{ac}
\ \stackrel\Sigma=\
&
(
\nabla^c Ric^a{}_b
) \bar F_{ac}
-
(\nabla_d^\top R^{ad}{}_{b}{}^c) \bar F_{ac}
\\
\stackrel\Sigma=\ &
\frac52 C^{ca}{}_b \bar F_{ac}+
\big(\bar \nabla^a (\bar J+P_{\hat n\hat n}-3 H^2)\big) \bar F_{ba}
-(\nabla^\top_d W^{ad}{}_b{}^c)\bar F_{ac}
\\&\
-(\nabla^\top_d P^{dc})\bar F_{bc}
-(\nabla^\top_c P^{a}{}_b)\bar F_{ac}
\\
\stackrel\Sigma=\ &-
\frac{5}2
\bar C_{acb} \bar F^{ac}
+\frac52 \hat n_b C_{ac\hat n} \bar F^{ac}
-(\bar \nabla^a \bar J)\bar F_{ab}
-(\bar \nabla^a P_{\hat n\hat n})\bar F_{ab}
-6H P_{\hat n}{}^a \bar F_{ab}
\\&\
+3\bar C_{bca} 
\bar F^{ac}
+(\bar \nabla^a \bar J) \bar F_{ab}
+8H P_{\hat n}{}^a \bar F_{ab}
-\frac12 \bar C_{cab}\bar F^{ac}
+H P_{\hat n}{}^a \bar F_{ab}
\\&\
-\hat n_b (\bar \nabla_c P_{\hat n a}^\top)\bar F^{ac}
-H P_{\hat n}{}^a \bar F_{ab}
\\
\stackrel\Sigma=\ &
-\frac12
\bar C_{acb} \bar F^{ac}
-(\bar \nabla^a P_{\hat n\hat n})\bar F_{ab}
+2H P_{\hat n}{}^a \bar F_{ab}
\, .
\end{align*}
The second line above relied on the definitions of the Schouten and Cotton tensors; see  Equations~\nn{Schouten} and~\nn{Cotton}.
We also employed Equation~\nn{Fialkow} and the decomposition of the Riemann tensor into Weyl and Schouten tensors in Equation~\nn{Weyl}.
In the third line we used Equation~\nn{cotton},
as well as Equation~\nn{W2W} in combination with the result for the divergence of the Weyl tensor in Equation~\nn{divWeyl}. For the last line we used Equation~\nn{Mainardi} as well as the vanishing of $C_{ac\hat n}$ by virtue of Equations~\nn{Wn} and~\nn{divWeyl} (see also Lemma~\ref{thelemma}).
One further computation is needed
\begin{multline*}
P^{ac}\nabla_c F_{ab}\stackrel\Sigma=
P^{ac}\nabla_c^\top \bar F_{ab}+
P_{\hat n}{}^a \nabla_{\hat n}F_{ab}
\\
\stackrel\Sigma=
P^{ac}(\bar\nabla_{c} \bar F_{ab}-H\hat n_a \bar F_{cb})+P_{\hat n}{}^a
\big(\tfrac12 \hat n_a\bj_b -\tfrac12\hat n_b \bj_a -2 H\bar F_{ab}\big)\\
\stackrel\Sigma=
\bar P^{ac}\bar\nabla_c \bar F_{ab}
-\frac12 H^2 \bj_b
-3HP_{\hat n}{}^a \bar F_{ab}
+\frac12 P_{\hat n\hat n} \bj_b-\frac12\hat n_b P_{\hat n}{}^a \bj_a\, .
\end{multline*}
%
%
%
%
%
%
%
In the  above we decomposed $\nabla\stackrel\Sigma= \nabla^\top + \hat n \nabla_{\hat n}$ and subsequently employed Equation~\nn{nablanF} as well as Equation~\nn{top2bar} to convert $\nabla^\top$ to $\bar \nabla$ plus terms involving mean curvature. 
The above calculation is completed using Equation~\nn{Fialkow}.

\medskip
Orchestrating all these computations gives
\begin{multline*}
\quad\frac13 \nabla_n^3 \big( \sigma \nabla^a F_{ab} -3 F_{nb}\big)
\stackrel\Sigma=
 \bar \nabla^a \big(
\hh\bar \nabla_{[a}\bj_{b]}
-2\bar P_{[a}{}^c \bar F_{b]c}\big)
-\bar J \hh \bj_b 
+ \bar P^a{}_b \bj_a
+\frac12 \big[ \bj^a,\bar F_{ab}\big]
\\
+\frac{1}2
\bar C_{acb} \bar F^{ac}
+\bar P^{ac}\bar\nabla_c \bar F_{ab}\, .
\end{multline*}
It is not difficult to verify that
$$
\bar P^a{}_b \bj_a
+\frac{1}2
\bar C_{acb} \bar F^{ac}
+\bar P^{ac}\bar\nabla_c \bar F_{ab}=-2 \bar \nabla^a 
(\bar P_{[a}{}^c \bar F_{b]c})-(\bar \nabla^a \bar J)F_{ab}\, .
$$
Remembering that $\bj_b=\bar \nabla^a \bar F_{ab}$, it follows that $\frac1{3!} \nabla_n^3 j|_\Sigma$ equals the result quoted for the obstruction current $\bar k$.
\end{proof}

 \section{Boundary Operators  and Dirichlet-to-Neumann~Maps
}
\label{models}

 To construct Dirichlet-to-Neumann
maps for the Yang--Mills system, we rely on a canonical expansion of the connection generalizing those given for the Maxwell
and Poincar\'e--Einstein systems   respectively in   Sections~\ref{Maxwell} and~\ref{canexp}.
This expansion also allows us to complete the proofs of Theorems~\ref{blabla} and~\ref{obsts}.

\subsection{Canonical Connection Expansion}\label{connexp}

To study the asymptotics of conformally compact Yang--Mills connections on Poincar\'e--Einstein structures $(M,\cc,\sigma)$, 
we first express the conformally compact Yang--Mills equation~\nn{YM} in the Graham--Lee coordinate system  $(x,r)$ for a collar   neighborhood of $\Sigma$. Those coordinates are 
 predicated on a choice of boundary metric $\bar h\in \cc_\Sigma$; see Section~\ref{canexp}. Since the density $\sigma = [\ext r^2+h(x,r),r]$  in the  collar,  we work in the scale  determined by the choice of compactified metric $\ext r^2+h(x,r)$. 
We also choose  frames $\Phi$ for the bundle ${\mathcal V}M$ so that the connection $\nabla^A=\ext +A^\Phi$, where~$A^\Phi$ is an endomorphism-valued one-form.

It is propitious to make a ``temporal gauge choice'', meaning frames $\Phi$ such that $$
A^\Phi_0:=\iota_{\frac{\partial}{\partial r}} A^\Phi = A^\Phi_n=0\, .
$$
To see that such a choice is possible, note that
$$
 A^{U\Phi}_0 \sim U^{-1} U' + U^{-1}A^{\Phi}_0U\, ,
$$
where $U$ is any smooth endomorphism.
Also, primes  denote derivatives with respect to~$r$. Thus we need to solve the linear ordinary differential equation
$$
U' = - A^\Phi_0 U\, ,
$$
for $U$ given $A^\Phi_0$. We take as initial data for $U$ the condition $U|_{r=0}=\operatorname{Id}$. Clearly this system always has a solution for a  sufficiently small  collar neighborhood of the boundary.
Having fixed a ``choice of gauge'', from now on we drop the label $\Phi$ on the endomorphism-valued one-form $A$ trusting that no  confusion this with the same label $A$ used to denote the connection $\nabla^A$ will arise.

We now employ an index notation $i=1,\ldots, d-1$ for the coordinates $x=(x^i)$ and recycle the abstract indices $a,b,c,\ldots$ to denote the collar coordinates $y^a=(x^i,r)$. 
Then, thanks to the above gauge choice, the curvature two-form is simply
$$
F= \ext r \wedge A' + \hat F\, ,
$$
where $A=A_i(x,r) \ext x^i$ and $\hat F = \frac12\ext x^i\wedge \ext x^j  F_{ij}(x,r)$ with 
$$
 F_{ij}=\partial_i A_j(x,r)-
\partial_j A_i(x,r)+\big[A_i(x,r),A_j(x,r)\big]\, .
$$
In this scale and coordinate system, the covector $n=\ext r$, so
\begin{equation}\label{FGE}
F_{nb } \ext y^b = A'_j\ext x^j\, .
\end{equation}
The divergence of the curvature is given by
$$
(\nabla^a F_{ab}) \ext y^b=-\hat \nabla^i A'_i\, 
\ext r
+(
A''_j+
\hat \nabla^i \hat F_{ij} +\tfrac12h'^{i}{}_i A'_j-
h'^i{}_j A'_i) \ext x^j\, .
$$
Here we have denoted by $\hat \nabla$ the Yang--Mills coupled connection along constant $r$ hypersurfaces so that, for a one-form valued endomorphism $X_j(x,r)$, one has $$\hat \nabla_i X_j=\partial_i X_j - \hat \Gamma_{ij}^k X_k + [A_i,X_j]\, .$$
We raise and lower indices $i,j,k,\ldots$ with the hypersurface metric $h_{ij}(x,r)$.
Also we have used that the Christoffel symbols of the metric $\ext r^2 + h(x,r)$ are
$
\Gamma^0_{00}=0=\Gamma^0_{0j}=\Gamma_{00}^k
$, 
$\Gamma^k_{0j}=\frac12 h'^k{}_j$,
 $
\Gamma^0_{ij}=-\frac12 h'_{ij}
$
and
$
\Gamma^k_{ij}=\hat \Gamma^k_{ij}$.
Thus the conformally compact Yang--Mills equations given in~\nn{YMC} become 
\begin{align}
j_0[A]&:=\, -r \hat \nabla^i A'_i=0\, ,
\label{Ga}
\\[2mm]
j_j[A]&:=\, \Big(r\frac{\partial}{\partial r} -d+4\Big)  A'_j+
r \hat \nabla^i \hat F_{ij} 
+\tfrac12 r h'^{i}{}_i A'_j-
r h'^i{}_j A'_i=0\, .\label{amp}
\end{align}
Let us call the first of these the Gau\ss\ law because $A'_i$ is the Yang--Mills analog of an electric field.  We dub  the second equation the Amp\`ere law, again in analogy with electromagnetism.

\medskip

We now seek a formal solution in the form
\begin{equation}\label{EXPAND}
A_j(x,r) = A^{(0)}_j(x)+
r A^{(1)}_j(x)+
r^2 A^{(2)}_j(x)+
\cdots\, .
\end{equation}
In what follows, recall that $h(x,r)$ obeys the expansion given in Equation~\nn{FG}.
The Amp\`ere law at order $r^0$, when $d>4$, implies that
\begin{equation}\label{Istartodd}
A^{(1)}=0.
\end{equation}
It is known (see for example~\cite{FG-ast,GrSrni}) that the order $r^2$ term in the asymptotic expansion for the metric $g$ is given by
$$
h_{ij}^{(2)}(X)= -\hat P_{ij}(x)\, ,
$$
where $\hat P_{ij}$ is the Schouten tensor of the boundary metric $h^{(0)}(x)$ and $\hat J$  denotes $\hat P^i{}_i$. Note that $\hat P_{ij}(x)$ is the boundary Schouten tensor $\bar P_{ij}$ in the boundary coordinate system $x^i$.
Hence at order $r^k$, 
assuming an expansion of $A_i$ of the form~\nn{EXPAND},
the Amp\`ere Law~\nn{amp} gives 
\begin{multline}\label{ampexp}
0=(k+1)(k-d+4) A_j^{(k+1)}
+\big(\bar \Delta-(k-1) \bar J\big) A_j^{(k-1)}
\\ 
- \big(\bar \nabla^i \bar \nabla_j 
-2(k-1) \bar P^i{}_j\big)A_i^{(k-1)} 
+\big[A_i^{(k-1)},\bar F^i{}_j\big]
+\mbox{LOTs}\, ,
\end{multline}
where LOTs denotes terms involving $A^{(\ell)}$ with $\ell<k-1$, and quantities  with bars are evaluated along the boundary at $ r=0$.
From the above we  find that
$$A^{\rm ev}_i(x,r) = A^{(0)}_i(x)+
r A^{(1)}_i(x)+
r^2 A^{(2)}_i(x)+
\cdots +
r^{d-4} A_i^{(d-4)}(x)\, ,
$$
solving
\begin{equation}\label{ji}
j_i[A^{\rm ev}]={\mathcal O}(r^{d-4})\, ,
\end{equation}
is such that each $A^{(\ell)}_i(x)$ for $1\leq \ell\leq d-4$ is uniquely determined in terms of $A^{(0)}_i(x)$. Moreover, since the metric expansion of $h(x,r)$ is even to order $r^{d-2}$, by inspecting Equations~\nn{Istartodd} and~\nn{ampexp}, we see that the Amp\`ere law  inductively implies  that all terms $A^{(\ell)}(x)$ with $\ell\leq d-4$ and odd must vanish, and thus~$A_i^{\rm ev}(x,r)$ is an even polynomial function of $r$.

\smallskip
We still have to impose the Gau\ss\ law. For that we note that the Bianchi identity in the form of Equation~\nn{BI} applied to $j[A^{\rm ev}]$ in tandem with Equation~\nn{ji} gives 
\begin{equation}\label{apollo}
\Big(r\frac{\partial}{\partial r} -d+3\Big)
 j_0=-r \hat \nabla^i j_{i}
 -\frac12 r h^{ij} h'_{ij} j_0=
  -\frac12 r h^{ij} h'_{ij} j_0+{\mathcal O}(r^{d-3})
 \, ,
 \end{equation}
 where $j_0=-r \hat \nabla^i A'^{\rm ev}_i$ is an even polynomial in $r$ of degree at most $d-4$ plus terms of order~$r^{d-3}$ (potentially involving logs coming from the metric expansion). 
 A simple induction again shows that $j_0[A^{\rm ev}]={\mathcal O}(r^{d-3})$.
Hence we have established that~$A^{\rm ev}$ satisfies 
\begin{equation}\label{artemis}
j[A^{\rm ev}]={\mathcal O}(r^{d-4})\kappa(x,r)+ {\mathcal O}(r^{d-3})\hh \ext r\, ,
\end{equation}
where $\kappa$ is some smooth endomorphism-valued one-form  in the kernel of $\iota_{\frac{\partial}{\partial r}}$.  
Hence $
j_i[A^{\rm ev}]= r^{d-4} k_i(x) + {\mathcal O}(r^{d-3})
$ where $k_i(x)$ is determined by the boundary connection and choice of boundary metric $\bar h\in \cc_\Sigma$.
In~\cite{GrSrni} it was shown  that upon changing the choice of boundary  metric representative $\bar h$ to $\bar \Omega^2 \bar h$, the compactified metric~$g$ changes to $\Omega(x,r)^2 g$ where~$\Omega(x,r)$ has an even expansion in $r$  and obeys $\Omega(x,0)=\bar \Omega$. Also the new canonical coordinate is~$\Omega(x,r)\hh r$. It follows that, under a change of boundary metric representative,~$k_i$ indeed transforms as a weight $3-d$ density-valued one-form and thus gives the obstruction current $\bar k$ in the boundary scale determined by $\bar h$.
This is a special case of our general result Theorem~\ref{recurse} for conformally compact structures.
In other words,~$A^{\rm ev}$  is an asymptotically Yang--Mills connection expressed in a distinguished coordinate system and we  have now reproduced Theorem~\ref{recurse}, but with the additional information that the expansion
in the canonical coordinate $r$ is even up to order $r^{d-4}$ when the background is specialized to a Poincar\'e--Einstein structure.

Note also that the terms in
Equation~\nn{apollo} of order $r^{d-3}$ together with Equation~\nn{artemis} imply
$$
0=\bar \nabla^i \bar k_i\, .
$$
This establishes that the obstruction current $\bar k$ is divergence-free along the boundary, which also follows from our earlier result that it is the variation of a boundary energy functional; see Theorem~\ref{fruit=vary}. 

\bigskip

We can now give the details behind the evenness arguments appealed to in the   proofs of Theorems~\ref{evenvarySren},~\ref{blabla}, and~\ref{obsts}.  
Theorem~\ref{evenvarySren} states that the renormalized action $S^{\rm ren}[A]$ is conformally invariant for Poincar\'e--Einstein  structures.   Its detailed  proof mirrors exactly that given for renormalized volumes in~\cite{GrSrni}. Again we rely on the evenness property
of the new canonical coordinate~$\Omega(x,r)\hh r$ corresponding to 
changing the choice of boundary  metric representative $\bar h$ to $\bar \Omega^2 \bar h$. In particular 
 the function~$e^\omega$ giving the change of scale in Theorem~\ref{evenvarySren}
equals $\Omega(x,r)$. 
It is now easy to check that both the metric and connection expansions are even to 
 sufficiently high enough order that there can be  no change in~$S^{\rm ren}[A]$ when varying the choice of boundary metric $\bar h$, which yields the result stated in Theorem~\ref{evenvarySren}.

\smallskip

Theorems~\ref{blabla}, and~\ref{obsts} give vanishing results for the energy and obstruction in even boundary dimensions.
For those we must study the boundary asymptotics  of the regulated Yang--Mills energy functional
\begin{equation}\label{evenreg}
S^\varepsilon[A^{\rm ev}]  :=-\frac14 \int_{M^\varepsilon} 
\ext {\rm Vol}(\go)\,{{\rm Tr}}\big(
\go^{ab}\go^{cd} F^{A^{\rm ev}}_{ac}F^{A^{\rm ev}}_{bd}\big)\, ,
\end{equation}
evaluated on the asymptotic solution $\nabla^{A^{\rm ev}}$. We split the region into
$$
M_{\varepsilon}= C_{\varepsilon}\cup M_{\rm int}\, ,
$$
where $C_{\varepsilon}$ is the collar $\{p\in M\, |\, \varepsilon\leq r(p) \leq r^*\}$ and $\varepsilon<r^*\in {\mathbb R}$ is some sufficiently small value such that the coordinate $r$ is defined. 
Since $M$ and hence $\Sigma=\partial M$ are compact, 
 such a value $r^*$ exists. The region $M_{\rm int} := M\backslash(C_{\varepsilon}\cup \{p\in M\, |\, r(p)<\varepsilon\})$.
To connect with the density-based approach  to the regulated action functional in Section~\ref{renorm}, note that the density $\sigma$ evaluated in the scale~$g$ is given by $\sigma\stackrel g= r$. Hence Equation~\nn{ME} then determines the same region $M_{\varepsilon}$ as that given above. Therefore the coefficients in the Laurent series of Theorem~\ref{Laurent} equal those determined by the expansion made here.

\smallskip

 In the above terms we have
$$
S^\varepsilon=\int_{\varepsilon}^{r^*}\! \frac{\ext r }{r^d}  \, {\mathcal L}(r)
-\frac14\int_{M_{\rm int}} 
\ext {\rm Vol}(\go)\,
{\rm Tr}\big(
\go^{ab}\go^{cd} F^{A^{\rm ev}}_{ac}F^{A^{\rm ev}}_{bd}\big)\, ,
$$
where ${\mathcal L}(r):=-\frac14\int_{\Sigma_r} \!
\ext {\rm Vol}(h(r))\, 
\go^{ac}\go^{bd}\hh {{\rm Tr}}(F_{ab} \hh \hh F_{cd})$ and
 $\Sigma_r$ is a constant $r$ hypersurface. 
 In even dimensions $d$,  the integrand of the hypersurface integral ${\mathcal L}(r)$ is  an even polynomial  in~$r$ plus terms of order $r^{d+1}$ because $h(r)=h^{\rm ev}+{\mathcal O}(r^{d-1})$, $A=A^{\rm ev}+{\mathcal O}(r^{d-3})$ and $\go^{ab}=r^2 g^{ab}$. Thus there is no contribution to the integrand of the above-displayed integral of order $r^{-1}$, so the coefficient of $\log\varepsilon$ in the regulated Yang--Mills functional  vanishes. Hence the energy functional $\overline{\operatorname{En}}[\bar A]={\rm En}[A^{\rm ev}[\bar A]]$ of Equation~\nn{HolFor} 
 vanishes in dimensions $d=6,8,10,\ldots$. This establishes the vanishing properties stated in the proofs of Theorems~\ref{blabla} and~\ref{obsts}.
 
It is also possible to show that the obstruction~$\bar k$ vanishes for even dimensional boundaries by directly examining the even connection expansion
$$A^{\rm ev}_i(x,r) = A^{(0)}_i(x)+
r^2 A^{(2)}_i(x)+
\cdots +
r^{d-4} A_i^{(d-4)}(x)\, ,
$$
and the   Amp\`ere equation~\nn{amp}, rather than relying on its variational property given in
Theorem~\ref{fruit=vary}. Indeed
the boundary operators constructed in Section~\ref{YMNO} below can be used to extract the above expansion coefficients. 
See in particular Equation~\nn{A2}.
We  have explicitly verified that when applied to the 
regulated Yang--Mills energy functional~\nn{evenreg}, these yield the
 energy functional~$\overline{\operatorname{ En}}[\bar A]$ for six dimensional boundaries quoted in Theorem~\ref{blabla}.  We suppress the details of those computations here.


\subsection{Yang--Mills Poincar\'e--Einstein Boundary Operators}
\label{YMNO}

We now want to construct Yang--Mills analogs of the boundary operators $\check\delta^{(2)}$ and $\check\delta^{(3)}$ in Proposition~\ref{wearenotnormal}
on Poincar\'e--Einstein structures.
While the Neumann data
of the harmonic problem for scalars is respectively accessed by second and third order boundary operators in dimensions $d=3$ and $d=4$, for the Yang--Mills system the corresponding ambient dimensions are $d=5$ and $d=6$, while the 
Neumann data 
corresponds to transverse derivatives of the normal components of the field strength; see Sections~\ref{EP} and~\ref{connexp}. Therefore we are interested in invariant 
covectors~${\sf E}^{(k)}_b\in \Gamma(T^*\Sigma\otimes {\mathcal V}\Sigma[-k])$ of the form $$
\hat n^{a_1} \cdots \hat n^{a_k}(\nabla_{a_1}\cdots \nabla_{a_{k-1}} F_{a_kb})\big|_\Sigma+\cdots\, .$$
We may view 
${\sf E}^{(k)}_b$ as the image of a non-linear map from connections $\nabla^A$ on ${\mathcal V}M$ to invariant boundary covectors taking values in sections of ${\mathcal V}\Sigma[-k]$, and hence term 
${\sf E}^{(k)}_b$ a {\it Yang--Mills boundary operator}.
The integer $k$ is called the {\it transverse order} of  ${\sf E}^{(k)}$ since it measures the maximal number of normal derivatives of the connection potential appearing in ${\sf E}^{(k)}$. 
The  ellipsis $\cdots$ denotes terms of lower transverse order necessary for~${\sf E}^{(k)}_b$ to be conformally invariant. We impose a naturalness condition on these terms, namely that they are built from (partial) contractions of connection-twisted covariant derivatives $\nabla^A$, curvatures $F^A$ and $R$, the metric $g$ and its inverse, as well as the defining density $\sigma$ of the 
Poincar\'e--Einstein structure 
$(M,\cc,\sigma)$. In Section~\ref{PEM} we explained how such quantities can be re-expressed in terms of tensors intrinsic to the boundary  as well as extrinsic curvatures; see in particular Lemma~\ref{69}.

 We are especially  interested in constructing Yang--Mills boundary operators that vanish when evaluated on conformally compact Yang--Mills connections for sufficiently high dimension~$d$ at fixed values of $k$. This means that the image of these operators are obstructions to a  connection $\nabla^A$ solving the conformally compact Yang--Mills system. 

\smallskip
Observe that  a single normal derivative of the connection is accessed by contracting the curvature $F$ with the unit normal $\hat n$, so we define
$$
{\sf E}^{(1)}_a:=F_{\hat n a}\big|_\Sigma\in \Gamma(T^*\Sigma\otimes {\mathcal V}\Sigma[-1])\, .
$$
This may be viewed as a Yang--Mills analog of the conformal-Robin operator $\delta^{(1)}$.
Clearly, for any connection $A$ that obeys the conformally compact Yang--Mills condition~\nn{CCYM}, one has ${\sf E}^{(1)}_a=0$. 

\medskip
The following result is important for the transverse order  $k=2$ case. 
\begin{lemma}
Let $(M,\cc,\sigma)$ be a Poincar\'e--Einstein structure
equipped with a smooth connection $\nabla^A$ on ${\mathcal V}M$. Then 
\begin{equation}\label{E2}
\check{\sf E}^{(2)}_{\, b}(g,A)
:=(d-5)\big[\hat n^a  (\nabla_{\hat n} F_{ab})|_\Sigma+2H^g {\sf E}^{(1)}_b\big]\, 
-{\bj}_b\, 
\end{equation}
  defines a section of $T^*\Sigma\otimes {\mathcal V}\Sigma[-2]$. 

\end{lemma}

\begin{proof}
First we may check that $\check{\sf E}^{(2)}$ can be rewritten as 
\begin{equation}\label{e2}
\check{\sf E}^{(2)}_{\, c}\stackrel{\Sigma}=\Big[
(d-4)(\hat n^a \hat n^b \nabla_a F_{bc}+3H^g F_{\hat n c}) - \nabla^a F_{ac}\Big]^\top\, .
\end{equation}
We now need to show that
$$
\check{\sf E}^{(2)}_{\, b}(\Omega^2 g,A)=\Omega^{-2} \check{\sf E}^{(2)}_{\, b}(g,A)\, .
$$
The proof is then an elementary application of the  transformation formula for Levi-Civita connections of conformally related metrics acting on covectors 
given in Equation~\nn{deltanablas}.
 Also one needs that the unit normals and mean curvatures of conformally related metrics obey
$
\hat n^{\Omega^2 g} = \Omega\hh \hat n\: \mbox{ and }\:
H^{\Omega^2 g}=\Omega^{-1}\big(H^g + g(\Upsilon,\hat n)\big)\, .
$
That computation is somewhat tedious, a quickfire proof employs the linearized conformal variations of Section~\ref{LCVs}.
That requires only application of Equation~\nn{deltanablas} to the two connections appearing in Equation~\nn{e2},
as well as Equation~\nn{varyH} to the mean curvature.
\end{proof}

When $d\neq 5$, we may define
$$
{\sf E}^{(2)}:=\tfrac1{d-5} \hh \check{\sf E}^{(2)}_a\, .
$$
In $d=5$ dimensions, only the boundary divergence term~$\bj_c=\bar \nabla^a \bar F_{ac}$ remains in~$\check{\sf E}^{(2)}_c$,
and in
 that case~$\bj$ is the conformally invariant Yang--Mills current.  Theorem~\ref{obsts} established that $\bj$ is  the $d=5$ obstruction current.   
This is a direct analog of the case $w=(3-d)/2$ for $\check\delta^{(2)}$ acting on scalar densities in Proposition~\nn{wearenotnormal}. 


\smallskip

Analogy with  the scalar operator $\delta^{(2_3)}$ suggests that when $d=5$ and $\bj =0$, we ought consider the tensor
\begin{equation}\label{E25}
{\sf E}^{(2_5)}_{\, c}(g,A):=
\hat n^a \hat n^b (\nabla_a F_{bc})|_\Sigma+2H^g {\sf E}^{(1)}_{ c}\, .
\end{equation}
Indeed a simple computation establishes that
$$
\Omega^2\hh
{\sf E}^{(2_5)}_{\, b}(\Omega^2g,A
)
-
{\sf E}^{(2_5)}_{\, b}(g,A
)
=\Upsilon^a \bar F_{ab}\, .
$$
Hence, if the boundary connection $\bar A$
is flat, then  the obstruction current $\bj_b=\bar \nabla^a \bar F_{ab}=0$,
and moreover~${\sf E}^{(2_5)}(g,A)$ defines a section of $T^*\Sigma\otimes {\mathcal V}\Sigma[-2]$. 
 In Section~\ref{FGrescale}
 we establish that this tensor  invariantly captures the Neumann data for global solutions to the $d=5$ conformally compact Yang--Mills system whose boundary Dirichlet data is a flat connection $\bar A$. Note that since ${\sf E}^{(2_5)}$ has weight $-2$, its divergence is conformally invariant. 
In fact we have the following result.
\begin{lemma}
 Let  $(M,\cc,\sigma)$ be a dimension $d=5$ Poincar\'e--Einstein structure. If~$\nabla^A$ is a conformally compact Yang--Mills connection on ${\mathcal V}M$  for which $F^A|_\Sigma=0$, then
 $$
 \bar \nabla^a {\sf E}^{(2_5)}_a =0\, .
 $$
 \end{lemma}

\begin{proof}
Firstly observe that, by virtue of the Noether identity $\nabla^a \nabla^b F_{ab}=0$ and the symmetry property $\nabla_a n_b = \nabla_b n _a$,  the divergence of the conformally compact Yang--Mills condition  $\sigma \nabla^a F_{ab} =(d-4) F_{nb}$ gives
$$
n^b  \nabla^a F_{ab} \stackrel\Sigma=(d-4) n^a \nabla^b F_{ab}\, .
$$
Thus, since $d\neq 3$, we have 
$$
n_c\nabla_b F^{bc}\stackrel\Sigma=0\, .
$$
In turn
\begin{multline*}
\bar \nabla^a {\sf E}^{(2_5)}_a \stackrel\Sigma=
\nabla_c^\top \big(\hat n^a \hat n^b (\nabla_a F_{b}{}^c)+2H^g F_{\hat n}{}^c\big)
\stackrel\Sigma=\nabla_c^\top (\nabla_b F^{bc}-\nabla^\top_b F^{bc})
\\
\stackrel\Sigma=\hh
\hat n_c \nabla_{\hat n} \nabla_b F^{bc}
\stackrel\Sigma=
(\nabla_{\hat n}-H) (n_c \nabla_b F^{bc})=0\, .
\:\:
\end{multline*}
In the above we again used the Noether identity, the vanishing condition $F^A|_\Sigma=0$, as well as Equation~\nn{nablann}.
\end{proof}

Finally for the $k=2$ case, in dimensions $d>5$, as promised we have the  following.
\begin{lemma}\label{warmerShowarma}
Let  $(M,\cc,\sigma)$ be a dimension $d\geq6$ Poincar\'e--Einstein structure. If~$\nabla^A$ is a conformally compact Yang--Mills connection on ${\mathcal V}M$, then
 $$
  {\sf E}^{(2)}_a =0\, .
 $$
\end{lemma}

\begin{proof}
Since the conformally compact Yang--Mills condition implies~${\sf E}^{(1)}$~$=$~$0$,  we have
$$
{\sf E}^{(2)}_b=\hat n^a  \nabla_{\hat n} F_{ab}|_\Sigma-\tfrac1{d-5}\hh {\bj}_b\, .
$$
The result now follows from direct application of Equation~\nn{first}.
\end{proof}

\subsubsection{Higher Order Boundary Operators}
We are now ready to construct ${\sf E}^{(3)}$ and ${\sf E}^{(4)}$.
Starting with the former we seek an operator that takes three transverse derivatives of the connection potential. This operator probes Neumann data
in dimension $d=6$ and is required to vanish on conformally compact Yang--Mills connections in dimensions $d\geq 7$.

\begin{proposition}\label{third}
 Let $(M,\cc,\sigma)$ be a dimension  Poincar\'e--Einstein structure
equipped with a smooth connection $\nabla^A$ on ${\mathcal V}M$. Then, when $d\neq 7$,
 \begin{multline*}
{\sf E}^{(3)}_c(g,A):= 
\hat n^a \hat n^b ([\nabla_{\hat n}+5H] \nabla_a F_{bc})|_\Sigma
+\tfrac2{d-1}\bar \nabla_c \bar \nabla^a {\sf E}_a^{(1)}
+2 \bar P^a{}_c {\sf E}_a^{(1)} 
-\tfrac32 \bar J {\sf E}_c^{(1)}
\\[1mm]
+
(5H^2-2 P_{\hat n\hat n}) 
{\sf E}_c^{(1)}
-(\bar\nabla^a H)\bar F_{ac}
-\tfrac{3}{d-7} 
\boxast {\sf E}_c^{(1)}
\end{multline*}
defines a section of $T^*\Sigma\otimes {\mathcal V}\Sigma[-3]$. 
\medskip
When $d=7$ and $\nabla^A$ is subject to ${\sf E}^{(1)}=0$, 
$$
{\sf E}^{(3_7)}_c(g,A):= 
\hat n^a \hat n^b (\nabla_{\hat n} \nabla_a F_{bc})|_\Sigma
+\tfrac52H(\check {\sf E}_c^{(2)}
+
\bj_c)
-(\bar\nabla^a H)\bar F_{ac}
$$
defines a section of $T^*\Sigma\otimes {\mathcal V}\Sigma[-3]$.

 \end{proposition}

\begin{remark}
In dimensions $d\neq 5,7$, the operator~${\sf E}^{(3)}$ has the simpler expression
  \begin{multline}\label{E3}
{\sf E}^{(3)}_c(g,A)= 
\hat n^a \hat n^b (\nabla_{\hat n} \nabla_a F_{bc})|_\Sigma
+5H
(  {\sf E}_c^{(2)}+ \tfrac1{d-5}\hh\bj_c)
\\[1mm]
-
(5H^2+2 P_{\hat n\hat n}) 
{\sf E}_c^{(1)}
-(\bar\nabla^a H)\bar F_{ac}
-\tfrac{3}{d-7} 
\boxast_{7}{\sf E}_c^{(1)}
\, .
\end{multline}
\end{remark}
 
 \begin{proof}[Proof of Proposition~\ref{third}]
The above operators were constructed in the first instance by teaching the symbolic manipulation software FORM~\cite{FORM} both the conformal transformation properties and identities obeyed by Yang--Mills coupled hypersurface invariants, and thereafter fitting coefficients for the most general expression whose leading term is a non-zero multiple of $\hat n^a \hat n^b (\nabla_{\hat n} \nabla_a F_{bc})|_\Sigma$ and 
obeys 
$$
{\sf E}^{(3)}_{\, a}(\Omega^2 g,A)=\Omega^{-3} {\sf E}^{(3)}_{\, a}(g,A)\, .
$$
Once one has the given operators at hand, it is easy to prove the proposition statement by considering their linearized conformal variations. 
(We used the full, rather than linearized, conformal transformations in the symbolic software computation as a check of that work, which by itself also constitutes a proof of the claimed result.)

A short computation establishes that the linearized conformal variation of the leading order term in ${\sf E}^{(3)}$,  when $d\neq 5$,
%
%
is given by
\begin{align*}
\delta_{-3} \big(\hat n^a \hat n^b (\nabla_{\hat n} \nabla_a F_{bc})|_\Sigma\big)
&=
(\bar \nabla^a\Upsilon_{\hat n})\bar F_{ac}-6H \bar F_{\bar \Upsilon c}
-2 \bar \Upsilon^a\bar \nabla_c^{\phantom{(1)}\!\!\!\!}{\sf E}^{(1)}_a
+3 \bar \nabla_{\bar \Upsilon} {\sf E}^{(1)}_c
\\&\hspace{1.7cm}
-2 \hat n^a \hat n^b (\nabla_a \Upsilon_b) {\sf E}^{(1)}_c
+10 H \Upsilon_{\hat n} {\sf E}^{(1)}_c
-\tfrac{5}{d-5}\Upsilon_{\hat n}
\big(
\check {\sf E}^{(2)}_c
+\bj_c \big)\, .
\end{align*}
For  ${\sf E}^{(3_7)}$  one can strike off the third through sixth terms on the right hand side above.

\smallskip

The stated invariance properties now follow almost by inspection if one employs the  linearized conformal variations given in Section~\ref{LCVs}.  For instance, when $d\neq 5$, 
\begin{multline*}
\delta_{-3}\Big(\tfrac{5}{d-5} H \bj_c -(\bar \nabla^a H) \bar F_{ac}\Big) = 
\tfrac5{d-5}\big(
(\delta_{-1} H)
\bj_c + H \delta_{-2} \bj_c \big)-
( \bar\nabla^a \delta_{-1}H-\bar \Upsilon^a H) \bar F_{ac}
\\
=
\tfrac5{d-5} \Upsilon_{\hat n} \bj_c + 6H \bar F_{\bar \Upsilon c} -(\bar \nabla^a\Upsilon_{\hat  n}) \bar F_{ac}\, ,
\hspace{3.45cm}
\end{multline*}
which  yields terms canceling the first, second and last terms on the right hand side of the preceding display.
The remaining cancellations are not more difficult to establish than these ones. We leave the special $d=5,7$ cases to the reader.
 \end{proof}


%
%

%

\begin{remark}
Theorem~\ref{recurse} and Corollary~\ref{forgotten-soap}
established that smoothness of  conformally compact Yang--Mills connections is obstructed at order $d-4$ by the  obstruction current~$\bar k$. This is the underlying reason why we introduced the boundary operator $\check {\sf E}^{(2)}$ with a factor $(d-5)$ multiplying the leading transverse derivative  term $\hat n^a  (\nabla_{\hat n} F_{ab})|_\Sigma$. One might have expected the same behavior for the third order operator ${\sf E}^{(3)}$ involving a factor of~$(d-6)$. However, 
 for the special case of Poincar\'e--Einstein structures, the $d=6$  obstruction current $\bar k$ vanishes; see Theorem~\ref{obsts}. Instead 
we encounter a $d=7$ pole in~${\sf E}^{(3)}$ whose residue $-3\boxast {\sf E}^{(1)}$ is conformally invariant.
In $d=7$ dimensions,  invariance of the boundary operator ${\sf E}^{(3_7)}$ is
conditioned 
on vanishing of ${\sf E}^{(1)}$. The operator ${\sf E}^{(3_7)}$  then probes the term in the connection  of one lower order than that of the obstruction.
 This
phenomenon mirrors exactly that observed in Remark~\ref{3} for the scalar operator $\check \delta^{(3)}$ and 
 has a direct analog in the theory of conformal fundamental forms developed in~\cite{Blitz} for the study of obstructions to conformally compact structures being 
Poincar\'e--Einstein.~$\blacklozenge$
\end{remark}

In $d=6$ dimensions, the operator ${\sf E}^{(3)}$ accesses the Neumann data of a conformally compact Yang--Mills connection. Necessarily---see Section~\ref{FGrescale}---this must be encoded by a divergence-free boundary covector.
\begin{lemma}
 Let  $(M,\cc,\sigma)$ be a dimension $d=6$ Poincar\'e--Einstein structure. Then, if~$\nabla^A$ is a
 conformally compact Yang--Mills connection on ${\mathcal V}M$,
 we have  $$
 \bar \nabla^a {\sf E}^{(3)}_a =0\, .
 $$
 \end{lemma}

\begin{proof}
Firstly observe that the asymptotically conformally compact Yang--Mills condition
\begin{equation}\label{falafel}
\sigma \nabla^a F_{ab} - 2 F_{nb}=
0\, ,
\end{equation}
here implies that 
 $${\sf E}^{(1)}=0= {\sf E}^{(2)}\, .$$
The first of the above equalities is a  direct consequence of the definition $ {\sf E}^{(1)}_b:=\hat n^a F_{ab}|_\Sigma$, while the second
was established in Lemma~\ref{warmerShowarma}.

\smallskip
Next we note that contracting~\nn{falafel}
with $n^b$ gives
$$
\sigma n^b \nabla^a F_{ab} =
0\, ,
$$
and thus, away from $\Sigma$---and in turn by smoothness of $A$ everywhere in $M$---we have
\begin{equation}\label{thediv}
 n^b \nabla^a F_{ab} =\nabla^a F_{an}=
 0\, .
\end{equation}
The proof strategy is now  to successively apply powers of $\nabla_n $ to the above display. Indeed
\begin{multline*}
0=\nabla_n^2 \nabla^a F_{an} = \nabla^a ( \nabla_n^2  F_{an})-[\nabla^a ,\nabla_n^2] F_{an}\\
=\nabla^a\big( (n^d \nabla_n^2  +[\nabla_n^2, n^d ]) F_{ad}\big)
-(\nabla^a n^b) \nabla_b \nabla_n F_{an}
\\
 -[F^a{}_n,\nabla_n F_{an} ]- R^a{}_{na}{}^d \nabla_n F_{dn}
 - \nabla_n ([\nabla^a ,\nabla_n] F_{an})\, .
\end{multline*}
We  proceed computing along $\Sigma$. For that we may repeatedly use $0={\sf E}^{(1)}_a = F_{na}|_\Sigma$ as well as Equation~\ref{first} to give $\nabla_n F_{na}|_\Sigma= \bj_a$.
Hence
\begin{multline*}
0\stackrel\Sigma = \nabla^a\big( n^c n^d \nabla_n\nabla_c F_{ad} + n^d (\nabla_n n^c) \nabla_c F_{ad}\big)
+\nabla^a\big((\nabla_n n^d) \nabla_n F_{ad}+ \nabla_{n}((\nabla_n n^d) F_{ad})\big)
\\
-H \nabla^a \nabla_n F_{an}
- 4(\bar \nabla^d H) \hh \bj _d
-\nabla_{\hat n}\big(
R^a{}_{na}{}^d F_{dn}+
(\nabla^a n^b)\nabla_b F_{an}\big)\, .
\end{multline*}
In the above we also used Equation~\nn{mainardi}.
Computing further along the same lines, but now using Equations
~\nn{nablan},~\nn{nablann},~\nn{nablanF} and~\nn{when},  gives
\begin{multline*}
0\stackrel\Sigma=
 \nabla_a^\top( \hat n^c \hat n^d \nabla_{\hat n} \nabla_c F^a{}_d)
 +\hat n^a \nabla_{\hat n} (n^c n^d \nabla_n \nabla_c F_{ad})
 + \hat n^d\big(\tfrac12 (\nabla^a \nabla^c n^2) \nabla_c F_{ad} + H \hat n^c \nabla^a \nabla_c F_{ad}  \big)\\
+2 (\nabla^a \nabla_n n^d) (\hat n_a \bj_d - \hat n_d \bj_a -2H \bar F_{ad})
+2H\hat n^d \nabla^a \nabla_n F_{ad}
\hspace{2.3cm}
\\
+ (\nabla^a\nabla_{\hat n} \nabla_n n^d ) \bar  F_{ad} 
+2(\nabla_n^2 n^d)\, \bj_d 
+H\bar \nabla^a\bj _a 
-H \hat n^a \nabla_{\hat n} \nabla_n F_{an}
\\
- 8(\bar \nabla^d H) \hh \bj _d
-(\nabla_{\hat n}\nabla^a n^b) \nabla_b F_{an}
-H\nabla_{\hat n} \nabla^a F_{an}\, .\hspace{2cm}
\end{multline*}
We now develop each non-trivial term above 
(the key identities underlying the following list of results can be found in Section~\ref{PEM}, and in particular Subsection~\ref{PEYM})
\begin{eqnarray*}
\hat n^a \nabla_{\hat n} (n^c n^d \nabla_n \nabla_c F_{ad})
\!\!&\stackrel\Sigma =&
\hat n^a \hat n^c (\nabla_{\hat n}  n^d) \nabla_n \nabla_c F_{ad}
\stackrel\Sigma =
H\hat n^a \hat n^c \hat n^d\nabla_n \nabla_c F_{ad}=0\, ,
\\[2mm]
\tfrac12 \hat n^d (\nabla^a \nabla^c n^2) \nabla_c F_{ad} 
\!\!&  \stackrel\Sigma=&
 -\hat n^d (\nabla^a \nabla^c (\rho
 \sigma) )\nabla_c F_{ad} 
  \stackrel\Sigma=
  H\hat n^d  (\nabla^a n^c) \nabla_c F_{ad} 
  -\hat n^d (\nabla^a \rho) \nabla_{\hat n} F_{ad}\\
\!\! & \stackrel\Sigma=&
 2H^2 \hat n^d \bj_d+\hat n^d(\bar \nabla^a H)  \nabla_{\hat n} F_{ad}
  \stackrel\Sigma=-(\bar \nabla^a H)\bj_a\,,
  \\[2mm]
H  \hat n^d  \hat n^c \nabla^a \nabla_c F_{ad}
& \stackrel\Sigma=&
H  \hat n^d  \hat n^c \nabla_a^\top \nabla_c F^a{}_{d} 
 \stackrel\Sigma=
 H  \hat n^d  \hat n^c \nabla_a^\top \nabla^\top_c \bar F^a{}_d 
 \!+\! 
 H  \hat n^d  \hat n^c \nabla_a^\top (\hat n_c \nabla_{\hat n} F^a{}_d )\\
\!\!& \stackrel\Sigma=&
 -H^2 \hat n^d \nabla_a^\top  \bar F^a{}_d 
 +H \hat n_d \nabla_a^\top \nabla_{\hat n} F^{ad}
 =-H \bar \nabla_a \bj^a=0\, ,
\\[2mm]
2 (\nabla_{\hat n}\nabla_n n^d) \bj_d \!\!&\stackrel\Sigma=&2 (\bar \nabla^d H)\bj_d
\, ,
\\[2mm]
 -2 \bj^a \hat n_d \nabla_a^\top \nabla_n n^d& \stackrel\Sigma=&- 2 (\bar \nabla^d H)\bj_d\, ,
 \\[2mm]
  2H\hat n^d \nabla^a \nabla_n F_{ad}
\!\!& \stackrel\Sigma=&
 2H\hat n_d \nabla_a^\top \nabla_n F^{ad}
 \stackrel\Sigma=-2H\bar \nabla^a \bj_a=0\, ,
 \\[2mm]
 -4 H \bar F_{ad}  \nabla^a \nabla_n n^d \!\!&\stackrel\Sigma=&
 0\, ,
 \\[2mm]
\bar F^{ad} \nabla_a \nabla_{\hat n} \nabla_n n_d 
\!\!&\stackrel\Sigma=&
0\, ,
\\[2mm]
2\bar j^d \nabla_n^2 n_d &\stackrel\Sigma=
&
2(\bar \nabla^d H)\bj_d\, ,
\\[2mm]
H \bar \nabla^a \bj_a 
\!\!&\stackrel\Sigma=& 0\, ,
\end{eqnarray*}
\begin{eqnarray*}
-H \hat n^a \nabla_{\hat n} \nabla_n F_{an}&\stackrel\Sigma=&
-H[n^a ,\nabla_{\hat n} \nabla_n] F_{an}
\stackrel\Sigma=H ^2 \hat n^a \nabla_{\hat n } F_{an}
+ H \nabla_{\hat n} \big(\nabla^a(-\rho\sigma) F_{an}\big)\!\stackrel\Sigma=\!0\, ,
\\[2mm]
-(\nabla_{\hat n} \nabla^a n^b)\nabla_b F_{an}
\!\!&\stackrel\Sigma=&
(P^{ab}+g^{ab} P_{\hat n\hat n})\nabla_b F_{an}
\stackrel\Sigma=
(P^{ab}+g^{ab} P_{\hat n\hat n}) \hat n_b \nabla_{\hat n} F_{an}
\stackrel\Sigma=(\bar \nabla^d H)\bj_d\, ,
\\[2mm]
-H \nabla_{\hat n} \nabla^a F_{an}
\!\!&\stackrel\Sigma=&-\frac12 H\nabla_{\hat n} \nabla^a(\sigma \nabla^b F_{ba})  
\stackrel\Sigma=
-\frac12 H \nabla_{\hat n} (n^a \nabla^b F_{ba})
\stackrel\Sigma=
0\, .
\end{eqnarray*}
Orchestrating gives
$$
0=
\nabla^\top_c(\hat n_a \hat n_b \nabla_{\hat n} \nabla^a F^{bc})
+6(\bar \nabla_cH)\bj^c
=
\bar \nabla^a {\sf E}^{(3)}_a \, .
$$
The second equality above used Equation~\nn{E3} with 
${\sf E}^{(1)}$ and ${\sf E}^{(2)}$ set to zero, as well as the (boundary) Noether identity.
\end{proof}

We also stipulated at the beginning of this section that the boundary  operators ${\sf E}^{(3)}$ and ${\sf E}^{(3_7)}$ obey vanishing properties when evaluated on conformally compact Yang--Mills connections in dimensions $d\geq 7$. Such is given below.
\begin{lemma}\label{CFF}
 Let  $(M,\cc,\sigma)$ be a dimension $d\geq7$ Poincar\'e--Einstein structure. If~$\nabla^A$ is a
 conformally compact Yang--Mills connection on ${\mathcal V}M$  then  $$
  \left\{
  \begin{array}{l}{\sf E}^{(3_7)}_a =0\,,\qquad d=7 \, ,
  \\
\,   {\sf E}^{(3)}_a \, =0\, ,\qquad\,    d\geq8 \, .
  \end{array}
  \right.
 $$
 \end{lemma}
 \begin{proof}
 From Lemma~\ref{warmerShowarma} we have
 $
 {\sf E}^{(2)}=0= {\sf E}^{(1)}
 $ so (for $d\geq 8$)
 $$
  {\sf E}^{(3)}_a
 = 
\hat n^a \hat n^b (\nabla_{\hat n} \nabla_a F_{bc})|_\Sigma
+\tfrac5{d-5}\big(H
\bj_c
+(1-\tfrac d5)(\bar\nabla^a H)\bar F_{ac}\big)
\, .
 $$
 The same expression applies for ${\sf E}^{(3_7)}_a$ when $d=7$.
 The result now follows by applying Equation~\nn{d66again} (remembering that $P_{\hat n a}^\top\stackrel\Sigma= -\bar \nabla_a H$; see Equation~\nn{mainardi}).
 \end{proof}

\bigskip
The final case we consider is the transverse order four boundary operator. 
As it takes four transverse derivatives of the potential, this operator will probe Neumann data
when  $d=7$ and is required to vanish on conformally compact Yang--Mills connections in dimensions~$d\geq8$. Because the critical dimension $d=7$ is odd, we expect it to exhibit special behavior relative to its ${\sf E}^{(3)}$ counterpart, but closer to its ${\sf E}^{(2)}$ companion.

\begin{proposition}\label{quartic}
 Let $(M,\cc,\sigma)$ be a dimension $d\neq 4,5,7,9$ Poincar\'e--Einstein structure
equipped with a smooth connection $\nabla^A$ on ${\mathcal V}M$. Then  
\begin{align}
\nonumber
\check{\sf E}^{(4)}_d(g,A)&:=
(d-7)\Big[
\hat n^a \hat n^b \hat n^c (\nabla_{\hat n} \nabla_a\nabla_b F_{cd})|_\Sigma
\color{black}+9H \big({\sf E}^{(3)}_d+\tfrac{3}{d-7}  \boxast_7 {\sf E}_d^{(1)}\big)\color{black}
\\
&\qquad\qquad
\color{black}
-(7P_{\hat n\hat n}
+\tfrac{45}2 H^2) \big({\sf E}^{(2)}_d+\tfrac1{d-5} \bj_d\big)
+(\bar\nabla^a H)\big(3\bar \nabla_d  {\sf E}^{(1)}_a -4\bar \nabla_a  {\sf E}^{(1)}_d\big)
\color{black}
\nonumber\\[1mm]
&\qquad\qquad
+\big((d+\color{black}14\color{black}) H^3 +2(d+6) H P_{\hat n\hat n} H
-2(\bar\Delta H)-2H\bar J-2\nabla_{\hat n}J
\big){\sf E}^{(1)}_d
\nonumber\\[1mm]
&\qquad\qquad
\color{black}
+\big((\bar \nabla^a P_{\hat n\hat n})
+9H (\bar \nabla^a H)
\big)
\bar 
F_{ad}
\color{black}
\label{E4}\\
&\qquad\qquad
\color{black}
+\tfrac{8}{d-4} \big[{\sf E}^{(1)}_d,
\bar \nabla^a{\sf E}^{(1)}_a\big]
+\tfrac{12}{(d-5)(d-9)}
[ \bar F_{ad}, \bj^a]
\color{black}
\nonumber\\
&\qquad\qquad
-\tfrac{6}{d-9}\Big(
\color{black}
 \boxast_9\! \big({\sf E}^{(2)}_d
  +\tfrac1{d-5}\bj_d\big)
+ \bar P^{ab} \bar \nabla_a \bar F_{bd}
+\tfrac12(\bar \nabla^a \bar J) \bar F_{ad}
+\tfrac12 \bar C_{abd} \bar F^{ab}
\Big) \Big]
\nonumber\\[2mm]\nonumber
& 
\color{black}
\phantom{:=}
+\tfrac{12}{d-9} \hat k_d \, ,
\end{align}
where
\begin{eqnarray*}
 \hat k_b&:=&\quad
 \tfrac12\bar \nabla^a
\Big(
 \bar \nabla_{[a} {\bj}_{b]} 
 - 4\bar P_{[a}{}^{c} \bar F_{b]c}
-\bar J \bar F_{ab}
\Big)
+\tfrac14 [\bj^a,\bar F_{ab}] \, ,
\end{eqnarray*}
defines a section of $T^*\Sigma\otimes {\mathcal V}\Sigma[-4]$.
\end{proposition} 

\begin{proof}
Once again, the expression~\ref{E4} was generated by the   symbolic manipulation software FORM so that the same remarks as made at the beginning of the proof of Proposition~\ref{third} apply here, and
 we shall rely on the linearized conformal variations developed in Section~\nn{LCVs}
to show that
$$
\check{\sf E}^{(4)}_{\, a}(\Omega^2 g,A)=\Omega^{-4} \check{\sf E}^{(4)}_{\, a}(g,A)\, .
$$
The most difficult of those variations is that of the first term in Equation~\nn{E4}. This is given by
\begin{align}
\delta_{-4} \big(\hat n^a  &\hat n^b \hat n^c (\nabla_{\hat n} \nabla_a\nabla_b F_{cd})|_\Sigma\big)
\:=\:
\color{black}\!-9\Upsilon_{\hat n}
\big(
 {\sf E}^{(3)}_d
+\tfrac{3}{d-7}\boxast_7 {\sf E}^{(1)}_d \big)
\!+\!21H \bar \Upsilon^a\bar 
\nabla_d^{\phantom{(1)}\!\!\!\!}{\sf E}^{(1)}_a
\!-31H \bar \nabla_{\bar \Upsilon} {\sf E}^{(1)}_d\color{black}
\nonumber\\&
\color{black}
-\big(7\hat n^a \hat n^b (\nabla_a \Upsilon_b)
-45 H\Upsilon_{\hat n}\big)
\big( {\sf E}^{(2)}_d +\tfrac1{d-5} \bj_d\big)
\!+\!9(P_{\hat n\hat n}+\tfrac{7}2H^2) \bar F_{\bar \Upsilon d}
-9H(\bar \nabla^a\Upsilon_{\hat n})\bar F_{ad}
\color{black}
\nonumber\\&
\color{black}+\big(\bar \nabla^a (\hat n^b  \nabla_{\hat n}\Upsilon_{b})\big)\bar F_{ad}
-9\Upsilon_{\hat n}
(\bar\nabla^a H) \bar F_{ad}
-(\bar \nabla^a \Upsilon_{\hat n}) \big(3\bar \nabla_d {\sf E}^{(1)}_a
-4   \bar \nabla_a {\sf E}^{(1)}_d\big)
\color{black}
\nonumber
\\&\nonumber
\color{black}
+(\bar \nabla_d H) {\sf E}^{(1)}_{\bar \Upsilon}
+2\bar \Upsilon_d(\bar \nabla^a H) {\sf E}^{(1)}_{a}
\color{black}
\\[-3mm]&
\label{leading-term}
\\[-3mm]&
-\big(
\color{black}
2\hat n^a \hat n^b \hat n^c (\nabla_a \nabla_b\Upsilon_c) 
-14H \hat n^a \hat n^b (\nabla_a \Upsilon_b)
\color{black}+9(\bar \nabla_{\bar \Upsilon } H)
+18\Upsilon_{\hat n} P_{\hat n\hat n}
\color{black}
+45\Upsilon_{\hat n} H^2
\big)  {\sf E}^{(1)}_d
\nonumber\\&
\color{black}-8\big[{\sf E}^{(1)}_d,{\sf E}^{(1)}_{\bar \Upsilon}\big]
\color{black}+6\Big( \bar \nabla_{\bar \Upsilon} \big({\sf E}^{(2)}_d+\tfrac1{d-5} \bj_d\big)
-\tfrac 12\bar \Upsilon^a\bar \nabla_d^{\phantom{(1)}\!\!\!\!}\big({\sf E}^{(2)}_a+\tfrac1{d-5} \bj_a\big)\Big)
\color{black}
\nonumber\\&
\color{black}+3\bar P_d{}^a \bar F_{\bar \Upsilon a}
+6\bar P_{\bar \Upsilon}{}^a \bar F_{a d}
\color{black}
\nonumber
\, .
\end{align}
While the above is tedious to compute, all that is needed is  $\delta F_{cd}=0$, the result in Equation~\nn{deltanablas} for the variation of connections, and $\delta_{1}\hat n=0$. Thereafter one can employ the identities in Section~\ref{PEM} to decompose the result into intrinsic and extrinsic tensor quantities, as well as the expressions above for ${\sf E}^{(1)}$, ${\sf E}^{(2)}$, ${\sf E}^{(3)}$, and that  for the operator~$\boxast_7$. 

To see that  the above-displayed variation cancels that of the remaining terms in Equation~\nn{E4}, we begin by noting that the first three lines of the above display
exactly cancel nearly all the variations  all the remaining terms in 
the first two lines of~\nn{E4}
as well as that of the   first term on the fourth line; this is easily established using Equations~\nn{varyH},~\nn{varyH1},~\nn{varybox7},~\nn{Pnnvary},~\nn{prairie_dog1}, and~\nn{gradPnnvary}.
The unaccounted-for variations on the first, second and fourth lines in Equation~\nn{E4} come from
$$
\delta_{-1}\big(3\bar \nabla_d  {\sf E}^{(1)}_a -4\bar \nabla_a  {\sf E}^{(1)}_d\big)=
-2\bar \Upsilon_d  {\sf E}^{(1)}_a
+5\bar \Upsilon_a  {\sf E}^{(1)}_d
-\bar g_{ad}   {\sf E}^{(1)}_{\bar \Upsilon}\, .
$$
The first and last terms on the right hand side 
are as required to cancel the two terms on the fourth line of the variation of $\hat n^a  \hat n^b \hat n^c (\nabla_{\hat n} \nabla_a\nabla_b F_{cd})|_\Sigma$ given above.


\bigskip

There are three further variations unaccounted for in the second and  fourth lines of~\nn{E4} coming from varying~$\bj$ according to~\nn{prairie_dog1}, the first term in the variation of the gradient of mean curvature coming from Equation~\nn{varyH1}, and the following variation
$$
\delta_{-1}\big(3\bar \nabla_d  {\sf E}^{(1)}_a -4\bar \nabla_a  {\sf E}^{(1)}_d\big)=
5\bar \Upsilon_a  {\sf E}^{(1)}_d
-2\bar \Upsilon_d  {\sf E}^{(1)}_a
-\bar g_{ad}   {\sf E}^{(1)}_{\bar \Upsilon}\, .
$$
At this point we can cancel all remaining variations proportional to one power of ${\sf E}^{(1)}$. For this one needs Equations~\nn{varyH},~\nn{varyH1},~\nn{Pnnvary},~\nn{varyJdot} and  the consequence $\delta_{-2} \bar J
=-\bar \nabla^a \bar\Upsilon_a$
of Equation~\nn{shouting}. 

It remains to cancel variations
of the last three lines in~\nn{E4} with the 
the last two  lines of~\nn{leading-term}.
For that we employ the variations displayed in Equation~\nn{resrem} as well as the variation of the last term proportional to $\hat k$ is given in~\nn{bark}. The remaining required variations are given below:
\begin{align*}
\delta_{-4}\boxast_9\! \big({\sf E}^{(2)}_d+\tfrac1{d-5}\bj_d\big)&=
\color{black}(d-9)\big(
 \bar \nabla_{\bar \Upsilon} \big({\sf E}^{(2)}_d+\tfrac1{d-5} \bj_d\big)
  -\tfrac12
  \bar \Upsilon^a\bar \nabla_d^{\phantom{(1)}\!\!\!\!}\big({\sf E}^{(2)}_a+\tfrac1{d-5} \bj_a\big)
 \big)
  \\[1mm]
 &
 \color{black}
 +\bar\nabla_{\bar\Upsilon} \bj_d -\tfrac12 \bar\Upsilon^a \bar \nabla_d \bj_a
 + (\bar\nabla^a \bar \nabla^b \bar \Upsilon_b) 
\bar F_{ad} 
-2 (\bar \nabla^a \bar \Upsilon^b) \bar \nabla_a \bar F_{db}
 \color{black}
 \\
 &
 \color{black}+
\tfrac{(d-5)}2 \bar P_d{}^a \bar F_{\bar\Upsilon a}
+2(d-4)\bar P_{\bar \Upsilon}{}^a \bar F_{ad}
+2\bar W_{da\bar\Upsilon b}\bar F^{ab}
+\tfrac{1}2 (\bar\nabla_a \bar \Upsilon_d) \bj^a
\color{black}
\\&
\color{black}
+2 [\bar F_{ad},\bar F_{\bar \Upsilon}{}^a]\, ,
\color{black}
\end{align*}
\begin{multline*}
\delta_{-4}\big( \bar P^{ab} \bar \nabla_a \bar F_{bd}
+\tfrac12(\bar \nabla^a \bar J) \bar F_{ad}
+\tfrac12 \bar C_{abd} \bar F^{ab}\big)
=
\color{black}
(\bar\nabla^a \bar\Upsilon^b) \bar\nabla_a \bar F_{db}
-3\bar P_{\bar \Upsilon}{}^{b} \bar F_{bd}
\color{black}\\
\color{black}-\bar P_d{}^b \bar F_{\bar \Upsilon b}
-\tfrac12 (\bar \nabla^a \bar \nabla^b \bar \Upsilon_b)\bar F_{ad}
-\bar W_{da\bar\Upsilon b} \bar F^{ab}\
\color{black}\, .
\end{multline*}
It is not difficult to check that the coefficients are exactly those required for cancellation.
\end{proof}

 \begin{remark}
 When $d\neq 4,5,7,9$, we may define
$$
{\sf E}^{(4)}:=\tfrac1{d-7} \hh \check{\sf E}^{(4)}\, .
$$
Also,  based on  invariance alone, it is possible to add the tensor 
$$ 
 \bar W_{dabc}\bar \nabla^c \bar F^{ab}
+\frac 32\bar C_{abd} \bar F^{ab}
$$
to $\check {\sf E}^{(4)}$ since the above defines a section  of
$T^*\Sigma\otimes {\mathcal V}\Sigma[-4])$. In particular, when $d=7$, this tensor equals the manifestly conformally invariant divergence $\bar \nabla^c( \bar W_{dabc}\bar F^{ab})$. In Lemma~\ref{lastcoeff} we establish that the coefficient of the above-displayed tensor  must vanish upon requiring that~$\check {\sf E}^{(4)}=0$  when evaluated on conformally compact Yang--Mills connections.
\hfill$\blacklozenge$
\end{remark}

Viewing $d$ as a formal variable, Equation~\nn{E4} has poles at $d=4,5,9$ with respective  residues
\begin{align}
\nonumber
-\tfrac1{24}
\operatorname{res}_{d=4}\check {\sf E}^{(4)}_d &=
 \big[{\sf E}^{(1)}_d,
\bar \nabla^a{\sf E}^{(1)}_a\big]
\stackrel{\delta_{-4}}\mapsto
(d-4)  \big[{\sf E}^{(1)}_d,
{\sf E}^{(1)}_{\bar \Upsilon}\big]
\, ,\\\label{resrem}
\tfrac1{6}
\operatorname{res}_{d=5}\check {\sf E}^{(4)}_d &=
\:\:\:\:
[\bar F_{ad}, \bj^a]
\:\:\:\:\:\:\:\hh
\stackrel{\delta_{-4}}
\mapsto
(d-5)[\bar F_{ad}, \bar F_{\bar\Upsilon}{}^a]
\, ,\
\\\nonumber
-\tfrac1{12}\operatorname{res}_{d=9}\check {\sf E}^{(4)}_d &=\:\:\:\:\:\:
\boxast_9 {\sf E}^{(2)}_d
\:\:\:\:\:\:\,\hh\hh
\stackrel{\delta_{-4}}\mapsto
(d-9)\big( \bar \nabla_{\bar \Upsilon} {\sf E}^{(2)}_c
-\tfrac 12 \bar \Upsilon^a \bar \nabla_c  {\sf E}^{(2)}_a\big)
\, . 
\end{align}
In the above we have also listed the (dimension $d$) linearized conformal variations of the corresponding tensor structures. These vanish in dimensions $4,5,9$ respectively, and therefore these three tensors define conformal invariants in those dimensions. 
In turn, this suggests we define the following three boundary operators by subtracting the relevant pole in $\check{\sf E}^{(4)}$  and then setting $d$ to the corresponding value.
\begin{align}
\nonumber
\hh
{\sf E}^{(4_4)}_d(g,&A):=
\hat n^a \hat n^b \hat n^c (\nabla_{\hat n} \nabla_a\nabla_b F_{cd})|_\Sigma
+9H \big({\sf E}^{(3)}_d-  \boxast_7 {\sf E}_d^{(1)}\big)
-(7P_{\hat n\hat n}
+\tfrac{45}2 H^2) \big({\sf E}^{(2)}_d- \bj_d\big)
\\[1.3mm]
\nonumber
&
+(\bar\nabla^a H)\big(3\bar \nabla_d  {\sf E}^{(1)}_a -4\bar \nabla_a  {\sf E}^{(1)}_d\big)
+2\big(9 H^3\! +\!10 P_{\hat n\hat n} H
-(\bar\Delta H)-H\bar J-\nabla_{\hat n}J
\big){\sf E}^{(1)}_d
\\[-2mm]
\label{E44}
\\[-8mm]
\nonumber
&
+\big((\bar \nabla^a \!P_{\hat n\hat n})
\!+\!9H\hh \!(\bar \nabla^a \!H)
\big)
\bar 
F_{ad}\!
+\!\tfrac{6}{5}\Big(\!\! \boxast_9\!\! \big({\sf E}^{(2)}_d\!\!-\!\bj_d\big)
\!\!+\! \!\bar P^{ac} \bar \nabla_a \bar F_{cd}
\!+\!\tfrac12(\bar \nabla^a \!\bar J) \bar F_{ad}
\!+\!\tfrac12 \bar C_{acd} \bar F^{ac}\!
\Big) 
\\[1mm]
\nonumber
& 
+\tfrac{4}5 \hat k_b
+\tfrac{12}{5}
[ \bar F_{ad}, \bj^a]
 \, ,
\end{align}
\begin{align}
\nonumber
{\sf E}^{(4_5)}_d(g,&A):=
\hat n^a \hat n^b \hat n^c (\nabla_{\hat n} \nabla_a\nabla_b F_{cd})|_\Sigma
+9H \big({\sf E}^{(3)}_d-\tfrac{3}{2}  \boxast_7 {\sf E}_d^{(1)}\big)
-(7P_{\hat n\hat n}
+\tfrac{45}2 H^2) {\sf E}^{(2_5)}_d
\\
\nonumber
&
+(\bar\nabla^a H)\big(3\bar \nabla_d  {\sf E}^{(1)}_a -4\bar \nabla_a  {\sf E}^{(1)}_d\big)
+\big(19 H^3 +22 P_{\hat n\hat n} H
-2(\bar\Delta H)-2H\bar J-2\nabla_{\hat n}J
\big){\sf E}^{(1)}_d
\\[-3mm]
\label{E45}
\\[-7mm]
\nonumber
&
+\big((\bar \nabla^a P_{\hat n\hat n})
+9H (\bar \nabla^a H)
\big)
\bar 
F_{ad}
+8 \big[{\sf E}^{(1)}_d,
\bar \nabla^a{\sf E}^{(1)}_a\big]
\\
\nonumber
&
+\tfrac32\Big( \boxast_9\! {\sf E}^{(2_5)}_d
+ \bar P^{ac} \bar \nabla_a \bar F_{cd}
+\tfrac12(\bar \nabla^a \bar J) \bar F_{ad}
+\tfrac12 \bar C_{acd} \bar F^{ac}
\Big) +\tfrac32 \hat k_d \, ,
\end{align}
and
\begin{align}
\nonumber
{\sf E}^{(4_9)}_d(g,&A):=
\hat n^a \hat n^b \hat n^c (\nabla_{\hat n} \nabla_a\nabla_b F_{cd})|_\Sigma
+9H \big({\sf E}^{(3)}_d+\tfrac{3}{2}  \boxast_7 {\sf E}_d^{(1)}\big)
-\tfrac14(7P_{\hat n\hat n}
+\tfrac{45}2 H^2)\bj_d
\\&
+(\bar\nabla^a H)\big(3\bar \nabla_d  {\sf E}^{(1)}_a -4\bar \nabla_a  {\sf E}^{(1)}_d\big)
+\big(23 H^3 +30 P_{\hat n\hat n} H
-2(\bar\Delta H)-2H\bar J-2\nabla_{\hat n}J
\big){\sf E}^{(1)}_d
\nonumber\\[-3mm]
\label{E49}\\[-8mm]
\nonumber
&
+\big((\bar \nabla^a P_{\hat n\hat n})
+9H (\bar \nabla^a H)
\big)
\bar 
F_{ad}
+\tfrac{8}{5} \big[{\sf E}^{(1)}_d,
\bar \nabla^a{\sf E}^{(1)}_a\big]
\nonumber\\
\nonumber
& 
 -\tfrac34 \big(
  \bar P_{da}\bj^a
  +2 \bar P^{ab} \bar \nabla_a \bar F_{bd}
  +(\bar \nabla^a \bar J) \bar F_{ad}
  + \bar C_{abd} \bar F^{ab}
\big)
 -\tfrac32 \hat k_d\, . 
\end{align}
In these terms we have the following result.
\begin{proposition}
\label{E44E45E49}
 Let $(M,\cc,\sigma)$ be a dimension $d$ Poincar\'e--Einstein structure
equipped with a smooth connection $\nabla^A$ on ${\mathcal V}M$. 
When  $\nabla^A$ is respectively subject to 
\begin{align*}
[{\sf E}^{(1)}_a,{\sf E}^{(1)}_b]&=0\, ,
 & d=4\, ,\\
[\bar F_{ad}, \bar F^{ac}]\, &=0\, ,  &d=5\, ,\\
{\sf E}^{(2)}_{a}\:\:\:\:&=0\, , & d=9\, ,
\end{align*}
then
${\sf E}^{(4_j)}_{\, b}$  defines  a section of $T^*\Sigma\otimes {\mathcal V}\Sigma[-4]$ when $d=j=4,5,9$. 
\end{proposition}
\begin{proof}
As explained above,
 the formul\ae\ for ${\sf E}^{(4_4)}$, ${\sf E}^{(4_5)}$,  ${\sf E}^{(4_9)}$ in Equations~\nn{E44},~\nn{E45} and~\nn{E49}
 are obtained by removing, respectively, the poles  
$
-\tfrac{8(d-7)}{24}
 \big[{\sf E}^{(1)}_d$, $\bar \nabla^a{\sf E}^{(1)}_a\big]$,
$
\frac{12(d-7)}{(d-5)(d-9)}
[\bar F_{ad}, \bj^a]$,
and
$-\frac{6(d-7)}{d-9}
\boxast_9 {\sf E}^{(2)}_d
$
from the expression for $\check {\sf E}^{(4)}$. It can be checked that the resulting expressions are not singular  at $d=4,5,9$, respectively.

The proof now follows that for Proposition~\ref{quartic}
quite closely. By construction, any poles in the formula for the variation of  the leading term $\hat n^a  \hat n^b \hat n^c (\nabla_{\hat n} \nabla_a\nabla_b F_{cd})|_\Sigma$ are removable. We now  examine the variations listed in Equation~\nn{resrem} of the residues of the poles that were subtracted to construct ${\sf E}^{(4_4)}$, ${\sf E}^{(4_5)}$ and ${\sf E}^{(4_9)}$. Then we
observe that the  respective conditions placed upon $\nabla^A$ in the proposition statement, imply vanishing of those variations. Cancellation of all remaining variations proceeds exactly along the lines presented in the proof of   Proposition~\ref{quartic}.

\end{proof}

The case of $d=7$ is of particular interest for $\check {\sf E}^{(4)}$. In particular note from Equation~\ref{E3} that the $d=7$ pole in the pair of terms
$\big({\sf E}^{(3)}_d+\tfrac{3}{d-7}  \boxast_7 {\sf E}_d^{(1)}\big)$ appearing in the first line of Equation~\nn{E4} is removable. Hence, in $d=7$ dimensions,  Equation~\nn{E4} determines an invariant expression for $\check {\sf E}^{(4)}$ 
 where it returns the 
tensor $-6\hat k$. In dimension~$d$, 
the linearized conformal variation of $\hat k$ is given by
\begin{multline}\label{bark}
\delta_{-4}\hat k_b = \tfrac{d-7}2\Big(
\color{black}
\bar \nabla_{\bar \Upsilon} \bj_b
-\tfrac12 \bar \Upsilon^a\bar \nabla_b\bj_a 
+\tfrac12 (\bar \nabla^a\bar \nabla^c \bar\Upsilon_c) \bar F_{ab}
+
\tfrac12 (\bar\nabla^a \bar \Upsilon_b) \bj_a
\color{black}
\color{black}
- (\bar\nabla^a \bar \Upsilon^c)\bar\nabla_a \bar F_{bc}
\color{black}
\\
\color{black}+ \bar P_b{}^a \bar F_{\bar\Upsilon a} 
+(d-2) 
\bar P_{\bar\Upsilon}{}^a \bar F_{ab} 
+\bar W_{ba\bar\Upsilon c}\bar F^{ac}
\color{black}
\Big)\, .
\end{multline}
This establishes that $\hat k$ is invariant in $d=7$ dimensions;  Theorem~\ref{obsts}  shows that it equals the $d=7$ obstruction current $\bar k$. 
Just as in the case of $\check {\sf E}^{(2)}$ when $d=5$, we  consider the conformally compact Yang--Mills system with a flat boundary connection, in which case both $\hat k$ and its conformal variation vanish.  We may then construct an operator that extracts the transverse order four Neumann data. This operator is given below, and discussed further in Section~\ref{FGrescale}.
\begin{proposition}
  Let $(M,\cc,\sigma)$ be a dimension $d=7$ Poincar\'e--Einstein structure
equipped with a smooth connection $\nabla^A$ on ${\mathcal V}M$.
 When the boundary curvature induced by $\nabla^A$ is subject to $$\bar F_{ab}=0\, ,$$
then
\begin{align}
\nonumber
{\sf E}^{(4_7)}_d(g,A)&:=
\hat n^a \hat n^b \hat n^c (\nabla_{\hat n} \nabla_a\nabla_b F_{cd})|_\Sigma
+9H {\sf E}^{(3_7)}_d
-(7P_{\hat n\hat n}
+\tfrac{45}2 H^2) {\sf E}^{(2)}_d
\\
 &
+(\bar\nabla^a H)\big(3\bar \nabla_d  {\sf E}^{(1)}_a -4\bar \nabla_a  {\sf E}^{(1)}_d\big)
+\tfrac{8}{3} \big[{\sf E}^{(1)}_d,
\bar \nabla^a{\sf E}^{(1)}_a\big]
\nonumber\\[-3mm]
\label{E47}
&
\\[-2mm]
\nonumber&
+\big(21 H^3 +26 P_{\hat n\hat n} H
-2\bar\Delta H-2H\bar J-2\nabla_{\hat n}J
\big){\sf E}^{(1)}_d
+3 \boxast_9\! {\sf E}^{(2)}_d
 \, ,
\end{align} 
defines a section of $T^*\Sigma\otimes {\mathcal V}\Sigma[-4]$. 
\end{proposition}

\begin{proof}
The proof follows that of Proposition~\ref{E44E45E49} {\it mutatis mutandis}.
\end{proof}

As we shall see in Section~\ref{FGrescale}, for the boundary operator   ${\sf E}^{(4_7)}$ to correspond to Neumann data, it must satisfy a divergence-free property.
  \begin{lemma}
 Let  $(M,\cc,\sigma)$ be a dimension $d=7$ Poincar\'e--Einstein structure. 
 Then if~$\nabla^A$ is a conformally compact Yang--Mills connection on ${\mathcal V}M$  for which $\bar F_{ab}=F_{ab}|_\Sigma=0$, 
 then  $$
 \bar \nabla^a {\sf E}^{(4_7)}_a =0\, .
 $$
 \end{lemma}
 
 \begin{proof}
 Using necessarily
(see Lemmas~\ref{warmerShowarma} and~\ref{CFF})
that
 \begin{equation}\label{EEEF0}
0={\sf E}^{(3_7)}={\sf E}^{(2)}= {\sf E}^{(1)}=\bar F\, ,
 \end{equation}
 we have 
 \begin{align*} 
 {\sf E}^{(4_7)}&=
\hat n^a \hat n^b \hat n^c (\nabla_{\hat n} \nabla_a\nabla_b F_{cd})|_\Sigma
\, .
\end{align*}
Returning to Equation~\nn{thediv}, we obtain
\begin{multline*}
0= n^b n^c \nabla_n\nabla_b \nabla_c\nabla^d F_{nd}
\stackrel\Sigma=
\hat n^a \hat n^b \hat n^c\nabla_d ^\top\nabla_a \nabla_b \nabla_c F_{n}{}^d
\\
+
\hat n^a \hat n^b \hat n^c\hat n^d[\nabla_{\hat n}\nabla_a \nabla_b \nabla_c ,n^e]F_{ed}
+
n^a n^b n^c [\nabla_a \nabla_b \nabla_c,\nabla^d] F_{nd}
\, .
\end{multline*}
Thus, combining the above two displays gives
\begin{eqnarray*}
0&\stackrel\Sigma=&
\bar \nabla^d{\sf E}^{(4_7)}_d
-[\nabla_d^\top,n^a n^b n^c]
\nabla_a\nabla_c F_n{}^d+
\hat n^a \hat n^b \hat n^c\hat n^d\nabla_{\hat n}[\nabla_a \nabla_b \nabla_c ,n^e]F_{ed}
\\&&
+n^b n^c R_n{}^d{}{}^\sharp \nabla_b \nabla_c F_{nd}
+n^b n^c \nabla_n [ \nabla_b \nabla_c,\nabla^d] F_{nd}\, .
\end{eqnarray*}
We must now show  that the last four terms on the right hand side
above vanish by virtue of Equation~\nn{EEEF0}.
First observe that this equation implies that
\begin{multline*}
0=(\nabla^\top)^\ell F_{ab}|_\Sigma=\hat n^a (\nabla^\top)^\ell F_{ab}|_\Sigma
=\hat n^a\hat n^b (\nabla^\top)^\ell\nabla_a F_{bc}|_\Sigma
\\
=\hat n^a\hat n^b (\nabla^\top)^\ell\nabla_a F_{bc}|_\Sigma
=\hat n^a\hat n^b \hat n^c  (\nabla^\top)^\ell\nabla_a \nabla_b F_{cd}|_\Sigma\, ,
\end{multline*}
for any positive integer $\ell$, and in turn
\begin{multline*}
0=(\nabla^\top)^\ell F_{ab}|_\Sigma=(\nabla^\top)^\ell F_{nb}|_\Sigma
=\hat n^a(\nabla^\top)^\ell\nabla_a F_{nc}|_\Sigma
\\
=\hat n^a (\nabla^\top)^\ell\nabla_a F_{nc}|_\Sigma
=\hat n^a\hat n^b   (\nabla^\top)^\ell\nabla_a \nabla_b F_{nd}|_\Sigma\, .
\end{multline*}
After developing  commutators,  we see that none of these four above-mentioned  terms involves more than three derivatives of the curvature $F$, so it remains to establish that they can be expressed in terms of the above vanishing set. 
This is easily achieved, once one realizes that the Bianchi identity can be used to write
$$
n^c \nabla_c F_{ab} = -2 \nabla_{[a} F_{b]n} +(\nabla_{[a} n^c) F_{b]c}\, .
$$
The last term above is lower order in derivatives of $F$ while  two or fewer derivatives of the first summand 
can readily be expressed in terms of the vanishing set.
 \end{proof}

 It remains to establish the vanishing property of $ {\sf E}^{(4)}$ on solutions.
 \begin{lemma}\label{lastcoeff}
 Let  $(M,\cc,\sigma)$ be a dimension $d\geq 8$ Poincar\'e--Einstein structure. 
 If~$\nabla^A$ is a conformally compact Yang--Mills connection on ${\mathcal V}M$, then
  $$
  \begin{array}{cl}
 {\sf E}^{(4)}_a =0\, ,&
d=8,10,11,12,\ldots\, ,\\[2mm]
{\sf E}^{(4_9)}_a =0\, , &
d=9
\, .
\end{array}
 $$
 \end{lemma}
 
 \begin{proof}
 We begin with the cases where $d\neq 9$. From Lemmas~\ref{warmerShowarma}
 and~\ref{CFF} we have
 $$
 0={\sf E}^{(1)}= {\sf E}^{(2)}
 = {\sf E}^{(3)}\, ,
 $$
 so that
  \begin{eqnarray*}
\check{\sf E}^{(4)}_d
 &=&
(d-7)\Big[
\hat n^a \hat n^b \hat n^c (\nabla_{\hat n} \nabla_a\nabla_b F_{cd})|_\Sigma
-\tfrac1{d-5}( 7P_{\hat n\hat n}+\tfrac{45}2 H^2)\bj_d
\\&&
\qquad\qquad
+\big((\bar \nabla^a P_{\hat n\hat n})
+9H (\bar \nabla^a H)
\big)\bar F_{ad}
\\&&
\qquad\qquad
-\tfrac{6}{d-9}\Big(\tfrac1{d-5}\boxast_9 
\bj_d
+\bar P^{ab} \bar \nabla_a \bar F_{bd}
+\tfrac12(\bar \nabla^a \bar J) \bar F_{ad}
+\tfrac12 \bar C_{abd} \bar F^{ab}\Big)\\
&&
\qquad\qquad
+\tfrac{12}{(d-5)(d-9)}
[ \bar F_{ad}, \bj^a]
\Big]
+\tfrac{12}{d-9} \, \hat k_d 
\, .
 \end{eqnarray*}
  Now we develop the first term of the above display.
 \begin{eqnarray*}
 \hat n^a \hat n^b \hat n^c \nabla_{\hat n}&\!\!\!\! \! \nabla_a\nabla_b\!\! \!\!\!&F_{cd}
 \stackrel\Sigma = 
 \hat n^a \hat n^b  \nabla_{\hat n} \nabla_a\nabla_b F_{nd} 
 -
  \hat n^a \hat n^b [ \nabla_{ n} \nabla_a\nabla_b,n^c] F_{cd} \\[1mm]
&  \stackrel\Sigma=&
  \tfrac1{d-4} \hat n^a \hat n^b  \nabla_{\hat n} \nabla_a\nabla_b (\sigma \nabla^c F_{cd}) 
 -
  \hat n^a \hat n^b [ \nabla_{ n} \nabla_a\nabla_b,n^c] F_{cd}
  \\
&   \stackrel\Sigma=&
    \tfrac1{d-4}\big(  \hat n^b  \nabla_{\hat n}\nabla_b \nabla^c F_{cd}
 +
  \hat n^a \hat n^b  \nabla_{\hat n} (n_a\nabla_b  \nabla^c F_{cd}) 
 +
  \hat n^a \hat n^b  \nabla_{\hat n} \nabla_a(n_b \nabla^c F_{cd}) 
  \big)
  \\&&
 -
  \hat n^a \hat n^b [ \nabla_{ n} \nabla_a\nabla_b,n^c] F_{cd}
  \\
&  \stackrel\Sigma=&
   \tfrac1{d-4}\big( 3 \hat n^b  \nabla_{\hat n}\nabla_b \nabla^c F_{cd}
 +
 H   \nabla_{\hat n}  \nabla^c F_{cd}
 +
  \hat n^a \hat n^b  [\nabla_{n} \nabla_a,n_b ]\nabla^c F_{cd} 
  \big)
  \\&&
 -
  \hat n^a \hat n^b [ \nabla_{n} \nabla_a\nabla_b,n^c] F_{cd}
    \\
 & \stackrel\Sigma=&
   \tfrac1{d-7}\big( 3\hat n^a \hat n^b  
   \nabla_c^\top\nabla_a\nabla_b F^c{}_{d}
 +3 \hat n^a\hat n^b  [\nabla_a\nabla_b, \nabla^c ]F_{cd}
 +
 H   \nabla_{\hat n}  \nabla^c F_{cd}
 \\&&\qquad
 +
  \hat n^a \hat n^b  [\nabla_{n} \nabla_a,n_b ]\nabla^c F_{cd} 
 -(d-4)
  \hat n^a \hat n^b [ \nabla_{n} \nabla_a\nabla_b,n^c] F_{cd}\big)\, .
 \end{eqnarray*}
 Working term by term, the first two tensor structures above are handled as follows (again we rely heavily on  the identities developed in  Section~\ref{PEM}). In the following a $\top$ above an equals sign denotes projection in hypersurface cotangent bundle directions.
 \begin{eqnarray*}
 \hat n^a \hat n^b  \!\!
 &\!\!\!\!\!\!\!\!
   \nabla_c^\top\nabla_a\nabla_b
 \!\! \!\! \!\!\!\!&\!\! F^c{}_{d}
  \stackrel\Sigma= 
     \nabla_c^\top(\hat n^b\nabla_{\hat n}\nabla_b F^c{}_{d})
-H \hat n^b \nabla^\top_c \nabla_b F^c{}_{d}
  -H\bar g^{bc}
   \nabla_{\hat n}\nabla_b F_{cd}\\[1mm]
   &  \stackrel{\Sigma,\top}=\!\!\!\!\! \!\!\!&  
 \tfrac2{d-5} 
\hh\bar \nabla^c\bar \nabla_{[c}\bj_{d]}
-2\bar \nabla^c(\bar P_{[c}{}^b \bar F_{d]b})
+\big(2(\bar \nabla^c P_{\hat n\hat n})+10H(\bar \nabla^c H)\big)\bar F_{cd}
\\&&
+(2P_{\hat n\hat n}+5H^2)\bj_d
+(d-2)H\Big(
P_{\hat n}{}^c\bar F_{dc}
-\tfrac{5}{d-5}\hh H
\bj_{d}
\Big)
\\&&
-2H\nabla^\top_c\big( \tfrac{1}{d-5}(\hat n^c \bj_d-\hat n_d \bj^c)
    -2H\bar F^c{}_d\big)
+H^2 \bj_d 
-H \bar g^{bc} [\nabla_n,\nabla_b] F_{cd}
\\
   &  \stackrel{\Sigma,\top}=
   \!\!\!\!\! \!\!\!
    &  
 \tfrac2{d-5} 
\hh\bar \nabla^c\bar \nabla_{[c}\bj_{d]}
-2\bar \nabla^c(\bar P_{[c}{}^b \bar F_{d]b})
+\big(2(\bar \nabla^c P_{\hat n\hat n})-(d+12)HP_{\hat n}{}^c\big)\bar F_{cd}
\\&&
+(2P_{\hat n\hat n}+\tfrac{4d-41}{d-5}H^2)\bj_d
-H R_{\hat n b}{}^{be}\bar F_{ed}
-H R_{\hat n bd}{}^{e}\bar F^b{}_{e}
\\
  &  \stackrel{\Sigma,\top}= 
  \!\!\!\!\! \!\!\!
  &  
   \tfrac2{d-5} 
\hh\bar \nabla^c\bar \nabla_{[c}\bj_{d]}
-2\bar \nabla^c(\bar P_{[c}{}^b \bar F_{d]b})
+\big(2(\bar \nabla^c P_{\hat n\hat n})-15HP_{\hat n}{}^c\big)\bar F_{cd}
\\&&
+\big(2P_{\hat n\hat n}
+\tfrac{4d-41}{d-5}H^2\big)
 \bj_d \, ,
   \\[2mm]
\\[2mm]
 \hat n^a\hat n^b  &
 \!\!\!\!
 [\nabla_a\nabla_b, \nabla^c ]
 \!\!\!\!\!\!
 &F_{cd}\stackrel\Sigma=
 R_{\hat n}{}^c{}_{\hat n}{}^e \nabla_e^\top \bar F_{cd}
 + R_{\hat n}{}^c{}_{c}{}^e \nabla_{\hat n} F_{ed}
  +R_{\hat n}{}^c{}_{d}{}^e \nabla_{\hat n} F_{ce}\\
  &&+  \hat n^b \nabla_{\hat n}\big(
  [F_b{}^c,F_{cd}]
  +R_b{}^c{}_{c}{}^e F_{ed}
    +R_b{}^c{}_{d}{}^e F_{ce}
  \big)
  \\&\stackrel{\Sigma,\top}=&
  P^{ce} \nabla^\top_e \bar F_{cd}
  -P_{\hat n}{}^e \hat n^c \nabla^\top_e \bar F_{cd}
 +P_{\hat n\hat n}\nabla^\top_e \bar F^e{}_{d}
 \\&&
-2(d-3) P_{\hat n}{}^e \nabla_{\hat n} F_{ed}
-2J\hat n^e \nabla_{\hat n} F_{ed}
-2P^c{}_d\hat n^e \nabla_{\hat n} F_{ce}
\\&&
+\hat n^b [(\nabla_{\hat n} F_b{}^c),\bar F_{cd}]
-\hat n^b (\nabla_{\hat n}Ric_b{}^e)\bar F_{ed}
+\hat n^b (\nabla_{\hat n} R_b{}^c{}_d{}^e)\bar F_{ce}
\\&\stackrel{\Sigma,\top}=&
P^{ce}\bar \nabla_e  \bar F_{cd}-HP_{\hat n}{}^e \bar F_{ed}
+HP_{\hat n}{}^e\bar F_{ed} 
+P_{\hat n\hat n} \bj_d
\\&&
-2(d-3) 
\big(\tfrac1{d-5}P_{\hat n\hat n}\bj_d - 2HP_{\hat n}{}^e \bar F_{ed}\big)
-\tfrac{2}{d-5}J \bj_d
+\tfrac{2}{d-5}P_{cd}^\top\hh \bj^c
\\&&
+\tfrac1{d-5} [\bj^c,\bar F_{cd}]
\!-\!(d-2)\hat n^b (\nabla_{\hat n} P_b{}^e)\bar F_{ed}
+\hat n^b (\nabla_{\hat n} W_b{}^c {}_d{}^e) \bar F_{ce}
\!-\!\hat n^b (\nabla_{\hat n} P_b{}^e)\bar F_{de}
\\&\stackrel{\Sigma,\top}=&
\bar P^{ce} \bar\nabla_e \bar F_{cd}
+\tfrac{d+1}{2(d-5)}H^2 \bj_d
-\tfrac{d+1}{d-5}P_{\hat n\hat n} \bj_d
 +4(d-3) H P_{\hat n}{}^e\bar F_{ed}
\\&&
-\tfrac{2}{d-5} \bar J \bj_d
+\tfrac{2}{d-5} \bar P^c{}_d \bj_c
+\tfrac1{d-5} [\bj^c,\bar F_{cd}]
-(d-3)\hat n^b (\nabla_{\hat n}P_b{}^e)\bar F_{ed}\\&&
+\hat n^b (\nabla_{\hat n} W_b{}^c{}_d{}^e)\bar F_{ce}
\\&\stackrel{\Sigma,\top}=& 
\bar P^{ce} \bar\nabla_e \bar F_{cd}
+\tfrac{d+1}{2(d-5)}H^2 \bj_d
-\tfrac{d+1}{d-5}P_{\hat n\hat n} \bj_d
+6(d-3) H P_{\hat n}{}^e\bar F_{ed}
\\&&
-\tfrac{2}{d-5} \bar J \bj_d
+\tfrac{2}{d-5} \bar P^c{}_d \bj_c
+\tfrac1{d-5} [\bj^c,\bar F_{cd}]
-(d-3) C_{\hat n}{}^b{}_{\hat n} \bar F_{bd}
\\&&
-(d-3)  (\bar\nabla_e P_{\hat n\hat n})
 \bar F^e{}_{d}
+ \bar C_{d}{}^{ec}\bar F_{ce}\, .
 \end{eqnarray*}
The third  tensor structure is precisely that appearing in Equation~\nn{d6} multiplied by~$H$.
$$
   H 
    \nabla_{\hat n}  \nabla^c 
       F_{cd}
\stackrel\Sigma=
 -(d-4)HP_{\hat n}{}^c \bar F_{cd}
-\tfrac
{2(d-4)}{d-5}\, H^2 \bj_d
\, .
$$
The remaining two structures  are then computed as below.
\begin{eqnarray*}
      \hat n^a \hat n^b  [\nabla_{n} \nabla_a,n_b ]\nabla^c F_{cd} \!\!
   &\stackrel\Sigma=&  \! \!
   2H \nabla_{\hat n} \nabla^c F_{cd}-\hat n^a \hat n^b 
   (P_{ab}+g_{ab}P_{\hat n\hat n}) \nabla^c F_{cd}
  \\
  &\stackrel\Sigma=&\!\!
    -2(d-4)HP_{\hat n}{}^c \bar F_{cd}
-\tfrac
{4(d-4)}{d-5}\, H^2 \bj_d
-\tfrac{2(d-4)}{d-5}P_{\hat n\hat n}\hh \bj_d\, ,
\end{eqnarray*}
\begin{eqnarray*}
 \hat n^a \hat n^b [ \nabla_{n} \nabla_a\nabla_b,n^c] F_{cd}\!\!
   &\stackrel\Sigma=&  \!\!
H    \hat n^b\hat n^c\nabla_{\hat n}\nabla_b F_{cd}
+\hat n^a \hat n^b \nabla_{\hat n}\big((-\sigma P_a{}^c -\delta_a^c \rho)\nabla_b F_{cd}\big)
\\&&
+\hat n^a \hat n^b \nabla_{\hat n}\nabla_a \big((-\sigma P_b{}^c -\delta_b^c \rho)F_{cd}\big)
\\
 &\stackrel\Sigma=&\!\!
3H    \hat n^b\hat n^c\nabla_{\hat n}\nabla_b F_{cd}
-2 P_{\hat n}{}^c \nabla_{\hat n} F_{cd}
-3 P_{\hat n\hat n} \hat n^a \nabla_{\hat n} F_{ad}
\\&&
- \hat n^b (\nabla_{\hat n} P_b{}^c) F_{cd}
-\hat n^a \hat n^b \nabla_{\hat n}(n_a P_b{}^c F_{cd})
\\
 &\stackrel\Sigma=&\!\!
3H    \hat n^b\hat n^c\nabla_{\hat n}\nabla_b F_{cd}
-3 P_{\hat n}{}^c \nabla_{\hat n} F_{cd}
-3 P_{\hat n\hat n} \hat n^a \nabla_{\hat n} F_{ad}
\\&&
- 2\hat n^b (\nabla_{\hat n} P_b{}^c) \bar F_{cd}
-H P_{\hat n}{}^c \bar F_{cd}
\\
&\stackrel{\Sigma,\top}=&\!\!
3H P_{\hat n}{}^a \bar F_{da}
-\tfrac{15}{d-5}\hh H^2 \bj_d
-\tfrac{6}{d-5} P_{\hat n\hat n} \bj_d
+5 H P_{\hat n}{}^e \bar F_{ed}
\\&&
-2C_{\hat n}{}^c{}_{\hat n} \bar F_{cd}
-2 \hat n^b \hat n^a (\nabla^c P_{ba})\bar F_{cd}
\\
&\stackrel{\Sigma,\top}=&\!
\!6H P_{\hat n}{}^a \bar F_{ad}
\!-\!\tfrac{15}{d-5}\hh H^2 \bj_d
\!-\!\tfrac{6}{d-5} P_{\hat n\hat n} \bj_d
\!-\!2(\bar \nabla^a \!P_{\hat n\hat n})\bar F_{ad}
\!-\!2C_{\hat n}{}^c{}_{\hat n} \bar F_{cd} \, .
 \end{eqnarray*}
 Hence
\begin{eqnarray*}(d-7)&&\hspace{-1cm}
 \hat n^a \hat n^b \hat n^c \nabla_{\hat n} \nabla_a\nabla_b F_{cd}
 \stackrel\Sigma = 
    \tfrac6{d-5}\bar \nabla^a \bar \nabla_{[a}\bj_{d]}
   -6\bar \nabla^a(\bar P_{[a}{}^c \bar F_{d]c}) 
   -(d-7)(\bar \nabla^aP_{\hat n\hat n})
    \bar F_{ad}
    \\&&
   +9(d-7)HP_{\hat n}{}^a \bar F_{ad}
     +\tfrac{7(d-7)}{d-5}P_{\hat n\hat n} \bj_d
  +\tfrac{45(d-7)}{2(d-5)}
 H^2\bj_d
+\tfrac{3}{d-5} \bar P^{ce} \bar\nabla_e \bar F_{cd}
\!-\!\tfrac6{d-5}\bar J \bj_d
\\&&
+\tfrac{6}{d-5}\bar P^c{}_d \bj_c
+\tfrac{3}{d-5} [\bj^c,\bar F_{cd}]
-(d-1) C_{\hat n}{}^b{}_{\hat n} \bar F_{bd}
+ 3\bar C_{d}{}^{ec}\bar F_{ce}\\[1mm]
& \stackrel\Sigma = &
\tfrac{12}{d-5} \hat k_d 
-\tfrac{6(d-9)}{d-5} \bar \nabla^a(\bar P_{[a}{}^c \bar F_{d]c})
+\tfrac{6}{d-5}(\bar \nabla^a \bar J) \bar F_{ad} 
  -(d-7)(\bar \nabla^aP_{\hat n\hat n})
    \bar F_{ad}
    \\&&
   +9(d-7)HP_{\hat n}{}^a \bar F_{ad}
     +\tfrac{7(d-7)}{d-5}P_{\hat n\hat n} \bj_d
  +\tfrac{45(d-7)}{2(d-5)}
 H^2\bj_d
+3 \bar P^{ce} \bar\nabla_e \bar F_{cd}
\\&&
+\tfrac{6}{d-5}\bar P^c{}_d \bj_c
+ 3\bar C_{d}{}^{ec}\bar F_{ce}\, .\\[1mm]
& \stackrel\Sigma = &
\tfrac{12}{d-5} \hat k_d 
+(d-7)\Big[\tfrac{3}{d-5}(\bar \nabla^a \bar J) \bar F_{ad} 
+\tfrac{6}{d-5} \bar P^{ab} \bar\nabla_a \bar F_{bd}
+\tfrac{3}{d-5}\bar P^c{}_d \bj_c
+\tfrac{3}{d-5}\bar C_{abd}\bar F^{ab}
\\&&
\qquad\qquad\qquad\qquad\quad
  +\tfrac{1}{d-5}(7P_{\hat n\hat n} 
  +\tfrac{45}{2}
 H^2)\bj_d
  -\big((\bar \nabla^aP_{\hat n\hat n})
  +9H(\bar \nabla^a H)\big) \bar F_{ad}
   \Big]
\, .\\[1mm]
\end{eqnarray*}
The proof is completed by employing the identity
\begin{equation}\label{boxj}
\boxast_9 \bj_d =4 \hat k_d -4 \bar P^{ab} \bar \nabla_a \bar F_{bd}
-2(\bar \nabla^a \bar J)\bar F_{ad}
-2\bar C_{abd}\bar F^{ab}
+\tfrac{d-9}2 \bar P^c{}_d \bj_c
+2[ \bar F_{ad}, \bj^a]\, ,
\end{equation}
and remembering that ${\sf E}^{(4)}=\frac1{d-7}\check {\sf E}^{(4)}$ in dimensions $8,10,11,12,\ldots$.
The special~$d$~$=$~$9$ case follows the same arguments as above with minor adjustments  to handle  the  removable $d=9$ singularities.
 \end{proof}
 
 The vanishing properties of boundary operators ${\sf E}^{(\ell)}$ on solutions can be employed to compute expansion coefficients for the formal solution in Equation~\nn{EXPAND}. Suppose, that~$d\geq 8$ and~$\nabla^A$ solves the conformally compact Yang--Mills equation (a similar analysis to that below can be applied when $d=5,6,7$, but in these cases fewer   expansion coefficients $A^{(\ell)}$ can be fixed). Then from Lemmas~\ref{warmerShowarma},~\ref{CFF}, and~\ref{lastcoeff}
we have \begin{equation}\label{E0}
0={\sf E}^{(1)}={\sf E}^{(2)}=
{\sf E}^{(3)}={\sf E}^{(4)}\, .
\end{equation}
(When $d=9$, we must replace ${\sf E}^{(4)}$ by ${\sf E}^{(4_9)}$ in the above).
Then, using Equation~\nn{FGE}, we have that 
$$
F_{nb } \ext y^b = 
\big( A^{(1)}_j(x)+
2rA^{(2)}_j(x)+
3r^2 A^{(3)}_j(x)+
4r^3 A^{(4)}_j(x)
\big)\ext x^j +{\mathcal O}(r^4)\, .
$$
Now ${\sf E}^{(1)}_b = F_{nb}|_\Sigma$, which implies
$$
A^{(1)}_j(x)=0\, .
$$
Turning to Equation~\nn{E2} we learn $\hat n^a  (\nabla_{\hat n} F_{ab})|_\Sigma=\frac1{d-5}\hh{\bj}_b$. But  using ${\sf E}^{(1)}=0$ and $\hat n_a \ext y^a =\ext r$, we have
$
\ext y^b\hat n^a  (\nabla_{\hat n} F_{ab})|_\Sigma=\ext y^b(\nabla_{\hat n}F_{nb})|_\Sigma=2A^{(2)}_j \ext x^j$,
whence 
\begin{equation}\label{A2}
A^{(2)}_j(x)=\tfrac{1}{2(d-5)}\hh  \bar \nabla^i \bar F_{ij}\, .
\end{equation}
Observe that the residue of the $d=5$ pole recovers the $d=5$ obstruction current.

A similar line of argument, now based on Equation~\nn{E3} and the fact that $H=0$ for the bulk metric representative ~\nn{block}, gives
$$
A^{(3)}_j(x)=0\, .
$$

Note that the vanishing of $A^{(1)}$ and $A^{(3)}$ also follows from our earlier discussion of the  evenness property of $A^{\rm ev}$. The  above result for $A^{(2)}$ can also be quickly obtained by a direct Fefferman--Graham type order by order computation. 

The fourth order coefficient~$A^{(4)}$ is significantly more involved. For that we turn to 
the fourth boundary operator of  Equation~\nn{E4}.
First, we note that in the 
canonical compactified
metric $g$, we have that $n=\ext r$ obeys $\nabla_n n=0$ and $n^2=1$. In turn Lemma~\ref{69} 
implies that $\rho=0$ and thus $P_{\hat n\hat n}\stackrel\Sigma=0$.
Hence, turning to Equation~\nn{E4} in dimensions $d\geq 10$, or Equation~\nn{E49} in dimension $d=9$, and using Equation~\nn{E0} as well as Lemma~\ref{lastcoeff}, allows us to extract~$A^{(4)}$. Upon applying the identity~\ref{boxj}, we find $$
A^{(4)}_k(x)=
 \tfrac{1}{2(d-5)(d-7)}\Big( \hat k_k + \tfrac14(d-7) \big(
  \bar P_{ji}\bj^{\hh i}
  +2 \bar P^{ij} \bar \nabla_i \bar F_{jk}
  +(\bar \nabla^i \bar J) \bar F_{ik}
  + \bar C_{ijk} \bar F^{ij}\big)\Big)\, .
$$

\subsection{The Dirichlet-to-Neumann Map}\label{FGrescale}

We now study solutions of the
 conformally compact Yang--Mills equation for Poincar\'e--Einstein structures  beyond their leading asym\-p\-totics.
Firstly in  the case $d$ is even, there is no obstruction and we can apply Theorem~\ref{all-orders-magnetic} to find a 
higher order asymptotic solution. 
To analyze the Neumann data in that case, specializing to Poincar\'e--Einstein backgrounds, we
consider adding an order $r^{d-3}$ term to $A^{\rm ev}$. The Amp\`ere law~\nn{amp}  and vanishing of the obstruction for this dimension parity gives
$$
j_i\big[A^{\rm ev}(x,r) + r^{d-3}A^{(d-3)}(x)\big]={\mathcal O}(r^{d-3})\, ,
$$
 for {\it any} $A^{(d-3)}(x)$.
 However, the Gau\ss\ law~\nn{Ga} at order $r^{d-3}$ does give a condition on~$A^{(d-3)}(x)$, namely
\begin{equation}\label{FGdiv}
\bar \nabla^i A^{(d-3)}_i=0\, .
\end{equation}
An evenness argument shows that
the Neumann data $A^{(d-3)}$ defines a conformal invariant:
When the  formal solution $A_{(d-3)}:=A^{\rm ev}(x,r) + r^{d-3}A^{(d-3)}(x)$ is re-expressed in terms of a conformally related boundary metric $\bar \Omega^2 \bar h$, we have
$A_{(d-3)}=A^{\rm ev}\big(x,\tilde r/\Omega(x,r)\big)+\tilde r^{d-3} \Omega (x,r)^{3-d} A^{(d-3)}(x) 
+{\mathcal O}(r^{d-2})$. Because the first term in this expression involves only even powers of $r$ and  $\tilde r=\Omega(x,r)\hh r$
where $\Omega(x,r)$ has a suitably even expansion in $r$, 
it follows that
the coefficient of the $(d-3)^{\rm rd}$ power of the defining function determines a section of $
T^*\Sigma[3-d]\otimes {\operatorname{End}\mathcal V}\Sigma$. 
Notice also  that the divergence in Equation~\nn{FGdiv} is an invariant operation 
when acting on (endomorphism-valued) one-forms of precisely this weight. 

\smallskip

Now let  ${\mathscr A}$ denote the subset of the space of conformally compact Yang--Mills connections that obey $j[A]={\mathcal O}(\sigma^{d-3})$.
Then we get a canonical {\it Dirichlet-to-Neumann 
operator}
$$
{\mathscr E}^{(d-3)}:
{\mathscr A}
\to
 T^*\Sigma\otimes\End{\mathcal V}\Sigma[3-d]\, ,
$$
whose output is the section $A^{(d-3)}$ in the kernel of the conformally invariant coupled divergence operator (meaning $\bar \nabla^a {\mathscr E}_a^{(d-3)}=0$) constructed above. 
 Here  $(M, \cc,\sigma)$ is Poincar\'e--Einstein, but a similar construction is available on general conformally compact manifolds. 
In the situations where $A$ is the unique  global solution to $j[A]=0$, this operator gives a Dirichlet-to-Neumann map from the boundary Dirichlet data $A^{(0)}(x)$ to the non-locally determined Neumann data.

The above coordinate-based construction, while canonical, is somewhat indirect. It would be propitious to construct a boundary operator that directly extracts Neumann data from solutions.  
Constructing this operator was a main aim of Section~\ref{YMNO}, where candidate operators where constructed for even dimensions $d=4,6$. 
\begin{proposition}\label{boundops}
Let $d=4$ or $6$ and $(M^d,\cc,\sigma,A,\bar A)$ be a Poincar\'e--Einstein structure
equipped with a smooth connection $\nabla^A$
with boundary condition $\bar \nabla^{\bar A}$ such that
$$
j[A]={\mathcal O}(\sigma^{d-3})
 \, .
$$
Then the Dirichlet-to-Neumann operator is given by
$$
{\mathscr E}^{(d-3)}=\frac1{(d-3)!}{\sf E}^{(d-3)}
\, .
$$
\end{proposition}

\begin{proof}
When $d=4$, we only need examine the boundary operator ${\sf E}^{(1)}_a$
and verify that it extracts the expansion coefficient $A^{(1)}(x)$. But ${\sf E}^{(1)}_a=\hat n^b F_{ba}$ so the result follows directly from~\nn{FGE}.

For the $d=6$ case we need to examine ${\sf E}^{(3)}_a$ as given by Equation~\nn{E3}. Because the connection $\nabla^A$ is  asymptotically Yang--Mills, we have that ${\sf E}^{(2)}=0={\sf E}^{(1)}$. Hence
$$
{\sf E}_a^{(3)}=
\hat n^a \hat n^b (\nabla_{\hat n} \nabla_a F_{bc})|_\Sigma
+5H
 \hh\bj_c
-(\bar\nabla^a H)\bar F_{ac}
\, .
$$
In the Fefferman--Graham expansion for a Poincar\'e--Einstein metric, the mean curvature as measured by the compactified metric vanishes, so  only the first term above remains. It not difficult to check that this then equals
$3! A^{(3)}(x)$.
\end{proof}

Theorem~\ref{evenvarySren} established conformal invariance of the renormalized Yang--Mills action~$S^{\rm ren}[A,\tau]$ evaluated on a conformally compact Yang--Mills solution in even dimensions $d$. One can also consider variations along  such solutions determined by a path of boundary connections. By construction, the result of this variation must be conformally invariant, divergence-free and probe asymptotics of the interior connection at order $d-3$. Indeed,  the critical boundary operators ${\sf E}^{(d-3)}$ could be defined this way. 
Uniqueness of the expression found for the $d=4,6$ cases ${\sf E}^{(1)}$ and ${\sf E}^{(3)}$, respectively,  establishes that these must be  functional gradients of the renormalized Yang--Mills action.
We leave the problem of determining/characterizing such variations in higher dimensions to a comprehensive study of the renormalized Yang--Mills action functional.

\smallskip

As mentioned above, when given  data $(M^d,\cc,\sigma,\bar A,{ E})$ for $d=4,6,8,\ldots$, and a conformally compact structure that is Poincar\'e--Einstein, then  Theorem~\nn{all-orders-magnetic} ensures a canonical  ${\mathcal O}(\sigma^\infty)$ solution to Problem~\ref{yeswehavethem}.

\medskip

It now remains to consider  $d$ odd.
In that case the above evenness argument no longer applies and instead Proposition~\ref{logsol} implies that when  
the obstruction current is non-vanishing, 
a polyhomogeneous expansion (involving both powers of $r$ and $\log r$) is required for the higher order  solution asymptotics. However, once  the obstruction current~$\bar k$ vanishes, one might hope for a result along the lines of Proposition~\ref{boundops}. Here there is an essential distinction to the case of $d$ even. While vanishing of the obstruction current~$\bar k$ does guarantee the absence of 
log terms in the expansion of $A_i(x)$, this
alone does not suffice to ensure that the coefficient $A_i^{(d-3)}(x)$ capturing the Neumann data defines a conformal invariant. This is because a new choice of $\bar g\in \cc_\Sigma$ changes the  higher even order  coefficients in the connection expansion which now include $A^{(d-3)}_i(x)$. 

Proposition~\ref{logsol} suggests, when $d$ is odd   and  the obstruction current $k_i(x)\neq 0$, that
$$
j_i\big[A^{\rm ev}(x,r) + r^{d-3}\big(\tfrac1{d-3}\hh k(x) \log r+A^{(d-3)}(x)\big)\big]={\mathcal O}(r^{d-3}\log r)+{\mathcal O}(r^{d-3})\, .
$$
Indeed the above is easily verified. Also the Gau\ss\ law at order $r^{d-3}$ still implies that $\bar \nabla^i A^{(d-3)}_i=0$ (recall that the obstruction obeys $\bar \nabla^i k_i=0$).
Thus, given the data 
%
 of any solution $A$ to the conformally compact Yang--Mills system to order $j[A]={\mathcal O}(r^{d-3})$ (with~$d$ odd   and $(M,\cc,\sigma)$ Poincar\'e--Einstein), given $\bar g\in \cc_\Sigma $ we get a canonical {\it Dirichlet-to-Neumann operator}
\begin{equation}
\label{RDN}
{\mathscr E}^{(d-3)}_{\bar g}:
\mathscr A
\to
 T^*\Sigma\otimes\End{\mathcal V}\Sigma\, ,
\end{equation}
whose output, given the choice of boundary metric representative~$\bar g$, is the section $A^{(d-3)}$ (in the kernel of the  coupled divergence operator) constructed above. 
In the above ${\mathscr A}$ again denotes the subset of the space of conformally compact Yang--Mills connections that obey $j[A]={\mathcal O}(\sigma^{d-3})$.
The situation simplifies  when the boundary connection $\bar A$ is flat.

\begin{proposition}\label{boundopsodd}
Let $d\geq 5$ be odd and $(M^d,\cc,\sigma,A,\bar A)$ be a Poincar\'e--Einstein structure
equipped with a smooth connection $\nabla^A$
and boundary condition $\bar \nabla^{\bar A}$ such that
$$
j[A]={\mathcal O}(\sigma^{d-3})\, .
$$
Then, when the curvature of the boundary connection obeys
$$
\bar F=0\, ,
$$
 it follows that ${\mathscr E}^{(d-3)}_{\bar g}$ defines a section of $\operatorname{End}{\mathcal V}\otimes T^*\Sigma[3-d]$ given for $d=5,7$ by
$$
{\mathscr E}^{(d-3)}=\tfrac1{(d-3)!}\, {\sf E}^{\left((d-3)_d\right)}
\, .
$$

\end{proposition}

\begin{proof}
To establish that 
Equation~\nn{RDN} defines a conformally invariant map we first note that by Equation~\nn{A2} and flatness of the boundary connection, we have $A^{(2)}=0$. The recursion determined by Equation~\nn{ampexp} then implies that $A^{(4)}=A^{(6)}=\cdots =A^{(d-5)}=0$.
Hence $A^{\rm ev}$ is independent of $r$ so $A^{(d-3)}$ cannot depend on the choice of boundary metric representative $\bar g$.

It remains to study the special cases $d=5,7$.
When $d=5$, we must examine
 ${\sf E}^{(2_5)}$ as given in Equation~\nn{E25}. Since ${\sf E}^{(1)}=0$ (because $\nabla^A$ is asymptotically Yang--Mills) we obtain
 $$
 {\sf E}^{(2_5)}_{\, c}(g,A)=
\hat n^a \hat n^b (\nabla_a F_{bc})|_\Sigma\, .$$
It is easy to check that the above equals $2A^{(2)}$ when evaluated in the 
canonical connection expansion.

For $d=7$, we consider Equation~\nn{E47}
and use Equation~\nn{EEEF0}, which gives
 \begin{align*} 
 {\sf E}^{(4_7)}&=
\hat n^a \hat n^b \hat n^c (\nabla_{\hat n} \nabla_a\nabla_b F_{cd})|_\Sigma
\, .
\end{align*}
Again it is not difficult to check that the above equals $4!A^{(4)}$ when evaluated in the canonical connection expansion.
\end{proof}

Once again, but now for  $d$  odd,
in situations where
 $A$ is the unique  global solution to~$j[A]=0$,~the operator of~\nn{RDN}, 
 and its conformally invariant refinement for Yang--Mills flat boundaries of the above proposition, gives a Dirichlet-to-Neumann map from the boundary Dirichlet data~$A^{(0)}(x)$ to the non-locally determined Neumann data.

\medskip

When the obstruction current $\bar k=0$, then
given  data $(M^d,\cc,\sigma,\bar A,\bar g,
{E}^{}_{\bar g}
)$ for $d=5,6,9,\ldots$, where the conformally compact structure $(M,\cc,\sigma)$  is Poincar\'e--Einstein, $\bar \nabla^{\bar A}$~is a boundary connection, $\bar g\in \cc_\Sigma$ and $E_{\bar g}$ is a divergence-free section of  $\operatorname{End}{\mathcal V}\Sigma\otimes T^*\Sigma$,
 Theorem~\ref{all-orders-magnetic} ensures a canonical  ${\mathcal O}(\sigma^\infty)$ solution to Problem~\ref{yeswehavethem}. 
In the special case where the curvature $\bar F$ of 
$\bar A$ is zero, we may refine the data 
$
{ E}^{}_{\bar g}$ to the conformal invariant~${\ E}\in \Gamma(\operatorname{End}{\mathcal V}
\Sigma\otimes T^*\Sigma[3-d])$, now subject to a conformally invariant divergence condition $\bar \nabla^a  { E}^{}_a=0$.
This statement is easily established by observing that flatness of $\bar \nabla^{\bar A}$ in fact implies $A^{\rm ev}(x,r)=0$. In that case one is studying the natural global analog of the formal electric problem of Section~\ref{EP}.
Of course, vanishing of $\bar F$ along $\Sigma$ need not imply that~$\bar \nabla^{\bar A}$ is trivial.

\section*{Acknowledgements}
A.R.G. and A.W. acknowledge support from the Royal Society of New Zealand via Marsden Grant 19-UOA-008.  A.W. was also supported by Simons Foundation Collaboration Grant for Mathematicians ID 686131. This article is based upon work from supported by the European Cooperation in
Science and Technology  Action 21109 CaLISTA, {\tt www.cost.eu}, HORIZON-MSCA-2022-SE-01-01 CaLIGOLA, MSCA-DN CaLiForNIA - 101119552. Y. Z. acknowledges support from
National Key Research and Development Project SQ2020YFA070080 and NSFC 12071450.

\end{document}